\documentclass{amsart}

\usepackage{diagrams}
\usepackage[all]{xy}
\usepackage{amssymb}
\usepackage{amsthm}
\usepackage[headings]{fullpage}
\usepackage[colorlinks]{hyperref}

\newcommand{\QQ}{\mathbb Q}
\renewcommand{\AA}{\mathbb A}
\newcommand{\ZZ}{\mathbb Z}
\newcommand{\NN}{\mathbb N}
\newcommand{\FF}{\mathbb F}
\newcommand{\BB}{\mathbb B}
\newcommand{\LL}{\mathcal{L}}
\newcommand{\CC}{\mathbb C}
\newcommand{\EE}{\mathbb E}
\newcommand{\Qp}{\QQ_p}
\newcommand{\Zp}{\ZZ_p}
\newcommand{\calH}{\mathcal{H}}

\newcommand{\calN}{\mathcal{N}}

\newcommand{\calO}{\mathcal{O}}
\newcommand{\J}{\mathfrak{J}}

\newcommand{\DD}{\mathbb D}

\newcommand{\vp}{\varphi}
\newcommand{\Dcris}{\DD_\mathrm{cris}}
\newcommand{\lp}{\log^+(1+\pi)}
\newcommand{\lm}{\log^-(1+\pi)}
\newcommand{\Brig}{\BB_{\mathrm{rig},\Qp}^+}
\newcommand{\Gn}{G_{\gamma}^{(n)}}
\newcommand{\Gnn}{G_{\gamma}^{(n+1)}}
\DeclareMathOperator{\rmCol}{Col}
\newcommand{\bfCol}{\mathbf{Col}}

\newcommand{\Cbar}{\bar{\mathcal{C}}}
\newcommand{\EA}{\calO_E\otimes\AA_{\Qp}^+}
\newcommand{\EB}{E\otimes\BB_{\mathrm{rig},\Qp}^+}

\newcommand{\Hh}{\mathbb{H}}
\newcommand{\Qpn}{\mathbb{Q}_{p,n}}
\newcommand{\kato}{\mathbf{z}^\mathrm{Kato}}
\DeclareMathOperator{\Char}{Char}
\DeclareMathOperator{\Sel}{Sel}

\newcommand{\HG}{\mathcal{H}(G_\infty)}

\DeclareMathOperator{\Iw}{Iw}

\DeclareMathOperator{\Gal}{Gal}

\DeclareMathOperator{\image}{Im}

\DeclareMathOperator{\Tr}{Tr}

\DeclareMathOperator{\cris}{cris}
\DeclareMathOperator{\dR}{dR}
\DeclareMathOperator{\Rep}{Rep}
\DeclareMathOperator{\Fil}{Fil}

\DeclareMathOperator{\temp}{temp}
\DeclareMathOperator{\rig}{rig}
\newcommand{\OO}{\mathcal{O}}
\newcommand{\la}{\mathrm{la}}

\DeclareMathOperator{\Step}{Step}
\DeclareMathOperator{\PStep}{PStep}
\DeclareMathOperator{\ord}{ord}
\DeclareMathOperator{\pr}{pr}
\DeclareMathOperator{\uCol}{\underline{\mathrm{Col}}}
\DeclareMathOperator{\upr}{\underline{\mathrm{pr}}}

\newtheorem{theorem}{Theorem}[section]
\newtheorem{proposition}[theorem]{Proposition}
\newtheorem{lemma}[theorem]{Lemma}
\newtheorem{corollary}[theorem]{Corollary}
\newtheorem{remark}[theorem]{Remark}
\newtheorem{definition}[theorem]{Definition}
\newtheorem{note}[theorem]{Note}
\newtheorem{notation}[theorem]{Notation}
\newtheorem{convention}[theorem]{Convention}

\newtheorem{conjecture}[theorem]{Conjecture}

\begin{document}

\title{Wach Modules and Iwasawa Theory for Modular Forms}

\author{Antonio Lei}
\address[Lei]{Department of Pure Mathematics and Mathematical Statistics \\
  University of Cambridge \\
  Cambridge CB3 0WB, UK
}
\curraddr{School of Mathematical Sciences\\
Monash University\\
Clayton, VIC 3800\\
Australia}
\email{antonio.lei@monash.edu}
\thanks{The first author is supported by Trinity College Cambridge and an ARC DP1092496 grant.}

\author{David Loeffler}
\address[Loeffler]{Warwick Mathematics Institute \\
  Zeeman Building \\
  University of Warwick \\
  Coventry CV4 7AL, UK
}
\email{d.loeffler.01@cantab.net}
\thanks{The second author is supported by EPSRC Postdoctoral Fellowship EP/F04304X/1}

\author{Sarah Livia Zerbes}
\address[Zerbes]{Department of Mathematics\\
  Harrison Building\\
  University of Exeter\\
  Exeter EX4 4QF, UK
}
\email{s.zerbes@exeter.ac.uk}

\thanks{The third author is supported by EPSRC Postdoctoral Fellowship EP/F043007/1.}

\date{16th October 2010}

\begin{abstract}
We define a family of Coleman maps for positive crystalline $p$-adic
representations of the absolute Galois group of $\Qp$ using the theory of Wach
modules. Let $f=\sum a_nq^n$ be a normalized new eigenform and $p$ an odd prime
at which $f$ is either good ordinary or supersingular. By applying our theory to the
$p$-adic representation associated to $f$, we
define Coleman maps $\uCol_i$ for $i=1,2$ with values in $\overline{\QQ}_p \otimes_{\ZZ_p} \Lambda$, where $\Lambda$ is 
the Iwasawa algebra of $\ZZ_p^\times$. Applying these maps to the Kato zeta elements gives a decomposition of the (generally unbounded) $p$-adic $L$-functions of
$f$ into linear combinations of two power series of bounded coefficients,
generalizing works of Pollack (in the case $a_p=0$) and Sprung (when $f$
corresponds to a supersingular elliptic curve). Using ideas of Kobayashi
for elliptic curves which are supersingular at $p$, we associate to each of
these power series a $\Lambda$-cotorsion Selmer group. This allows us to
formulate a ``main conjecture". Under some technical conditions, we prove one
inclusion of the ``main conjecture" and show that the reverse inclusion is equivalent to Kato's
main conjecture.
\end{abstract}


\maketitle

\tableofcontents
 
 \section{Introduction}
 
  \subsection{Background}
  
 Let $E$ be an elliptic curve defined over $\QQ$ which has good ordinary reduction at the prime $p$. In~\cite{MSD74}, Mazur and Swinnerton-Dyer constructed a $p$-adic $L$-function, $\tilde{L}_{p,E}$, which interpolates complex $L$-values of $E$. Let $\QQ_\infty=\QQ(\mu_{p^\infty})$. If $G_\infty$ denotes the Galois group of $\QQ_\infty$ over $\QQ$, then $\tilde{L}_{p,E}$ is an element of $\Lambda_{\QQ_p}(G_\infty)=\QQ\otimes\ZZ_p[[G_\infty]]$. It is conjectured that $\tilde{L}_{p,E}$ is in fact an element of the Iwasawa algebra $\Lambda(G_\infty)=\ZZ_p[[G_\infty]]$. 
 
 Recall that the $p$-Selmer group of $E$ over any finite extension $F$ of $\QQ$ is defined as
 \[ \Sel_p(E/ F)=\ker\left(H^1(F,E_{p^\infty}) \rTo \prod_v \frac{H^1(F_v,E_{p^\infty})}{E(F_v)\otimes \QQ_p / \ZZ_p}\right),\]
 where the product is taken over all places of $F$. If we let $\Sel_p(E / \QQ_\infty)=\varinjlim_n \Sel_p(E/ \QQ(\mu_{p^n}))$, then $\Sel_p(E/\QQ_\infty)$ is equipped with an action of $G_\infty$ which extends to an action of the Iwasawa algebra. It is not difficult to show that the Pontryagin dual $\Sel_p(E/\QQ_\infty)^\vee$ is finitely generated over $\Lambda(G_\infty)$, and a theorem of Kato-Rohrlich (conjectured by Mazur) states that it is in fact $\Lambda(G_\infty)$-torsion. We can therefore associate to it a characteristic ideal for each $\Delta$-isotypical component, where $\Delta$ is the torsion subgroup of $G_\infty$, and the main conjecture of cyclotomic Iwasawa theory for $E$ predicts that this ideal is generated by the corresponding isotypical component of $\tilde{L}_{p,E}$.

 The construction of $p$-adic $L$-functions has been generalized to more general primes and modular forms in \cite{amicevelu75,vishik}. If $f=\sum a_nq^n$ is a normalized new eigenform of weight $k \ge 2$, level $N$ and character $\epsilon$, $p\nmid N$, then there exists a $p$-adic $L$-function $\tilde{L}_{p,\alpha}$, for any root $\alpha$ of $X^2-a_pX+\epsilon(p)p^{k-1}$ such that $v_p(\alpha)<k-1$, interpolating complex $L$-values of $f$. Perrin-Riou \cite{perrinriou95} and Kato \cite{kato93} have established theories of $p$-adic $L$-functions for a wide class of $p$-adic Galois representations and formulated respective Iwasawa main conjectures. When the representation corresponds to a modular form, these main conjectures have been reformulated by Kato \cite{kato04} using the theory of Euler systems. If $f$ is good ordinary at $p$ (in other words, $p\nmid N$ and $a_p$ is a $p$-adic unit) and $\alpha$ is the unique unit root, then $\tilde{L}_{p,\alpha}$ is an element of $\Lambda_{\QQ_p}(G_\infty)$. At a $\Delta$-isotypical component, the main conjecture is equivalent to asserting that $\tilde{L}_{p,\alpha}$ generates the characteristic ideal of $\Sel_p(f/\QQ_\infty)^\vee$. In \textit{op.cit.}, Kato has shown that $\tilde{L}_{p,\alpha}$ is contained in the characteristic ideal of $\Sel_p(f/\QQ_\infty)^\vee$ under some technical assumptions; his proof relies on the construction of certain zeta elements (which we will refer to as \emph{Kato zeta elements}).

 When $f$ is supersingular at $p$ (by which we mean $p\nmid N$ and $a_p$ is not a $p$-adic unit), two problems arise: on the one hand, the $p$-adic $L$-functions of Amice--V\'elu and Vishik are no longer elements of $\Lambda(G_\infty)$, but they lie in the algebra $\HG$ of distributions on $G_\infty$, and on the other hand, $\Sel_p(f/\QQ_\infty)^\vee$ is no longer $\Lambda(G_\infty)$-torsion. Perrin-Riou's (and hence Kato's) main conjecture can therefore not be translated into a statement relating $\tilde{L}_{p,\alpha}$ and $\Sel_p(f/\QQ_\infty)$ as in the ordinary case. When $a_p = 0$, a remedy was made possible by the works of Pollack \cite{pollack03}. If $\alpha_1$ and $\alpha_2$ are the roots of $X^2+\epsilon(p)p^{k-1}$, Pollack showed that there is a decomposition 
\[
\tilde{L}_{p,\alpha_i}=\log_{p,k}^+\tilde{L}_p^+ + \alpha_i\log_{p,k}^-\tilde{L}_p^-
\]
for $i=1,2$, where $\tilde{L}_{p}^\pm\in\Lambda_{\QQ_p}(G_\infty)$ and $\log_{p,k}^\pm$ are some explicit elements of $\HG$ which only depend on $k$. When $f$ corresponds to an elliptic curve $E/\QQ$ (and $p > 2$), the $\tilde{L}_p^\pm$ are in fact elements of $\Lambda(G_\infty)$. In~\cite{kobayashi03}, Kobayashi formulates a main conjecture giving an arithmetic interpretation of these new $p$-adic $L$-functions. In analogy to the ordinary reduction case, he defines even and odd Selmer groups $\Sel_p^\pm(E/\QQ_\infty)$ by modifying the local conditions at $p$ in the definition of the usual Selmer group. Let $T_pE$ be the Tate module of $E$. Kobayashi shows that $\Sel_p^\pm(E/\QQ_\infty)$ is $\Lambda(G_\infty)$-cotorsion by constructing the so-called plus and minus Coleman maps
\[
\rmCol^\pm:H^1_{\Iw}(\Qp,T_pE)\rightarrow\Lambda(G_\infty),
\]
which depend on the structure of the formal group attached to $E$. (Here $H^1_{\Iw}(\Qp,T_pE)$ is the Iwasawa cohomology, defined as $\varprojlim_n H^1(\Qp(\mu_{p^n}), T_p E)$; see \S \ref{sect-fontaineisom} below.) Kobayashi's modified main conjecture then asserts that in each $\Delta$-isotypical component, the functions $\tilde{L}_p^\pm$ generate the respective characteristic ideals of $\Sel_p^\pm(E/\QQ_\infty)^\vee$. This main conjecture is in fact equivalent to \cite[Conjecture~12.10]{kato04} (to which we refer as \emph{Kato's main conjecture} from now on). Using the fact that the maps $\rmCol^\pm$ send the localization of the Kato zeta elements to $\tilde{L}_p^\pm$, Kobayashi shows that the elements $\tilde{L}_p^\pm$ are contained in the characteristic ideals of $\Sel_p^\pm(E/\QQ_\infty)^\vee$ (possibly after inverting $p$ if $p$ is one of the finitely many primes for which the $p$-adic Galois representation of $E$ is not surjective), establishing half of the main conjecture. (When the elliptic curve has complex multiplication, the full conjecture has been proved by Pollack and Rubin \cite{pollackrubin04}.)

 Sprung~\cite{sprung09} has extended the results of Kobayashi to elliptic curves with supersingular reduction at $p$ and $a_p\ne0$ (which forces $p$ to be $2$ or $3$). He constructs Coleman maps
 \[ \rmCol^\vartheta,\rmCol^\upsilon: H^1_{\Iw}(\QQ_p,T_pE) \rTo \Lambda(G_\infty)\]
 and defines $\tilde{L}_{p}^\vartheta, \tilde{L}_{p}^\upsilon \in\Lambda(G_\infty)$ by applying these Coleman maps to the Kato zeta element. Analogously to the case $a_p=0$ discussed above, he defines two Selmer groups $\Sel_p^\vartheta(E/\QQ_\infty)$ and $\Sel_p^\upsilon(E/\QQ_\infty)$ to formulate the corresponding main conjectures. Moreover, he constructs a matrix $M\in M_2(\HG)$ whose entries are functions of logarithmic growth depending only on $a_p$ such that
 \[
 \begin{pmatrix}\tilde{L}_{p,\alpha}\\ \tilde{L}_{p,\beta} \end{pmatrix}=M\begin{pmatrix}\tilde{L}_{p}^\vartheta\\ \tilde{L}_{p}^\upsilon \end{pmatrix}
 \]
 generalizing Pollack's results.

 Generalizing Kobayashi's work in a different direction, the first author has shown in \cite{lei09} that the definition of the maps $\rmCol^\pm$ can be extended to general modular forms with $a_p = 0$, using $p$-adic Hodge theory in place of formal groups. For a normalized new eigenform $f$, there exists a $p$-adic representation $V_f$ of $G_\QQ=\Gal(\overline{\QQ}/\QQ)$ attached to $f$, as constructed by Deligne \cite{deligne69}. When $a_p=0$, one can then construct $\pm$-Coleman maps
 \[
 \rmCol^\pm:H^1_{\Iw}(\Qp,V_f)\rightarrow\Lambda_{\QQ_p}(G_\infty),
 \]
 using the structure of $\DD_{\cris}(V_f)$ and Perrin-Riou's exponential map (see Section 2 in~\cite{lei09}). Generalizing Kobayashi's construction, one can use $\rmCol^\pm$ to define $\pm$-Selmer groups, which again turn out to be $\Lambda(G_\infty)$-cotorsion and whose characteristic ideals at each $\Delta$-isotypical component contain Pollack's $p$-adic $L$-functions. Analogous to the work of Pollack and Rubin for elliptic curves, one can show that equality holds for forms of CM type; see \cite{lei09} for details. 

  \subsection{Statement of the main results}\label{mainresults}

 Looking at all these results raises some natural questions: Is there a uniform explanation for Sprung's logarithmic matrix $M$ and Pollack's $\pm$-logarithms? Can one generalize the construction of the two Coleman series to more general modular forms which are supersingular at $p$?
  
 In this paper, we approach these questions using methods from the theory of $(\vp,G_\infty)$-modules. As shown by Fontaine (unpublished -- for a reference see~\cite{cherbonniercolmez99}), for any $\ZZ_p$-linear representation $T$ of $G_{\QQ_p}$ there is a canonical isomorphism $h^1_{\QQ_p,\Iw}:H^1_{\Iw}(\QQ_p,T)\cong\DD(T)^{\psi=1}$, where $\DD(T)$ denotes the $(\vp,G_\infty)$-module\footnote{More familiarly known as a $(\vp, \Gamma)$-module -- our $G_\infty$ is denoted by $\Gamma$ in Fontaine's work, while we use $\Gamma$ for its torsion-free part.} of $T$ and $\psi$ is a certain left inverse of $\vp$. Recall that $\DD(T)$ is a module over the $p$-adic completion $\AA_{\QQ_p}$ of the power series ring $\ZZ_p[[\pi]][\pi^{-1}]$. Also, $\Lambda(G_\infty)$ can be identified with the additive group $\ZZ_p[[\pi]]^{\psi=0}$ via the Mellin transform (c.f. Section~\ref{appendix}). It seems therefore natural to expect that by carefully choosing a basis of $\DD(T)$, it should be possible to define the two Coleman maps as certain maps on the coefficients of an element $x\in\DD(T)^{\psi=1}$. Such a construction would generalize the classical case  $T=\ZZ_p(1)$: in this case, the Coleman map $H^1_{\Iw}(\QQ_p,\ZZ_p(1))\cong\AA_{\QQ_p}^{\psi=1}\rightarrow \ZZ_p[[\pi]][\pi^{-1}]^{\psi=0}$ is just the map $\vp-1$.
 
 Here, we develop this idea using Berger's theory of Wach modules~\cite{berger03}, which is a refined version of $(\vp,G_\infty)$-modules for crystalline representations over unramified base fields originally studied by Wach in \cite{wach96}. The Wach module $\NN(V)$ of a crystalline $G_{\Qp}$-representation $V$ is a certain subspace of the $(\vp,G_\infty)$-module $\DD(V)$ which is a finitely-generated module over the simpler ring
$\BB^+_{\QQ_p}=\ZZ_p[[\pi]]\otimes_{\ZZ_p}\QQ_p$. If $V$ is a crystalline representation of $G_{\QQ_p}$ with non-negative Hodge-Tate weights, and $V$ has no quotient isomorphic to $\QQ_p$, then Berger has shown in~\cite{berger03} that $\DD(V)^{\psi=1}=\NN(V)^{\psi=1}$. 
 Let $\vp^*\NN(V)$ be the $\BB^+_{\QQ_p}$-submodule of $\DD(V)$ generated by the image of $\vp$. For any such representation, $1-\vp$ gives a map
 \[ 1-\vp: \NN(V)^{\psi=1} \rTo (\vp^*\NN(V))^{\psi=0}.\]

 Our first main result relates this map to Perrin-Riou's theory. Suppose that $V_f$ is the $p$-adic representation associated to a modular form $f$ with $p$ a good prime for $f$, i.e. $p$ does not divide the level of $f$ (we assume here for notational simplicity that the coefficient field of the modular form is $\QQ$, so $V$ is a $2$-dimensional $\QQ_p$-vector space). Let $V=V_f(k-1)$, then $V$ is a crystalline representation with Hodge-Tate weights $0,k-1$. We fix $\bar{\nu}_1,\bar{\nu}_2$ a basis of $\DD_{\cris}(V_f)$ in Section~\ref{Colemanmf}. It lifts to a basis $n_1,n_2$ of $\NN(V_f)$. Note that $\pi^{1-k}n_1\otimes e_{k-1}, \pi^{1-k}n_2\otimes e_{k-1}$ then gives a basis of $\NN(V)$. Let $M=(m_{ij})\in M_2\big(\vp(\BB^+_{\rig,\QQ_p})\big)$ be such that 
 \[ \begin{pmatrix} \vp(\pi^{1-k}n_1\otimes e_{k-1}) \\ \vp(\pi^{1-k}n_2\otimes e_{k-1}) \end{pmatrix} = M  \begin{pmatrix} \bar{\nu}_1\otimes t^{1-k}e_{k-1} \\ \bar{\nu}_2\otimes t^{1-k}e_{k-1} \end{pmatrix} .\]

 \begin{proposition}[see Proposition~\ref{samemaps}] For $i=1,2$ we have a commutative diagram
 \begin{diagram}
  \NN(V)^{\psi=1} & \rTo^{h^1_{\QQ_p,\Iw}} &  H^1_{\Iw}(\QQ_p,V)\\
               \dTo_{1-\vp}      &                               &  \\
    (\vp^*\NN(V))^{\psi=0} &  &   \\
            \dTo_{M}         &                               &    \dTo_{\mathcal{L}_{1,\bar{\nu}_i\otimes(1+\pi)}} \\
    \big((\BB^+_{\rig,\QQ_p})^{\psi=0}\big)^{\oplus 2}  &           &  \\
            \dTo_{\pr_i}         &                               &    \\
    (\BB^+_{\rig,\QQ_p})^{\psi=0}  &  \rTo^{\mathfrak{M}^{-1}}          & \HG
 \end{diagram}
 \end{proposition}

 Here, $\mathcal{L}_{1,\bar{\nu}_i\otimes(1+\pi)}$ is a certain $\Lambda_{\Qp}$-module homomorphism whose definition is given in equation \eqref{lmap} below, defined using Perrin-Riou's exponential map and the Perrin-Riou pairing $H^1_{\Iw}(V) \times H^1_{\Iw}(V^*(1)) \to \Lambda_{\Qp}$. Also, $\mathfrak{M}$ is the inverse Mellin transform (see \eqref{mellin}), $\pr_i$ is the projection map onto the $i$-th component, and for an element $x\in(\vp^*\NN(V))^{\psi=0}$, $M.x$ is defined as follows: if  $x=x_1\vp(\pi^{1-k}n_1\otimes e_{k-1})+x_2\vp(\pi^{1-k}n_2\otimes e_{k-1})$ for some $x_i\in (\BB^+_{\QQ_p})^{\psi=0}$, then $M.x=M\begin{pmatrix} x_1 \\ x_2\end{pmatrix}$. 
 
 By applying this diagram to Kato's zeta element ${\kato}\in H^1_{\Iw}(\QQ_p,V)$, we deduce that there exist $\mathcal{M}\in M_2\big(\overline{\QQ}_p\otimes_{\QQ_p}\vp(\BB^+_{\rig,\QQ_p})\big)$ and $L_{p,1},L_{p,2}\in (\BB^+_{\QQ_p})^{\psi=0}$ (c.f. Section~\ref{decomposition}), depending only on the basis $n_1,n_2$, such that we have a decomposition
 \begin{equation}\label{wrongdecomposition} 
  \begin{pmatrix} \mathfrak{M}(\tilde{L}_{p,\alpha})\\ \mathfrak{M}(\tilde{L}_{p,\beta}) \end{pmatrix} = \mathcal{M} \begin{pmatrix} L_{p,1}\\ L_{p,2}\end{pmatrix}.
 \end{equation}
 In order to interpret this decomposition in terms of measures, we need to study the structure of $(\vp^*\NN(V))^{\psi=0}$ as a $\Lambda(G_\infty)$-module. The following result was proven independently by Berger (Theorem~\ref{Lambda}, for general Wach modules) and ourselves (Theorem~\ref{lambdabasis}, for the Wach module of the representation arising from a supersingular modular form).
 
 \begin{theorem}
  Let $\calN$ be a Wach module of rank $d$. Then $(\vp^*\calN)^{\psi=0}$ is a free $\Lambda_{\QQ_p}(G_\infty)$-module of rank $d$. Moreover, there exists a basis $n_1,\dots,n_d$ of $\calN$ such that $(1+\pi)\vp(n_1),\dots,(1+\pi)\vp(n_d)$ is a $\Lambda_{\QQ_p}(G_\infty)$-basis of $(\vp^*\NN(V))^{\psi=0}$.
 \end{theorem}
 
 When $V$ is the $p$-adic representation associated to a modular form with $v_p(a_p)>\lfloor\frac{k-2}{p-1}\rfloor$, then there is an explicit choice of the $\BB^+_{\QQ_p}$-basis of $\NN(V)$ which was constructed in~\cite{bergerlizhu04} (c.f. Section~\ref{ss}). We show (Theorem \ref{lambdabasis}) that this basis $(n_1, n_2)$ has the additional property that $(1 + \pi)\vp(n_1), (1 + \pi)\vp(n_2)$ is a $\Lambda_{\QQ_p}(G_\infty)$-basis of $(\vp^*\NN(V))^{\psi=0}$. Hence we may define the \emph{Iwasawa transform}
 \[ \J: (\vp^*\NN(V))^{\psi=0} \rTo \Lambda_{\QQ_p}(G_\infty)^{\oplus 2} \]
 to be the induced isomorphism of $\Lambda_{\QQ_p}(G_\infty)$-modules associated to this basis. This map has the following property: if $a_p=0$, then $\J$ fits into the commutative diagram
 \begin{diagram}
  \NN(V)^{\psi=1} & \rTo^{h^1_{\QQ_p,\Iw}} & H^1_{\Iw}(\QQ_p,V) \\
  \dTo^{1-\vp} &                    & \dTo^{(\rmCol^\pm)} \\
  (\vp^*\NN(V))^{\psi=0} & \rTo^{\J}       & \Lambda_{\QQ_p}(G_\infty)^{\oplus 2}
 \end{diagram}
 where $\rmCol^\pm$ are the Coleman maps constructed in \cite{kobayashi03} and \cite{lei09}. In other words, if $i=1,2$ and we define $\uCol_i: \NN(V)^{\psi=1}\rightarrow \Lambda_{\QQ_p}(G_\infty)$ to be the composition of $\J\circ (1-\vp)$ with the projection of $\Lambda_{\QQ_p}(G_\infty)^{\oplus 2}$ onto the $i$-th component, then we recover the constructions in \textit{op.cit.} In Sections~\ref{leiswork} and~\ref{elliptic}, we use this new description of the Coleman maps to give alternative proofs of their main properties. 
 
 When $v_p(a_p)>\lfloor\frac{k-2}{p-1}\rfloor$, we define the maps $\uCol_i$ in the same manner, and it it follows from Proposition~\ref{commutativediagram} that the following diagram is commutative:
   \[
    \xymatrix{
    \NN(V)^{\psi=1} \ar[r]^{h^1_{\QQ_p,\Iw}} \ar[d]_{1-\varphi}        & H^1_{\Iw}(\QQ_p,V) \ar[d]_{(\uCol_1,\uCol_2)} \ar@/^1.3cm/[ddd]^{{\mathcal{L}_{1,\bar{\nu}_i\otimes(1+\pi)}}}\\
    (\vp^*\NN(V))^{\psi=0} \ar[r]^{\mathfrak{I}} \ar[d]_{M} & \Lambda_{\QQ_p}(G_\infty)^{\oplus 2} \ar[d]_{\underline{M}} \\
    \big((\BB^+_{\rig,\QQ_p})^{\psi=0}\big)^{\oplus 2} \ar[r]^{\mathfrak{M}^{-1}} \ar[d]_{\pr_i} & \calH(G_\infty)^{\oplus 2} \ar[d]_{\upr_i} \\
    (\BB^+_{\rig,\QQ_p})^{\psi=0}  \ar[r]^{\mathfrak{M}^{-1}}         & \calH(G_\infty).
    }
   \]
 Here, the map $\upr_i$ is the projection onto the $i$-th component in $\HG^{\oplus 2}$.
 In particular, this diagram allows us to translate~\eqref{wrongdecomposition} in terms of $\Lambda_{\QQ_p}(G_\infty)$: 
  
 \begin{theorem}[see Theorem~\ref{Ldecomposition}] 
  For $i=1,2$, define $\tilde{L}_{p,i}=\uCol_i(\kato)$. There exists a $2\times 2$-matrix $\underline{\mathcal{M}}\in M_2(\overline{\QQ}_p\otimes_{\QQ_p}\HG)$ depending only on $k$ and $a_p$ such that
  \begin{equation}
  \begin{pmatrix}\tilde{L}_{p,\alpha}\\ \tilde{L}_{p,\beta} \end{pmatrix}=\underline{\mathcal{M}}\begin{pmatrix}\tilde{L}_{p,1}\\ \tilde{L}_{p,2} \end{pmatrix}\label{supersingulardecomposition-intro}
  \end{equation}   
 \end{theorem}
 
 We show in Proposition~\ref{pmtransform} that this decomposition reduces to the decompositions of $\tilde{L}_{p,\alpha},\tilde{L}_{p,\beta}$ given by Pollack when $a_p = 0$.
 
 Assume now that $V_f$ is the $p$-adic representation associated to a modular form $f$ which is good ordinary at $p$, and let $V=V_f(k-1)$. By choosing a suitable basis for $\DD_{\cris}(V_f)$ (c.f Section~\ref{ordinary}) and applying Theorem~\ref{Lambda} to $(\vp^*\NN(V))^{\psi=0}$, we can prodeed analogously to the supersingular case discussed above to construct Coleman maps $\uCol_i:\NN(V)^{\psi=1}\rightarrow \Lambda_{\QQ_p}(G_\infty)$. Let $\alpha$ and $\beta$ be the unit and non-unit eigenvalues of the Frobenius respectively. The Kato zeta element gives rise to two $p$-adic $L$-functions $\tilde{L}_{p,\alpha}$ and $\tilde{L}_{p,\beta}$, where $\tilde{L}_{p,\beta}$ conjecturally agrees with the critical-slope $p$-adic $L$-function constructed by Pollack and Stevens in~\cite{pollackstevens09} when $V_f$ is not locally split at $p$. The analogue of \eqref{supersingulardecomposition-intro} becomes
\begin{equation}\label{ordinarydecomposition}
\begin{pmatrix}\tilde{L}_{p,\alpha}\\ \tilde{L}_{p,\beta} \end{pmatrix}=\begin{pmatrix}0& \bar{u}\\  -\alpha\log_{p,k} & * \end{pmatrix}\begin{pmatrix}\tilde{L}_{p,1}\\ \tilde{L}_{p,2} \end{pmatrix}.
\end{equation}
for some $\bar{u}\in\Lambda_E(G_\infty)^\times$ (c.f.~\eqref{orddecomp}).
Note that a similar decomposition can be obtained from works of Perrin-Riou for elliptic curves with good ordinary reduction at $p$ (see \cite[Section~1.4]{perrinriou93}). The decomposition \eqref{ordinarydecomposition} allows us to show that $\tilde{L}_{p,1}, \tilde{L}_{p,2}\ne0$ under some technical assumptions.

 As in the cases studied in~\cite{kobayashi03} and \cite{lei09}, we can use the maps $\uCol_i$ to construct Selmer groups $\Sel_p^i(f/\QQ_\infty)$ (see Definition \ref{definitionselmer}), and we prove the following results. Define assumptions
 
 (A) (when $f$ is supersingular at $p$) $k\ge3$ or $a_p=0$;
 
 (A') (when $f$ is good ordinary at $p$) $k\geq 3$ and $V_f$ is not locally split at $p$.
 
\begin{theorem}[see Theorem~\ref{torsionSelmer}]
 \label{signedSelmergroups} Under assumption (A) (if $f$ is supersingular at $p$) or assumption (A') (if $f$ is good ordinary at $p$), the group $\Sel_p^i(f/\QQ_\infty)$ is $\Lambda_{\calO_E}(G_\infty)$-cotorsion for $i=1,2$. Moreover, there exist some $n_i\ge0$ such that
 \[
 \varpi^{n_i}\tilde{L}_{p,i}^\eta\in\Char_{\Lambda_{\calO_E}(\Gamma)}(\Sel_p^i(f/\QQ_\infty)^{\vee,\eta})
 \]
 where $\eta$ is any character on $\Delta$ when $i=1$ and it is the trivial character when $i=2$.
\end{theorem}

\begin{corollary}[see Corollary~\ref{Katoequivalence}]
  Let $\eta$ be a character on $\Delta$ as in Theorem~\ref{signedSelmergroups}. If either assumption (A) or assumption (A') is satisfied, and the image of $\Gal(\overline{\QQ} / \QQ_\infty)$ in $\operatorname{GL}(V_f)$ contains a conjugate of $\operatorname{SL}_2(\ZZ_p)$, then Kato's main conjecture is equivalent to
 \[\Char_{\Lambda_{\calO_E}(\Gamma)}(\Sel_p^i(f/\QQ_\infty)^{\vee,\eta})=\Char_{\Lambda_{\calO_E}(\Gamma)}(\image(\uCol_i)^\eta/(\tilde{L}_{p,i}^\eta) ).\]
\end{corollary}

 Note that in the ordinary case, $\tilde{L}_{p, 2}$ agrees with the usual $p$-adic $L$-function of $f$ up to a unit in $\Lambda_E(G_\infty)$. It will be shown in a forthcoming paper of the first and third authors \cite{leizerbes2} that the corresponding Selmer group is the usual $\Sel_p(f/\QQ_\infty)$; whereas the first Coleman map gives a new $p$-adic $L$-function $\tilde{L}_{p,1}$ and a new Selmer group.

  
  \subsection{Notation}
 
   Throughout this paper, let $p$ be an odd prime. Fix embeddings of $\overline{\QQ}$ into $\overline{\QQ}_p$, and into $\mathbb{C}$. For $n\ge0$, write $\Qpn=\Qp(\mu_{p^n})$ (resp. $\QQ_n=\QQ(\mu_{p^n})$) for the extension of $\Qp$ (resp. $\QQ$) obtained by adjoining the $p^n$-th roots of unity. Let $G_n$ denote its Galois group. Let $\QQ_{p,\infty}=\bigcup\Qpn$, and write $G_\infty$ for the Galois group of $\QQ_{p,\infty}$ over $\Qp$. We identify $G_\infty$ with the Galois group of $\QQ_\infty=\bigcup_{n\geq 1}\QQ_n$ over $\QQ$. Then $G_\infty\cong\Delta\times\Gamma$ where $\Delta$ is a finite group of order $p-1$ and $\Gamma\cong\Zp$, the Galois group of $\QQ_{p,\infty}$ over $\Qp(\mu_{p})$. We fix a topological generator $\gamma$ of $\Gamma$ and write $\chi$ for the cyclotomic character of $G_\infty$. Let $G_{\QQ_p}=\Gal(\overline{\QQ}_p/\QQ_p)$ and $H_{\QQ_p}=\Gal(\overline{\QQ}_p/\QQ_{p,\infty})$, where $\overline{\QQ}_p$ denotes an algebraic closure of $\QQ_p$.
  
   Given a finite extension $K$ of $\Qp$ with ring of integers $\calO_K$, $\Lambda_{\calO_K}(G_\infty)$ (respectively $\Lambda_{\calO_K}(\Gamma)$) denotes the Iwasawa algebra of $G_\infty$ (respectively $\Gamma$) over $\calO_K$. We further write $\Lambda_K(G_\infty)=\Lambda_{\calO_K}(G_\infty)\otimes\QQ$ and $\Lambda_K(\Gamma)=\Lambda_{\calO_K}(\Gamma)\otimes\QQ$.

   Given a module $M$ over $\Lambda_{\calO_K}(G_\infty)$ (respectively $\Lambda_K(G_\infty)$) and a character $\eta:\Delta\rightarrow\Zp^\times$, $M^\eta$ denotes the $\eta$-isotypical component of $M$. For any $m\in M$, we write $m^\eta$ for the projection of $m$ into $M^\eta$.


\section{Representations of \texorpdfstring{$G_{\QQ_p}$}{G(Qp)}}

 In this section we review some aspects of the theory of $p$-adic representations of $G_{\QQ_p}$. Most of our account is reproduced from~\cite{berger04} and~\cite[\S 2]{bergerlizhu04}. Let $E$ be a finite extension of $\QQ_p$ with ring of integers $\calO_E$. An $E$-linear representation of $G_{\QQ_p}$ is a finite dimensional $E$-vector space $V$ with a continuous $E$-linear action of $G_{\QQ_p}$. We similarly have the notion of an $\OO_E$-linear representation of $G_{\QQ_p}$, which is a finitely-generated (not necessarily free) $\calO_E$-module with a continuous $\calO_E$-linear action of $G_{\QQ_p}$. Define $\Rep_E(G_{\QQ_p})$ (respectively $\Rep_{\calO_E}(G_{\QQ_p})$) to be the categoy of $E$-linear (respectively $\calO_E$-linear) representations of $G_{\QQ_p}$.

 \subsection{\texorpdfstring{$p$}{p}-adic Hodge theory}
 
  In this section, we recall the definitions of some of Fountain's rings of periods. Let $\CC_p$ be the completion of $\overline{\QQ}_p$ for the $p$-adic topology, endowed with the usual valuation $v_p$ normalized such that $v_p(p) = 1$. Let 
  \[ \tilde{\EE}=\varprojlim_{x\mapsto x^p}\CC_p =\big\{ (x^{(0)},x^{(1)},\dots): (x^{(i+1)})^p=x^{(i)}\big\}, \]
  and let $\tilde{\EE}^+$ be the set of $x\in\tilde{\EE}$ such that $x^{(0)}\in\calO_{\CC_p}$. 
  We can equip $\tilde{\EE}$ naturally with the structure of an algebraically closed field of characteristic $p$: if $x=(x^{(i)})$ and $y=(y^{(i)})$, define $x+y$ and $xy$ by
  \begin{align*}
   (x+y)^{(i)} &:= \lim_{j\rightarrow+\infty}\big(x^{(i+j)}+y^{(i+j)}\big)^{p^j} \\ (xy)^{(i)} &:= x^{(i)}y^{(i)}.
  \end{align*}

  Define a complete valuation on $\tilde{\EE}$ by $v_{\tilde{\EE}}(x)=v_p(x^{(0)})$ if $x=(x^{(i)})\in\tilde{\EE}$. Let $\tilde{\AA}^+=W(\tilde{\EE}^+)$ be the ring of Witt vectors of $\tilde{\EE}^+$, and let $\tilde{\BB}^+=\tilde{\AA}^+[p^{-1}]$. An element $x\in\tilde{\BB}^+$ can then be written uniquely in the form 
  \[ x=\sum_{i\gg -\infty}p^i[x_i],\]
  where $x_i\in\tilde{\EE}^+$ and $[x_i]$ denotes the Teichm\"uller lift. The ring $\tilde{\BB}^+$ is equipped with the Witt vector Frobenius map $\vp$ (lifting the map $x \mapsto x^p$ on $\tilde{\EE}^+$), and with a map
  \[ \theta: \tilde{\BB}^+  \rTo \CC_p\]
  via $\theta\left( \sum_{i\gg -\infty}p^i[x_i]\right)=\sum_{i\gg\infty} p^ix_i^{(0)}$. Fix an element $\varepsilon=(\varepsilon^{(n)})\in\tilde{\EE}^+$ with $\varepsilon^{(0)}=1$ and $\varepsilon^{(1)}\neq 1$. Let $\pi=[\varepsilon]-1$, $\pi_1=[\vp^{-1}(\varepsilon)]-1$ and $\omega=\frac{\pi}{\pi_1}$.
   
   The ring $\BB^+_{\dR}$ is defined as $\BB^+_{\dR}=\varprojlim \tilde{\BB}^+/\ker(\theta)^n$. It is a discrete valuation ring, and its maximal ideal is generated by $t=\log([\varepsilon])$. Define $\BB_{\dR}=\BB^+_{\dR}[t^{-1}]$ to be the fraction field of $\BB_{\dR}^+$, which is equipped with an action of $G_{\QQ_p}$ and a filtration defined by $\Fil^i\BB_{\dR}=t^i\BB^+_{\dR}$. 
   
   Define the ring $\BB_{\cris}^+$ as
   \[ \BB_{\cris}^+=\big\{\sum_{n\geq 0}a_n\frac{\omega^n}{n!}\hspace{3ex}\text{where $a_n\in\tilde{\BB}^+$ is a sequence converging to $0$}\big\},\]
   and $\BB_{\cris}=\BB^+_{\cris}[t^{-1}]$. The ring $\BB_{\cris}$ injects canonically into $\BB_{\dR}$, and it is endowed with the induced Galois action and filtration, as well with a continuous Frobenius $\vp$ which extends the map $\vp:\tilde{\BB}^+\rightarrow \tilde{\BB}^+$. If $V$ is a $\QQ_p$-linear representation of $G_{\QQ_p}$, then $\DD_{\cris}(V)=(V\otimes\BB_{\cris})^{G_{\QQ_p}}$ is a filtered $\vp$-module of dimension $\leq\dim_{\QQ_p}(V)$. We define $V$ to be crystalline if equality holds. 
   
   If $V$ is a $\QQ_p$-linear representation of $G_{\QQ_p}$, say that $V$ is Hodge-Tate, with Hodge-Tate weights $h_1,\dots,h_d$, if we have a decomposition $\CC_p\otimes_{\QQ_p}V\cong \oplus_{i=1}^d\CC_p(h_i)$. Say that $V$ is positive if its Hodge-Tate weights are negative. It is easy to see that a crystalline representation $V$ is Hodge-Tate, and that its Hodge-Tate weights are those integers $h$ such that $\Fil^{-h}\DD_{\cris}(V)\neq \Fil^{1-h}\DD_{\cris}(V)$. 
   
   If $V$ is an $E$-linear representation of $G_{\QQ_p}$, then we define its Hodge-Tate weights to be the weights of the underlying $\QQ_p$-vector space, and we say that $V$ is crystalline if and only if the underlying $\QQ_p$-linear representation is crystalline. In this case, $\DD_{\cris}(V)$ is an $E$-vector space, and the filtration and Frobenius are $E$-linear.


  \subsection{Crystalline representations and Wach modules}\label{crysrepsandwachmods}

    Let $\tilde{\AA}=W(\tilde{\EE})$, and let $\AA_{\QQ_p}$ be the completion of $\ZZ_p[[\pi]][\pi^{-1}]$ in $\tilde{\AA}$ in the $p$-adic topology, so $\AA_{\QQ_p}$ is a complete discrete valuation ring with residue field $\FF_p((\varepsilon-1))$. Let $\BB$ be the completion of the maximal unramified extension of $\BB_{\QQ_p}=\AA_{\QQ_p}[p^{-1}]$ in $\tilde{\BB}$, and define $\AA=\BB\cap\tilde{\AA}$ and $\BB^+=\BB\cap\tilde{\BB}^+$. These rings are endowed with an action of $G_{\QQ_p}$ and of the Frobenius operator $\vp$. One can show that $(\BB^+)^{H_{\QQ_p}}=\ZZ_p[[\pi]][p^{-1}]$, which we denote by $\BB^+_{\QQ_p}$.
    
    We define a left inverse $\psi:\BB\rightarrow \BB$ by $x\rightarrow \vp^{-1}\left(p^{-1}\Tr_{\BB/\vp(\BB)}(x)\right)$. If $x=f(\pi)\in \BB_{\QQ_p}$, then the value of $\psi(x)$ can also be calculated by 
    \[ \vp\circ\psi(x)=\frac{1}{p}\sum_{\zeta^p=1}f(\zeta(\pi+1)-1).\]
    Since the residual extension $\tilde{\EE} / \vp(\tilde{\EE})$ is inseparable of degree $p$, $\psi$ preserves $\AA$ and $\AA_{\QQ_p}$.

    An \'etale $(\vp,G_\infty)$-module over $\AA_{\QQ_p}$ is a finitely generated $\AA_{\QQ_p}$-module $M$, with semi-linear $\vp$ and a continuous action of $G_\infty$ commuting with each other, such that $\vp(M)$ generates $M$ as an $\AA_{\QQ_p}$-module. In~\cite{fontaine90}, Fontaine constructs a functor $T\rightarrow \DD(T)$ which associates to every $\ZZ_p$-linear representation of $G_{\QQ_p}$ an \'etale
    $(\vp,G_\infty)$-module over $\AA_{\QQ_p}$. Moreover, he shows that this functor is an equivalence of categories. By inverting $p$, one also gets an equivalence of categories between the category of $\QQ_p$-linear $p$-adic representations and the category of \'etale $(\vp,G_\infty)$-modules over $\BB_{\QQ_p}$. The left inverse $\psi$ of $\vp$ extends to the $(\vp,G_\infty)$-module.
    
    If $E$ is a finite extension of $\QQ_p$, we extend the Frobenius and the action of $G_\infty$ to $E\otimes\BB_{\QQ_p}$ by $E$-linearity. We then get an equivalence of categories from the category of $E$-linear (or $\calO_E$-linear) representations to the category of \'etale $(\vp,G_\infty)$-modules over $E\otimes\BB_{\QQ_p}$ (resp. over $E\otimes\AA_{\QQ_p}$). 
    
    If $V$ is a crystalline representation, we can say more about the $(\vp,G_\infty)$-module. Let $\AA_{\QQ_p}^+=\ZZ_p[[\pi]]$ and $\BB^+_{\QQ_p}=\AA_{\QQ_p}^+[p^{-1}]$ as above. The following result is shown in~\cite[\S II.1 and \S III.4]{berger03} and~\cite[\S 2]{bergerlizhu04}: If $V$ is an $E$-linear representation, then $V$ is crystalline with Hodge-Tate weights in $[a,b]$ if and only if there exists a (necessarily unique) $E\otimes_{\QQ_p}\BB^+_{\QQ_p}$-module $\NN(V)$ contained in $\DD(V)$ such that the following conditions are satisfied:
    \begin{enumerate}
     \item $\NN(V)$ is free of rank $d=\dim_E(V)$ over $E\otimes_{\QQ_p}\BB^+_{\QQ_p}$;
     \item the action of $G_\infty$ preserves $\NN(V)$ and is trivial on $\NN(V)/\pi\NN(V)$;
     \item $\vp(\pi^b\NN(V))\subset \pi^b\NN(V)$ and $\pi^b\NN(V)/ \vp^*(\pi^b\NN(V))$ is killed by $q^{b-a}$ where $q=\frac{\vp(\pi)}{\pi}$. (If $M$ is a $R$-module  equipped with a Frobenius $\vp$ where $R$ is any ring, then $\vp^*(M)$ denotes the $R$-module generated by $\vp(M)$.) 
    \end{enumerate}

    If $V$ is crystalline and positive, then we can take $b=0$ above, so $\vp$ preserves $\NN(V)$. In this case, if we endow $\NN(V)$ with the filtration $\Fil^i\NN(V)=\{x\in\NN(V)\mid \vp(x)\in q^i\NN(V)\}$, then $\NN(V)/\pi\NN(V)$ is a filtered $E$-linear $\vp$-module, and as shown in~\cite[\S III.4]{berger03} we have an isomorphism $\NN(V)/\pi\NN(V)\cong\DD_{\cris}(V)$. 
    
    If $T$ is a $G_{\QQ_p}$-stable lattice in $V$, then $\NN(T)=\NN(V)\cap\DD(T)$ is an $\calO_E\otimes_{\ZZ_p}\AA_{\QQ_p}^+$-lattice in $\NN(V)$, and by~\cite[\S III.4]{berger03} the functor $T\rightarrow \NN(T)$ gives a bijection between the $G_{\QQ_p}$-stable lattices $T$ in $V$ and the $\calO_E\otimes_{\ZZ_p}\AA_{\QQ_p}^+$-lattices in $\NN(V)$ satisfying 

    \begin{enumerate}
     \item $\NN(T)$ is free of rank $d=\dim_E(V)$ over $\calO_E\otimes_{\ZZ_p}\AA_{\QQ_p}^+$;
     \item the action of $G_\infty$ preserves $\NN(T)$;
     \item $\vp(\pi^b\NN(T))\subset \pi^b\NN(T)$ and $\pi^b\NN(T)/ \vp^*(\pi^b\NN(T))$ is killed by $q^{b-a}$ where $q=\frac{\vp(\pi)}{\pi}$.
    \end{enumerate}
    
    Let $\BB_{\rig,\QQ_p}^+$ be the set of $f(\pi)\in\QQ_p[[\pi]]$ such that $f(X)$ converges for all $X$ in the open unit disc in $\CC_p$. Note that $t\in\BB_{\rig,\QQ_p}^+$. If $V$ is a positive representation of $G_{\QQ_p}$, then as shown in~\cite[\S I.5]{berger03}, we can recover $\DD_{\cris}(V)$ from $\NN(V)$ as $\DD_{\cris}(V)=\big(\BB_{\rig,\QQ_p}^+\otimes_{\BB^+_{\QQ_p}}\NN(V)\big)^{G_\infty}$. Moreover, the inclusion $\DD_{\cris}(V)\subset \BB_{\rig,\QQ_p}^+\otimes_{\BB^+_{\QQ_p}}\NN(V)$ gives rise to an isomorphism
    \[ \iota: \BB^+_{\rig,\QQ_p}[t^{-1}]\otimes_{\QQ_p} \DD_{\cris}(V)\cong \BB_{\rig,\QQ_p}^+[t^{-1}]\otimes_{\BB^+_{\QQ_p}}\NN(V).\]
    In~\cite[proposition 2.12]{berger02}, Berger shows that for all $n\geq 0$ there is an injective map $\vp^{-n}(\BB^+_{\rig,\QQ_p})\rightarrow \BB^+_{\dR}$, which is compatible with the natural map $\vp^{-n}(\tilde{\BB}^+)\rightarrow \BB^+_{\dR}$. It is characterized by the fact that it sends $\pi$ to $\varepsilon^{(n)}\exp(t/ p^n)-1$. Define a derivation $\partial:\BB_{\rig,\QQ_p}^+\rightarrow \BB_{\rig,\QQ_p}^+$ by $\partial=(1+\pi)\frac{d}{d\pi}$. Under the map $ \BB^+_{\rig,\QQ_p}\rightarrow \QQ_p[[t]]$ given by $\pi \mapsto \exp(t) - 1$, $\partial$ corresponds to the derivation $\frac{d}{dt}$. 
    
    If $z\in \Qpn((t))\otimes_{\QQ_p}\DD_{\cris}(V)$, denote the constant coefficient of $z$ by $\partial_V(z)\in\Qpn\otimes_{\QQ_p}\DD_{\cris}(V)$.
    
   \subsection{Iwasawa cohomology and the Fontaine isomorphism}\label{sect-fontaineisom}

    If $T\in \Rep_{\calO_E}(G_{\QQ_p})$, define     
    \[ H^1_{\Iw}(\QQ_p,T)=\varprojlim_n H^1(\QQ_{p,n},T),\]
    where the inverse limit is taken with respect to the corestriction maps. As shown by Fontaine (unpublished -- for a reference see~\cite[Section II]{cherbonniercolmez99}), for any $T \in \Rep_{\calO_E}(G_{\QQ_p})$, there is a canonical isomorphism of $\Lambda_{\calO_E}(G_\infty)$-modules
    \begin{equation}\label{fontaineisom}
     h^1_{\Qp,\Iw}:\DD(T)^{\psi=1} \rTo^\cong H^1_{\Iw}(\Qp,T).
    \end{equation}

    Similarly, for $ V \in \Rep_E (G_{\QQ_p})$, define $H^1_{\Iw}(\QQ_p,V)=H^1_{\Iw}(\QQ_p,T)\otimes_{\ZZ_p}\QQ_p$, where $T$ is any $G_{\QQ_p}$-invariant lattice of $V$; this is independent of the choice of $T$, and $h^1_{\Qp,\Iw}$ extends to an isomorphism of $\Lambda_E(G_\infty)$-modules $\DD(V)^{\psi=1}\cong H^1_{\Iw}(\Qp,V)$.


\section{The Coleman maps}

  \subsection{Positive crystalline representations}\label{positive}
    
   In this subsection, we shall define $d$ Coleman maps for a $d$-dimensional positive crystalline representation $V$, depending on a choice of basis of the Wach module $\NN(T)$ for a lattice $T$ in $V$.

   Let $E$ be a finite extension of $\Qp$. Let $V$ be a positive crystalline $d$-dimensional $E$-linear representation of $G_{\Qp}$ with Hodge-Tate weights $-r_d\le -r_{d-1}\le\cdots\le-r_1\le0$. We assume that $V$ has no quotient isomorphic to $E(-r_d)$ and fix an $\calO_E$-lattice $T$ in $V$ which is stable under $G_{\Qp}$. Write $\NN(T)$ for its Wach module, which is a free $\EA$-module of rank $d$, whereas $\NN(V)=\NN(T)\otimes\Qp$ is a free $E\otimes\BB_{\Qp}^+$-module of rank $d$. Choose an $\EA$-basis $n_1,\ldots,n_d$ of $\NN(T)$ and write $P$ for the matrix of $\vp$ with respect to this basis. Then
    \[
    \begin{pmatrix}\vp(n_1)\\ \vdots \\ \vp(n_d) \end{pmatrix}=P^T\begin{pmatrix}n_1\\ \vdots \\ n_d \end{pmatrix}
    \]
    where $A^T$ denotes the transpose of $A$ if $A$ is a square matrix. Moreover, by \cite[section 3]{bergerbreuil10}, the determinant of $P$ is $q^{r_1+\cdots+r_d}$ up to a unit, where $q=\frac{\vp(\pi)}{\pi}$ as above.

    Let $m=\sum_{i=1}^d r_i$. Then, for $x\in\DD(T(m))^{\psi=1}$, we have $x\in\NN(T(m))^{\psi=1}$ by \cite[appendix A]{berger03}. But $\NN(T(m))=\pi^{-m}\NN(T)\otimes e_m$, where $e_m$ is a vector space basis of $\ZZ_p(m)$. 
    Hence, there exist unique $x_1,\ldots,x_d\in\EA$ such that
    \begin{equation}\label{expressingx}
     x=\pi^{-m}\begin{pmatrix}x_1&\cdots& x_d \end{pmatrix}\begin{pmatrix}n_1\\ \vdots \\ n_d \end{pmatrix}\otimes e_m.
    \end{equation}

    \begin{lemma}\label{kernelofpsi}
     For any $x\in\DD(T(m))^{\psi=1}$, the entries of the row vector
     \[
     {\bfCol}(x):=\begin{pmatrix}x_1&\cdots& x_d \end{pmatrix}q^{m}(P^T)^{-1}-\begin{pmatrix}\vp(x_1)&\cdots& \vp(x_d) \end{pmatrix}\
     \]
     are elements of $(\EA)^{\psi=0}$.
    \end{lemma}
    \begin{proof}
     Recall that the determinant of $P$ is $q^m$ up to a unit in $\EA$, so the entries of $\bfCol(x)$ are indeed elements of $\EA$. Since $\vp(\pi)=\pi q$, (\ref{expressingx}) implies that
     \[
      x=\begin{pmatrix}x_1&\cdots& x_d \end{pmatrix}q^{m}(P^T)^{-1}\vp(\pi^{-m})\begin{pmatrix}\vp(n_1)\\ \vdots \\ \vp(n_d) \end{pmatrix}\otimes e_m.
     \]
     Hence,
     \[
      \psi(x)=\psi\left(\begin{pmatrix}x_1&\cdots& x_d \end{pmatrix}q^{m}(P^T)^{-1}\right)\pi^{-m}\begin{pmatrix}n_1\\ \vdots \\ n_d \end{pmatrix}\otimes e_m.
     \]
     Therefore, $\psi(x)=x$ implies that
     \[
      \psi\left(\begin{pmatrix}x_1&\cdots& x_d \end{pmatrix}q^{m}(P^T)^{-1}\right)=\begin{pmatrix}x_1&\cdots& x_d \end{pmatrix}.
     \]
     Hence the result.
    \end{proof}
    
    \begin{definition}
     For $1\le i\le d$, we define the $i$-th Coleman map $\rmCol_i:\DD(T(m))^{\psi=1}\rightarrow(\EA)^{\psi=0}$ by sending $x$ to the $i$-th component of $\bfCol(x)$.
    \end{definition}

    \begin{lemma}
     Let $n_1,\ldots,n_d$ and $n_1',\ldots,n_d'$ be two bases of $\NN(T)$ with $\begin{pmatrix}n_1\\\vdots\\n_d\end{pmatrix}=M'' \begin{pmatrix}n_1'\\\vdots\\n_d'\end{pmatrix}$. Then, the Coleman maps defined by these two bases, $\bfCol$ and $\bfCol'$ are related by $\bfCol(x)\vp(M'')=\bfCol'(x)$ for all $x\in\DD(T(m))^{\psi=1}$.
    \end{lemma}
    \begin{proof}
     For any $x\in\DD(T(m))^{\psi=1}$, write $x=x_1n_1+\cdots x_dn_d=x_1'n_1'+\cdots x_d'n_d'$. Then,
     \[
     \begin{pmatrix}x_1'&\cdots&x_d'\end{pmatrix}=\begin{pmatrix}x_1&\cdots&x_d\end{pmatrix}M''
     \]
     Let $P$ and $P'$ be the matrices of $\vp$ with respect to $n_1,\ldots,n_d$ and $n_1',\ldots,n_d'$ respectively. Then $P^TM''=\vp(M'')P'^T$. Therefore,
     \begin{align*}
      \bfCol'(x)&=\begin{pmatrix}x_1'&\cdots&x_d'\end{pmatrix}q^m(P'^T)^{-1}-\begin{pmatrix}\vp(x_1')&\cdots&\vp(x_d')\end{pmatrix}\\
      &=\begin{pmatrix}x_1&\cdots&x_d\end{pmatrix}q^mM''(P'^T)^{-1}-\begin{pmatrix}\vp(x_1)&\cdots&\vp(x_d)\end{pmatrix}\vp(M'')\\
      &=\begin{pmatrix}x_1&\cdots&x_d\end{pmatrix}q^m(P^T)^{-1}\vp(M'')-\begin{pmatrix}\vp(x_1)&\cdots&\vp(x_d)\end{pmatrix}\vp(M'').
     \end{align*}
     Hence the lemma.
    \end{proof}

    It is clear that we can extend $\rmCol_i$ to a map from $\DD(V(m))^{\psi=1}$ to $(E\otimes\BB_{\Qp}^+)^{\psi=0}$. By an abuse of notation, we will write this map as $\rmCol_i$ as well. We now relate $\bfCol(x)$ to $(1-\vp)(x)$. By writing down $\vp(x)$, we have the following:
    \begin{equation}
     (1-\vp)(x) = \bfCol(x)\cdot\vp(\pi)^{-m}P^T\begin{pmatrix}n_1\\ \vdots \\ n_d \end{pmatrix}\otimes e_m\label{1-vp} 
    \end{equation}
    
    \begin{remark}
     We see from~\eqref{1-vp} that for any $x$ as above, $(1-\vp)x\in(\vp^*\NN(T(m)))^{\psi=0}$.
    \end{remark}
    
     Note that the maps $\rmCol_i$ are \emph{not} $\Lambda(G_\infty)$-homomorphisms under the canonical action of $G_\infty$ on $(\BB^+_{\QQ_p})^{\psi=0}$ because $G_\infty$ acts non-trivially on the basis $\{ n_i\}_{1\leq i\leq d}$ of $\NN(V)$. We deal with this problem using Theorem~\ref{Lambda} below. Its proof is due to Laurent Berger; we quote it with his permission. In the case when $V$ is the $p$-adic representation associated to a modular form with $v_p(a_p)\geq \lfloor \frac{k-2}{p-1}\rfloor$, we have independently found a proof of this result which uses the basis of $\NN(V)$ constructed in~\cite{bergerlizhu04}. It is more explicit than Berger's proof, and we give it in Section~\ref{ss} since it will be needed to analyse the images of the Coleman maps. For notational simplicity, we take $E=\Qp$ for the time being. Conceptually, there is no difficulty in extending the result to an $E$-linear representation. 
    
   \begin{theorem}\label{Lambda}
    Let $V$ be a crystalline $p$-adic representation of $G_{\QQ_p}$ of dimension $d$, and let $T$ be a $G_{\QQ_p}$-stable lattice in $V$. Then $(\vp^*\NN(T))^{\psi=0}$ is a free $\Lambda(G_\infty)$-module of rank $d$. Moreover, if $n_1^0,\dots,n_d^0$ is a basis of $\NN(T)$, then there exists a basis $n_1,\dots,n_d$ such that $n_i\equiv n_i^0\mod\pi$ for all $i$ and $(1+\pi)\vp(n_1),\dots,(1+\pi)\vp(n_d)$ forms a $\Lambda(G_\infty)$-basis of $(\vp^*\NN(T))^{\psi=0}$.
   \end{theorem}
   
   Note that in this theorem we do not assume that $V$ is positive. The proof of this result requires several preliminary lemmas. We assume without loss of generality that $\chi(\gamma)=1+p$. For $k\geq 0$, define
   \[ p_k=(1-\gamma)(1-\chi(\gamma)^{-1}\gamma)\dots(1-\chi(\gamma)^{1-k}\gamma),\]
   which is an element of $\Lambda(\Gamma)$. 
   
   \begin{lemma}\label{lem1}
    If $a\in\ZZ_p$ and $x\in\NN(T)$ and $f\in\AA^+_{\QQ_p}$ and $g\in G_\infty$, then 
    \[ (1-ag)(fx)=((1-ag)f)x+ag(f)((1-g)x).\]
   \end{lemma}
   \begin{proof}
    Immediate.
   \end{proof}

   \begin{lemma}\label{lem2}
    The map $\mathfrak{M}:\Lambda(G_\infty)\rightarrow (\AA^+_{\QQ_p})^{\psi=0}$ given by $f\rightarrow f(1+\pi)$ is an isomorphism of $\Lambda(G_\infty)$-modules, which takes $p_k\Lambda(G_\infty)$ to $\vp(\pi)^k(\AA^+_{\QQ_p})^{\psi=0}$.
   \end{lemma}

   \begin{proof}
    The first assertion is standard (we recall the relevant theory in section \ref{appendix} below). Note that $\gamma(\pi)=\chi(\gamma)\pi + O(\pi^2)$, which implies that the image of $p_k\Lambda(G_\infty)$ is contained in $\vp(\pi)^k(\AA^+_{\QQ_p})^{\psi=0}$. 
    Hence the surjection $\Lambda(G_\infty) \twoheadrightarrow (\AA^+_{\QQ_p})^{\psi=0}$ gives a surjection $\Lambda(G_\infty) / p_k \twoheadrightarrow (\AA^+_{\QQ_p})^{\psi=0}/\vp(\pi)^k$. Since both are free $\ZZ_p$-modules of rank $k(p-1)$, this must be an isomorphism.
   \end{proof}
   
   \begin{remark} Following the terminology of \cite[\S II.6]{berger03}, we refer to the inverse of $\mathfrak{M}$ as the \emph{Mellin transform}. 
   \end{remark}

   Let $n_1^0,\dots,n_d^0$ be a basis of $\NN(T)$. Since the action of $G_\infty$ on $\NN(T)$ is trivial modulo $\pi$, we have $(1-g)n_i^0\in \pi\NN(T)$ for all $1\leq i\leq d$ and for all $g\in G_\infty$. 
    
   \begin{lemma}\label{lem3}
    Let $g$ be a topological generator of $G_\infty$, and write $(1-g)n_i^0=\pi m_i$ for some $m_i\in\NN(T)$. If we put $n_i=n_i^0-\frac{\pi m_i}{1-\chi(g)}$, then $n_1,\dots, n_d$ is a basis of $\NN(T)$, and $(1-\gamma)n_i\in \pi^2\NN(T)$.  
   \end{lemma}
   \begin{proof}
    Note that since $p\neq 2$ and $g$ is a topological generator of $G_\infty$, $1-\chi(g)\in\ZZ_p^\times$, so $n_i\in\NN(T)$ for all $i$, and they are obviously a basis. Since $g(\pi)=\chi(g)\pi +\calO(\pi^2)$, this basis is designed such that $(1-g)n_i\in \pi^2\NN(T)$, and this implies that $(1-g)n_i\in\pi^2 \NN(T)$. 
   \end{proof}
   
   Let $\calN$ be the $\Lambda(G_\infty)$-submodule of $(\vp^*(\NN(T))^{\psi=0}$ generated by $(1+\pi)\vp(n_1),\dots,(1+\pi)\vp(n_d)$. 
   
   \begin{lemma}\label{lem4}
    Let $y\in(\vp^*\NN(T))^{\psi=0}$. Then there exist $\mathfrak{n}\in \calN$ and $z\in (\vp^*(\NN(T))^{\psi=0}$ such that $y=\mathfrak{n}+\vp(\pi)z$. 
   \end{lemma}
   \begin{proof}
    Write $y=\sum_{i=1}^d y_i\vp(n_i)$ with $y_i\in (\AA^+_{\QQ_p})^{\psi=0}$. By Lemma~\ref{lem2} we can write $y_i=b_i(1+\pi)$ for some $b_i\in \Lambda(G_\infty)$, and Lemma~\ref{lem3} implies that $b_in_i\equiv n_i\mod\pi^2\NN(T)$. Therefore, we have
    \[ \sum_{i=1}^d b_i\big( (1+\pi)\vp(n_i)\big)- \sum_{i=1}^d y_i\vp(n_i)\in \vp(\pi)^2(\vp^*\NN(T))^{\psi=0},\]
    which is slightly better than the lemma.  
   \end{proof}
   
   Lemma~\ref{lem4} can be generalized to all $k\geq 0$:
   
   \begin{proposition}\label{induction}
    Let $k\geq 0$ and $y\in\vp(\pi)^k(\vp^*\NN(T))^{\psi=0}$. Then there exists $\mathfrak{n}\in p_k\calN$ and $z\in (\vp^*\NN(T))^{\psi=0}$ such that $y=\mathfrak{n}+ \vp(\pi)^{k+1}z$.
   \end{proposition}
   \begin{proof}
    The case $k=0$ is just Lemma~\ref{lem4}. Assume that $k\geq 1$, and that the result is true for $k-1$. If $y=\sum_{i=1}^d y_i\vp(n_i)$ with $y_i\in \vp(\pi)^{k}(\AA^+_{\QQ_p})^{\psi=0}$, then we can write $y_i=b_i(1+\pi)$ with $b_i\in p_k\Lambda(G_\infty)$ by Lemma~\ref{lem2}. By the definition of $p_k$, we can write $b_i=(1-a\gamma)c_i$ with $a=\chi(\gamma)^{1-k}$ for some $c_i\in\Lambda(G_\infty)$. Moreover, $p_{k-1}|c_i$ for all $i$. Let $x_i=c_i(1+\pi)$, then 
    \begin{align*}
     \sum_{i=1}^d y_i\vp(n_i) &= \sum_{i=1}^d((1-a\gamma)x_i)\vp(n_i)\\
     &= (1-a\gamma)\left(\sum_{i=1}^d x_i\vp(n_i)\right)-a\sum_{i=1}^d\gamma(x_i)((1-\gamma)\vp(n_i))
    \end{align*}    
    by Lemma~\ref{lem1}. Let $z_0:= \sum_{i=1}^d\gamma(x_i)((1-\gamma)\vp(n_i))$. By Lemma~\ref{lem3} and the fact that $p_{k-1}|c_i$ (so $\vp(\pi)^{k-1}|x_i$), we have $z_0\in \vp(\pi)^{k+1}(\vp^*\NN(T))^{\psi=0}$.
    
    Consider the element $\sum_{i=1}^d x_i\vp(n_i)$ where $x_i=c_i(1+\pi)$ is divisible by $\vp(\pi)^{k-1}$ by Lemma~\ref{lem2} as $p_{k-1}|c_i$. Therefore, by induction, we can write $\sum_{i=1}^d x_i\vp(n_i)$ as $x+\vp(\pi)^kw$ with $x\in p_{k-1}\calN$ and $w\in (\vp^*\NN(T))^{\psi=0}$. If we set
    \begin{align*}
\mathfrak{n}&=(1-a\gamma)(x),\\
     \vp(\pi)^{k+1}z & = z_0+(1-\chi(\gamma)^{-k}\gamma)(\vp(\pi)^kw)\hspace{3ex}\text{and} \\
     py_1& = \big(\chi(\gamma)^{1-k}-\chi(\gamma)^{-k}\big)\gamma(\vp(\pi)^kw),
    \end{align*}
    then $y=\mathfrak{n}+\vp(\pi)^{k+1}z+py_1$ with $\mathfrak{n}\in p_k\calN$, $z\in (\vp^*\NN(T))^{\psi=0}$ and $y_1\in \vp(\pi)^k(\vp^*\NN(T))^{\psi=0}$.
    
    Iterating this gives us $y_j\in (\vp^*\NN(T))^{\psi=0}$ and converging sequences $\mathfrak{n}_j\in \calN$ and $z_n\in (\vp^*\NN(T))^{\psi=0}$ such that 
    \[ y= \mathfrak{n}_j+\vp(\pi)^{k+1}z_j+p^jy_j.\]
    The proposition follows by taking $\mathfrak{n}$ and $z$ to be the limits of $\mathfrak{n}_j$ and $z_j$, respectively. 
   \end{proof}
   
   \begin{proof}[Proof of Theorem~\ref{Lambda}] If $y\in (\vp^*\NN(T))^{\psi=0}$, the iterating Proposition~\ref{induction} shows that for all $k\geq 0$ we can write
   \[ y=\mathfrak{n}_0+\mathfrak{n}_1+\dots+\mathfrak{n}_k+\vp(\pi)^{k+1}z\]
   with $\mathfrak{n}_j\in p_j\calN$. Passing to the limit over $k$ shows that $y=\sum_{i\geq 0}\mathfrak{n}_i\in\calN$, which shows that $(1+\pi)\vp(n_1),\dots,(1+\pi)\vp(n_d)$ form a generating set of the $\Lambda(G_\infty)$-module $(\vp^*\NN(T))^{\psi=0}$. 
   
   Finally, the map $\Lambda(G_\infty)^{\oplus d}\slash p_k\Lambda(G_\infty)^{\oplus d}\rightarrow (\vp^*\NN(T))^{\psi=0}\slash \vp(\pi)^k(\vp^*\NN(T))^{\psi=0}$ is a surjective map between two $\ZZ_p$-modules of equal rank, so that it is injective, and therefore the kernel of $\Lambda(G_\infty)^{\oplus d}\rightarrow (\vp^*\NN(T))^{\psi=0}$ is equal to $\bigcap_{k\geq 0}p_k\Lambda(G_\infty)^d=0$. This finishes the proof.\end{proof}
   
   We now resume our assumption that $V$ is a positive crystalline $E$-linear representation of $G_{\QQ_p}$, with Hodge--Tate weights $-r_i$ such that $\sum_i r_i = m$, and $T \subset V$ an $\calO_E$-lattice, as above. Applying theorem \ref{Lambda} to the representation $V(m)$, we find that for any basis $n_1^0, \dots, n_d^0$ of $\NN(T)$, there is a basis $n_1, \dots, n_d$ of $\NN(T)$ with $n_i = n_i^0 \bmod \pi$ such that the vectors $(1 + \pi) \vp(\pi^{-m} n_i \otimes e_m)$ are a basis of $(\vp^* \NN(T(m))^{\psi = 0}$ as a $\Lambda_{\calO_E}$-module. With respect to such a basis $n_1, \dots, n_d$, we make the following definitions:

   \begin{definition}
    Define the \emph{Iwasawa transform} to be the $\Lambda_{\calO_E}(G_\infty)$-equivariant isomorphism
    \[ \J: (\vp^*\NN(T(m)))^{\psi=0}\rTo \Lambda_{\calO_E}(G_\infty)^{\oplus d}\]
    determined by sending $(1+\pi)\vp(n_i\otimes\pi^{-m}e_m)$ to $(0,\dots,0,1,0,\dots,0)$, where the $1$ is the $i$-th entry. 
   \end{definition}
   
   \begin{definition} 
    Define $\uCol:\NN(T(m))^{\psi=1}\rightarrow \Lambda_{\calO_E}(G_\infty)^{\oplus d}$ as $\J\circ (1-\vp)$, and for $1\leq i\leq d$, let $\uCol_i:\NN(T(m))^{\psi=1}\rightarrow \Lambda_{\calO_E}(G_\infty)$ be the composition of $\uCol$ with the projection onto the $i$-th component. 
   \end{definition}
   
   \begin{note}
    For all $1\leq i\leq d$, the map $\uCol_i$ is $\Lambda_{\calO_E}(G_\infty)$-equivariant. 
   \end{note}
   
   \subsection{Comparison with \texorpdfstring{$\Dcris$}{Dcris}}

   We now give an alternative formula for the Coleman maps of the previous subsection using the comparison isomorphisms between the Wach module $\NN(V)$ and $\Dcris(V)$.

   Recall from section \ref{crysrepsandwachmods} that for any positive crystalline representation $V$ we have a canonical isomorphism $\NN(V)/\pi\NN(V)\cong\Dcris(V)$ (\cite[\S~III.4]{berger03}).

   \begin{lemma}\label{choice}
    Let $V$ be a positive crystalline $E$-linear representation of $G_{\Qp}$. Given any basis $\nu_1,\dots,\nu_d$ of $\Dcris(V)$ over $E$, we can lift it to a basis of $n_1,\dots,n_d$ of $\NN(V)$ over $E \otimes \BB_{\Qp}^+$. Moreover, we may assume that $(1+\pi)\vp(\pi^{-m} n_1 \otimes e_m),\dots,(1+\pi)\vp(\pi^{-m} n_d \otimes e_m)$ is a $\Lambda_{E}(G_\infty)$-basis of $(\vp^*\NN(V(m)))^{\psi=0}$.
   \end{lemma}

   \begin{proof}
    Let $T$ be a $G_{\QQ_p}$-stable $\OO_E$-lattice in $V$. By theorem \ref{Lambda} above, we may choose a $\calO_E\otimes\AA_{\Qp}^+$-basis $\bar{n}_1,\ldots,\bar{n}_d$ of $\NN(T)$ such that $(1 + \pi)\vp(\pi^{-m} \bar{n}_i \otimes e_m)$ is a $\Lambda_{\calO_E}(G_\infty)$-basis of $(\vp^* \NN(T(m)))^{\psi = 0}$. Hence these elements are also a $\Lambda_E(G_\infty)$-basis of $(\vp^* \NN(V(m)))^{\psi = 0}$.

    By the comparison isomorphism, the elements $\{\bar{\nu}_i:=\bar{n}_i\mod\pi:i=1,\ldots,d\}$ give a basis of $\Dcris(V)$ over $E$. Let $A\in GL_d(E)$ be the change of basis matrix from $\nu_1,\ldots,\nu_d$ to $\bar{\nu}_1,\ldots,\bar{\nu}_d$. On applying $A^{-1}$ to $\bar{n}_1,\ldots,\bar{n}_d$, we obtain a basis $n_1,\ldots,n_d$ of $\NN(V)$ lifting $\nu_1, \dots, \nu_d$. Now it is clear that $(1+\pi)\vp(\pi^{-m} n_1 \otimes e_m),\dots,(1+\pi)\vp(\pi^{-m} n_d \otimes e_m)$ is a $\Lambda_E(G_\infty)$-basis of $(\vp^* \NN(V(m)))^{\psi = 0}$, since it differs from the original basis by the scalar matrix $A^{-1}$, which is clearly invertible in $\Lambda_E(G_\infty)$.
   \end{proof}

    With respect to such a basis $n_1, \dots, n_d$ of $\NN(V)$, we can clearly define an Iwasawa transform and Coleman map as above but with $E$-coefficients,
    \begin{align*}
     \J &: (\vp^*\NN(V(m)))^{\psi=0}\rTo^\cong \Lambda_E(G_\infty)^{\oplus d}\\
     \uCol &: \NN(V(m))^{\psi=1}\rTo \Lambda_{E}(G_\infty)^{\oplus d},
    \end{align*}
    which are homomorphisms of $\Lambda_E(G_\infty)$-modules.

    \begin{remark}\label{integrality}
     If $T$ is an $\calO_E$-lattice in $V$ stable under $G_{\Qp}$ and the $\calO_E$-lattice in $\Dcris(V)$ spanned by $\nu_1, \dots, \nu_d$ is the reduction of $\NN(T)$, then we can define the Coleman maps integrally, as in the previous section. In section \ref{ss} below we will work with a specific basis $\nu_i$ for which such a lattice $T$ can be explicitly constructed.
    \end{remark}

    Now, let $\nu_1,\ldots,\nu_d$ be a basis of $\Dcris(V)$ over $E$, and $n_1, \dots, n_d$ a basis of $\NN(V)$ lifting $\nu_1, \dots, \nu_d$ as in lemma \ref{choice}. We write $A_\vp$ for the matrix of $\vp$ on $\Dcris(V)$ with respect to the basis $\nu_1, \dots, \nu_d$. Again by \cite[section 3]{bergerbreuil10}, $\EB$ is a B\'{e}zout ring and
    \begin{equation}\label{elementary}
    [(\EB)\otimes_{E\otimes\BB_{\Qp}^+}\NN(V):(\EB)\otimes_{E}\Dcris(V)]=\left[\left(\frac{t}{\pi}\right)^{r_1};\cdots;\left(\frac{t}{\pi}\right)^{r_d}\right].
    \end{equation}
    In other words, there exists $\EB$-bases $w_1,\ldots,w_d$ and $v_1,\ldots,v_d$ for $(\EB)\otimes_{E\otimes\BB_{\Qp}^+}\NN(V)$ and $(\EB)\otimes_{E}\Dcris(V)$ respectively such that $v_i=(t/\pi)^{r_i}w_i$ for $i=1,\ldots,d$. Therefore, the change of basis matrix $M'\in M_d(E\otimes\Brig)$ with
    \begin{equation}\label{Mprime}
    \begin{pmatrix}\nu_1\\ \vdots \\ \nu_d \end{pmatrix}=M'\begin{pmatrix}n_1\\ \vdots \\ n_d \end{pmatrix},
    \end{equation}
    has determinant $(t/\pi)^m$ up to a unit in $E\otimes\Brig$. Moreover, since $n_1, \dots, n_d$ lifts $\nu_1, \dots, \nu_d$, we have $M'|_{\pi=0}=I$, the identity matrix. The compatibility of the action of $\vp$ implies that
    \begin{equation}\label{relatingphi}
      \vp(M')P^T=A_\vp^TM',
    \end{equation}
    where $P$ is the matrix of $\vp$ on $\NN(V)$ with respect to the basis $n_1, \dots, n_d$ as in the previous subsection. We can now rewrite \eqref{expressingx}:
    \begin{equation}\label{xoverDcris}
     x=\begin{pmatrix}x_1&\cdots& x_d \end{pmatrix}\cdot\left(\frac{t}{\pi}\right)^mM'^{-1}\begin{pmatrix}\nu_1\\ \vdots \\ \nu_d \end{pmatrix}\otimes t^{-m}e_m
    \end{equation}  
    with $(t/\pi)^mM'^{-1}\in M_d(E\otimes\Brig)$ and $\nu_i\otimes t^{-m}e_m$, $i=1,\ldots,d$ a basis of $\Dcris(V(m))$.

   Rewriting \eqref{1-vp} using this, we see that
   \begin{equation} 
   (1-\vp)(x) = \bfCol(x)\cdot\left(\frac{t}{\pi q}\right)^mP^TM'^{-1}\begin{pmatrix}\nu_1\\ \vdots \\ \nu_d \end{pmatrix}\otimes t^{-m}e_m.\label{xDcris}
   \end{equation}


   \subsection{Supersingular modular forms}\label{Colemanmf}
   
    We now apply the theory of Coleman maps developed above to the Galois representations attached to modular forms.

    Let $f=\sum a_nq^n$ be a normalized new eigenform of weight $k$ and character $\epsilon$. Let $p$ be an odd prime which does not divide the level of $f$. For simplicity, we will always assume that $\epsilon(p)=1$. In particular $a_p=\bar{a}_p$. We write $E=\QQ_p(a_n:n\ge1)$, which is the completion of the coefficient field $F$ of $f$ at the prime above $p$ determined our choice of embeddings. Then, by Deligne \cite{deligne69}, we can associate to $f$ a 2-dimensional $E$-linear representation $V_f$ of $G_{\QQ}$. Moreover, when restricted to $G_{\Qp}$, $V_f$ is crystalline and its de Rham filtration is given by
    \begin{equation}\label{filtration}
     \Dcris^i(V_f)=
     \left\{
     \begin{array}{ll}
      E\nu_1\oplus E\nu_2         & \text{if $i\le0$}\\
      E\nu_1                     & \text{if $1\le i\le k-1$}\\
      0                          & \text{if $i\ge k$}
     \end{array}\right.
    \end{equation}
    for some basis $\nu_1$, $\nu_2$ over $E$. We further assume that $v_p(a_p)\ne0$, i.e. $f$ is supersingular at $p$. Then $\nu_1$ is not an eigenvector of $\vp$ by \cite[Theorem 16.6]{kato04} and we may choose $\nu_2=p^{1-k}\vp(\nu_1)$ so that the matrix $A_\vp$ of $\vp$ with respect to the basis $\nu_1$, $\nu_2$ is given by
    \[
    \begin{pmatrix} 0 & -1 \\ p^{k-1}& a_p \end{pmatrix}
    \]
    since $\vp^2-a_p\vp+p^{k-1}=0$ (c.f. \cite{scholl90}). We call such a basis a `good basis' for $\Dcris(V_f)$.

    Let $\bar{\nu}_1$ and $\bar{\nu}_2$ be a `good basis' of $\Dcris(V_{\bar{f}})$. Then, the matrix of $\vp$ with respect to this basis is equal to $A_\vp$ also since $a_p=\bar{a}_p$.

    Note that $V_{\bar{f}}$ has Hodge-Tate weights $0$ and $-k+1$, so it is positive. Fix a basis $n_1,n_2$ of $\NN(V_{\bar{f}})$ satisfying the conditions in Lemma~\ref{choice}, so $\begin{pmatrix}\bar{\nu}_1\\\bar{\nu}_2\end{pmatrix}=M'\begin{pmatrix}n_1\\n_2\end{pmatrix}$ with $M'|_{\pi=0}=I$. We obtain two pairs of Coleman maps associated to $f$:
    \begin{align*}
    \rmCol_{i}:\DD(V_{\bar{f}}(k-1))^{\psi = 1} & \rTo (E \otimes \BB_{\QQ_p}^+)^{\psi=0}, \\
    \uCol_{i}:\DD(V_{\bar{f}}(k-1))^{\psi = 1} & \rTo \Lambda_{E}(G_\infty), 
    \end{align*}
    for $i=1,2$.

    Recall the isomorphism \eqref{fontaineisom} above:
    \[
     h^1_{\Qp,\Iw}:\DD(V_{\bar{f}}(k-1))^{\psi=1}\cong H^1_{\Iw}(\Qp,V_{\bar{f}}(k-1)).
    \]    
    We can therefore consider the localization of Kato's zeta element $\kato$ from \cite{kato04} (see section \ref{katozeta} below), which \textit{a priori} is an element of $H^1_{\Iw}(\Qp, V_{\bar{f}}(k-1))$, as an element of $\DD(V_{\bar{f}}(k-1))^{\psi=1}$. We can now define two pairs of $p$-adic $L$-functions:
    
    \begin{definition}\label{padicl}
     For $i=1,2$, define $L_{p,i}=\rmCol_i(\kato) \in (E\otimes\BB_{\Qp}^+)^{\psi=0}$ and $\tilde{L}_{p,i}=\uCol_i(\kato)\in \Lambda_E(G_\infty)$ where $\kato$ is the localization of the Kato zeta element.
    \end{definition}

    The reason why we consider $V_{\bar{f}}$ instead of $V_f$ will become apparent in section \ref{construction} below. In addition, below is a list of assumptions which we will need later when we prove different results.
    \begin{itemize}
     \item\textbf{Assumption (A)}: $k\ge3$ or $a_p=0$.

     \item\textbf{Assumption (B)}: $a_p$ is not of the form $p^j+p^{k-2-j}$ for some integer $1\le j\le k-3$.

     \item\textbf{Assumption (C)}: $v_p(a_p)>\lfloor(k-2)/(p-1)\rfloor$.

     \item\textbf{Assumption (D)}: $p\ge k-1$.
    \end{itemize}

   
   \subsection{Relation to the Perrin-Riou pairing}\label{construction}
   
    Let $\alpha$ and $\beta$ be the roots of the quadratic $X^2-a_pX+p^{k-1}$. By the work of Amice--V\'elu and Vishik cited in the introduction, we can associate to $\alpha$ and $\beta$ $p$-adic $L$-functions $L_{p,\alpha}$ and $L_{p,\beta}$ respectively; see \cite[\S 11]{MTT} for an account of the construction. We will relate them to $L_{p,i}$, $i=1,2$, as defined above. We first prove some preliminary results on general crystalline representations.
    
    Let $\gamma$ be a topological generator of $\Gamma$. Define 
    \[ \HG=\big\{ f(\gamma-1)\mid f(X)\in\QQ_p[\Delta][[X]]\hspace{1ex}\text{such that $f$ converges for all $X\in\CC_p$ with $| X | < 1$} \big\}.\]
     We can identify $\HG$ with $(\BB_{\rig,\QQ_p}^+)^{\psi=0}$ via the map
    \begin{equation}\begin{aligned}
     \mathfrak{M}: \HG & \rTo (\BB_{\rig,\QQ_p}^+)^{\psi=0} \\
     f(\gamma-1) & \rMapsto f(\gamma-1)(\pi+1),
    \end{aligned}\label{mellin}\end{equation}
    where any $g\in G_\infty$ acts on $\pi$ by $(\pi+1)^{\chi(g)}-1$. As shown in~\cite[B.2.8]{perrinriou00}, this map is a bijection, extending the isomorphism $\Lambda(G_\infty) \to (\AA^+_{\Qp})^{\psi = 0}$ of Lemma \ref{lem2}. For $r\geq 1$, define
    \[ \calH_r^{\temp}=\left\{ \sum_{\sigma\in\Delta}\sum_{n\geq 0} c_{n,\sigma}\sigma X^n: \lim_{n\rightarrow +\infty}\frac{| c_{n,\sigma}|_p}{n^r}=0\right\}.\]
    Let $\calH^{\temp}=\bigcup_{r\geq 1}\calH^{\temp}_r$, and define $\calH^{\temp}(G_\infty)=\{ f(\gamma-1)\mid f(X)\in\calH^{\temp}\}$.

    Let $V$ be any crystalline $E$-linear representation of $G_{\QQ_p}$, and let $h$ be a positive integer such that $\Fil^{-h}\DD_{\cris}(V)=\DD_{\cris}(V)$. .
    Denote by 
    \[ \Omega_{V,h}: \big( \calH^{\temp}(G_\infty)\otimes \DD_{\cris}(V)\big)^{\Sigma=0}\rTo \calH^{\temp}(G_\infty)\otimes_{\Lambda_{\QQ_p}}H^1_{\Iw}(\Qp,V)\]   
    Perrin-Riou's exponential map as constructed in~\cite{perrinriou94}. Here, 
    \[ \Sigma:\Brig\otimes_{\Qp}\DD_{\cris}(V)\rightarrow \bigoplus_{k=0}^h \big(\DD_{\cris}(V)/ (1-p^k\varphi)\big)(k)\] 
    is the map sending $f\in \Brig\otimes_{\Qp}\DD_{\cris}(V)$ to the class of $\oplus\partial^k(f)(0)$, where $\partial=(1+\pi)\frac{d}{d\pi}$ is the derivation on $\BB_{\rig,\QQ_p}^+$ defined in \S \ref{crysrepsandwachmods}. 
    Since $\Omega_{V, h}$ is a homomorphism of $\calH^{\temp}(G_\infty)$-modules, we can extend scalars to get 
    \begin{equation}\label{Omega}
     \Omega_{V,h}: \big( (\Brig)^{\psi=0}\otimes \DD_{\cris}(V)\big)^{\Sigma=0}\rTo     
     \HG\otimes_{\Lambda_{\QQ_p}}H^1_{\Iw}(\Qp,V), 
    \end{equation}
    where we identify $(\Brig)^{\psi=0}$ with $\HG$ via $\mathfrak{M}$. 
    
    \begin{remark}
     We will only apply \eqref{Omega} to elements in which lie in the image of $\calH^{\temp}(G_\infty)\otimes \DD_{\cris}(V)$ under $\mathfrak{M}$, so we can refer to~\cite{perrinriou94} for the properties of $\Omega_{V,h}$. The reason for extending scalars to $\HG$ is that we want to be able to use Berger's description of the exponential map in~\cite[\S II.5]{berger03}.
    \end{remark}
    
    Recall that we have chosen a $p$-power compatible system $\varepsilon^{(n)}$, $n\geq 0$, of $p$-power roots of unity. 
    
   \begin{proposition}\label{PRvalues}
    Assume that $V$ is a crystalline representation of $G_{\QQ_p}$. Let $h\geq 1$ such that $\Dcris^{-h}(V)=\Dcris(V)$ and $p^{-j}$ is not an eigenvalue of $\vp$ on $\Dcris(V)$ for $j\in\ZZ$ with $0\le j\le h$. Then, for all $v\in\DD_{\cris}(V)$, the projection to the $n$-th local cohomology $H^1(\QQ_{p,n},V)$ of $\frac{1}{(h-1)!}\Omega_{V,h}((1+\pi)\otimes v)$ is given by
    \begin{equation}
     \left\{
      \begin{array}{ll}
        p^{-n}\exp_{F_n,V}\left(\sum_{m=0}^{n-1}\varepsilon^{(n-m)}\otimes\vp^{m-n}(v)+(1-\vp)^{-1}(v)\right)    & \text{if $n\geq 1$}\\
            \exp_{{\QQ_p},V}\left(\left(1-\frac{\vp^{-1}}{p}\right)(1-\vp)^{-1}(v)\right)           & \text{if $n=0$.}
      \end{array}\right.
    \end{equation}
   \end{proposition}
   \begin{proof}
    Let $g\in(\Brig)^{\psi=0}\otimes_{\QQ_p}\Dcris(V)$. We write $\Delta_j(g)=\partial^j(g)(0)$ and
    \[
     \tilde{g}=g-\sum_{j=0}^{h}\frac{1}{j!}\log_p^j\Delta_j(g).
     \]
     By \cite[section 2.2]{perrinriou94}, the sum $\sum_{n=0}^{\infty}\vp^n(\tilde{g})$ converges. A solution to $(1-\vp)G=g$ with $G\in \big(\BB^+_{\rig,{\QQ_p}}\otimes_{\Qp}\DD_{\cris}\big)^{\psi=1}$ is given by
     \[
      G=\sum_{n=0}^{\infty}\vp^{n}(\tilde{g})+\sum_{j=0}^h\frac{1}{j!}\log_p^jv_j
     \]  
     where $v_j\in\Dcris(V)$ is such that $\Delta_j(g)=(1-p^j\vp)v_j$. Now, take $g=(1+\pi)\otimes v$, so $\Delta_j(g)=v$ for all $j$. Let $n$ be a positive integer, then     
     \begin{equation}
      \vp^{m}(\tilde{g})(\varepsilon^{(n)}-1)=
      \left\{
      \begin{array}{ll}
      (\varepsilon^{(n-m)}-1)\otimes\vp^m(v)         & \text{if $m<n$}\\
      0                          & \text{otherwise.}
      \end{array}\right.
     \end{equation}
    Therefore, we have
    \begin{align*}
     G(\varepsilon^{(n)}-1) &= \sum_{m=0}^{n-1}(\varepsilon^{(n-m)}-1)\otimes\vp^m(v)+(1-\vp)^{-1}(v)\\
     &= \sum_{m=0}^{n-1}\varepsilon^{(n-m)}\otimes\vp^m(v)+(1-\vp)^{-1}\vp^n(v)
    \end{align*}
    Hence, by the main result in~\cite{perrinriou94}, the $n$-th component of $\frac{1}{(h-1)!}\Omega_{V,h}((1+\pi)\otimes v)$ is given by the image of
    \begin{equation}
     p^{-n}\vp^{-n}G(\varepsilon^{(n)}-1)=\frac{1}{p^n}\left(\sum_{m=0}^{n-1}\varepsilon^{(n-m)}\otimes\vp^{m-n}(v)+(1-\vp)^{-1}(v)\right)
    \end{equation}
    under the map $\exp_{\QQ_{p,n},V}$. For the $0$-th level, it is given by the image of
    \begin{align*}
     \Tr_{\QQ_{p,1}/ {\QQ_p}}\left(\frac{1}{p}\vp^{-1}G(\varepsilon^{(1)}-1)\right)
        &= \frac{1}{p}\Tr_{\QQ_{p,1}/ \Qp}(\varepsilon^{(1)}\otimes\vp^{-1}(v)+(1-\vp)^{-1}(v))\\
     &= \frac{1}{p}\left(-1\otimes\vp^{-1}(v)+(p-1)(1-\vp)^{-1}(v)\right)\\
     &= \left(1-\frac{\vp^{-1}}{p}\right)(1-\vp)^{-1}(v).
    \end{align*}
    under the map $\exp_{{\QQ_p},V}$, so we are done.
   \end{proof}
    
    Define the Perrin-Riou pairing $\langle\hspace{1ex},\hspace{1ex}\rangle_V$ by
    \begin{eqnarray*}
     \langle\hspace{1ex},\hspace{1ex}\rangle_V  :  H^1_{\Iw}(\Qp,V)\times H^1_{\Iw}(\Qp,V^*(1)) & \rTo & \Lambda_{E}(G_\infty),\\
     \langle (x_n),(y_n)\rangle_V & = &\varprojlim \Sigma_{\tau\in G_{\Qp}/ G_{\Qp}^{p^n}}(\tau(x_n)\cup y_n)\tau.
    \end{eqnarray*}
    
    \begin{remark}
     In \cite{perrinriou94}, the pairing is defined by
     \[\langle (x_n),(y_n)\rangle_V = \varprojlim \Sigma_{\tau\in G_{\Qp}/ G_{\Qp}^{p^n}}(\tau^{-1}(x_n)\cup y_n)\tau.\]
     We use the different convention so that the map $\mathcal{L}_{h,z}$ defined in~\eqref{lmap} below is a $\Lambda(G_\infty)$-homomorphism. 
    \end{remark}
    
    We can extend the pairing $\langle\hspace{1ex},\hspace{1ex}\rangle_V$ to 
    \[ \langle\hspace{1ex},\hspace{1ex}\rangle_V: \Big( \HG \otimes_{\Lambda_{\QQ_p}}H^1_{\Iw}(\Qp,V)\Big)\times \Big(\HG \otimes_{\Lambda_{\QQ_p}}H^1_{\Iw}(\Qp,V^*(1))\Big) \rTo \HG.\]
   Any $z\in \big( (\Brig)^{\psi=0}\otimes \DD_{\cris}(V)\big)^{\Sigma=0}$ therefore defines a map 
   \begin{equation}\begin{aligned}
     \mathcal{L}_{h,z}: H^1_{\Iw}(\Qp,V^*(1)) & \rTo \HG,\\
           (y_n)_{n\geq 0}& \rMapsto \big\langle \Omega_{h,V}(z), (y_n)\big\rangle_V.
   \end{aligned} \label{lmap}\end{equation}
   As recalled in section \ref{sect-fontaineisom} above, for any $p$-adic representation $V$ of $G_{\QQ_p}$ we have a canonical isomorphism
    \[ h^1_{\Qp,\Iw}:\DD(V)^{\psi=1}\cong H^1_{\Iw}(\Qp,V).\]
   
    \begin{lemma}\label{twist}
     For all $j\in\ZZ$ and for all $y\in\DD(V)^{\psi=1}$ and $y'\in\DD(V^*(1))^{\psi=1}$, we have
     \[ \partial^j\langle h^1_{\Qp,\Iw}(y),h^1_{\Qp,\Iw}(y')\rangle_V = \langle h^1_{\Qp,\Iw}(y\otimes e_j), h^1_{\Qp,\Iw}(y'\otimes e_{-j})\rangle_{V(j)}.\]
    \end{lemma}
    \begin{proof}
     See Lemme ii) Section 3.6 in~\cite{perrinriou94}.
    \end{proof}
   
    We now return to the setting in Section~\ref{Colemanmf}. We will apply Perrin-Riou's theory that we recalled above to the crystalline representation $V_f(1)$. In particular, $V_f(1)^*(1)\cong V_{\bar{f}}(k-1)$. By (\ref{filtration}), we can take $h=1$. Note that $\varphi$ acts on $\DD_{\cris}(V_f(1))$ by $\begin{pmatrix} 0 & -p^{-1} \\ p^{k-2} & p^{-1}a_p \end{pmatrix}$ with respect to a `good basis' $\nu_i\otimes t^{-1}e_1$, $i=1,2$ as chosen in Section~\ref{Colemanmf}. But $a_p\neq p+p^{k-2}$ by the Weil bound, so both $1-\varphi$ and $1-p\varphi$ are isomorphisms on $\DD_{\cris}(V_f(1))$ and $\Sigma=0$. Let $\bar{\nu}_1,\bar{\nu}_2$ be a `good basis' for $\Dcris(V_{\bar{f}})$. We can identify $\DD_{\cris}(V_f)$ with $\DD_{\cris}(V_f(1))$ (resp. $\DD_{\cris}(V_{\bar{f}})$ with $\Dcris(V_{\bar{f}}(k-1))$) via $\nu_i\mapsto \nu_i\otimes e_1t^{-1}$ (resp. $\bar{\nu}_i\mapsto \bar{\nu}_i\otimes e_{k-1}t^{1-k}$). Under these identifications, the natural pairing 
    \begin{equation}\label{crispairing}
     [\hspace{1ex},\hspace{1ex}]: \DD_{\cris}(V_f(1))\times \DD_{\cris}(V_{\bar{f}}(k-1))\rightarrow \DD_{\cris}(E(1))=E\cdot e_1t^{-1}
    \end{equation}
    satisfies $[\nu_1,\bar{\nu}_1]=0$. By applying $\vp$, we have $[\nu_2,\bar{\nu}_2]=0$, too. We also have $[\nu_1,\bar{\nu}_2]=-[\nu_2,\bar{\nu}_1]\ne0$. Without loss of generality, we may assume that $[\nu_1,\bar{\nu}_2]=-[\nu_2,\bar{\nu}_1]=1$
    
    Let $x\in\DD(V_{\bar{f}}(k-1))^{\psi=1}$. It follows from the construction of the Coleman maps in Section~\ref{Colemanmf} that if we let
    \begin{equation}\label{M}
     M=\begin{pmatrix}m_{11}&m_{12}\\m_{21}&m_{22}\end{pmatrix}=\left(\frac{t}{\pi q}\right)^{k-1}P^TM'^{-1},
    \end{equation}
    then, by (\ref{xDcris}), $(1-\vp)(x)$ can be written as
    \begin{equation}\label{formulaforphi}
     (1-\vp)(x)=(y_1m_{11}+y_2m_{21})\bar{\nu}_1\otimes t^{1-k}e_{k-1}+(y_1m_{12}+y_2m_{22})\bar{\nu}_2\otimes t^{1-k}e_{k-1},
    \end{equation}
    where $y_i=\rmCol_i(x)$ for $i=1,2$.

    \begin{proposition}\label{samemaps}
     On identifying with their images under $\mathfrak{M}$, we have 
     \begin{eqnarray}
      \langle\Omega_{V_f(1),1}((1+\pi)\otimes\nu_1),x\rangle_{V_f(1)}&=&y_1m_{12}+y_2m_{22},\label{eqnu1}\\
      -\langle\Omega_{V_f(1),1}((1+\pi)\otimes\nu_2),x\rangle_{V_f(1)}&=&y_1m_{11}+y_2m_{21}. \label{eqnu2}
     \end{eqnarray}
    \end{proposition}

   The rest of this section is devoted to proving this result. We follow closely Berger's proof of Perrin-Riou's explicit reciprocity law in \cite{berger03}. We first make the following definition: let $V\in \Rep_{\Qp}(G_{\QQ_p})$. For an element $x\in H^1_{\Iw}(\QQ_p,V)$, define $h^1_{\QQ_p,V}(x)$ to be the image of $x$ under the projection map $H^1_{\Iw}(\QQ_p,V)\rightarrow H^1(\QQ_p,V)$.

   Recall also the map $\partial_V$ defined in subsection \ref{crysrepsandwachmods}: for $z\in \Qpn((t))\otimes_{\QQ_p}\DD_{\cris}(V_f(1+j))$, we denote the constant coefficient of $z$ by $\partial_{V_f(1+j)}(z)\in\Qpn\otimes_{\QQ_p}\DD_{\cris}(V_f(1+j))$. 
   
   \begin{lemma}\label{omegaformula}
    Let $i\in\{1,2\}$, and choose $\mathfrak{y}_i\in\big(\BB^+_{\rig,\QQ_p}\otimes_{\QQ_p}\DD_{\cris}(V_f(1))\big)^{\psi=1}$ such that $(1-\varphi)\mathfrak{y}_i=(1+\pi)\otimes \nu_i$. Then
    \begin{multline*}
     h^1_{\QQ_p,V_f(1+j)}\Omega_{V_f(1+j),1+j}\big(\partial^{-j}(1+\pi)\otimes \nu_i\otimes t^{-j}e_j\big)=\\ j! \exp_{\QQ_p,V_f(1+j)}\Big(\big(1-\frac{\varphi^{-1}}{p}\big)\partial_{V_f(1+j)}\big(\partial^{-j}\mathfrak{y}_i\otimes t^{-j}e_j\big)\Big).
    \end{multline*}
   \end{lemma}
   \begin{proof}
    By Proposition~\ref{PRvalues}, we need to prove that $\partial_{V_f(1+j)}\big(\partial^{-j}\mathfrak{y}_i\otimes t^{-j}e_j\big)=(1-\vp)^{-1}(\nu_i\otimes t^{-j}e_j)$. 
    Note that $\varphi$ commutes with $\partial_{V_f(1+j)}$ and $\varphi\circ\partial^{-j}=p^j\partial^{-j}\circ\varphi$, so 
    \[ (1-\vp)\partial_{V_f(1+j)}\big(\partial^{-j}\mathfrak{y}_i\otimes t^{-j}e_j\big)=\partial_{V_f(1+j)}\big(\partial^{-j}(1+\pi)\otimes \nu_i\otimes t^{-j}e_j\big).\]
    Note that $\partial(1+\pi)=1+\pi$, so $\partial_{V_f(1+j)}\big(\partial^{-j}(1+\pi)\otimes \nu_i\otimes t^{-j}e_j\big)=\nu_i\otimes t^{-j}e_j$. Also, as observed above, $1-\vp$ is invertible on $\DD_{\cris}(V_f(1))$, which proves the result.     
   \end{proof}
   
   We can now prove Proposition~\ref{samemaps}. We will only prove~\eqref{eqnu2} here; the proof of~\eqref{eqnu1} is analogous. 
   
   \begin{proof}
    For $i=1,2$, let $\mathfrak{y}_i\in\big(\BB^+_{\rig,\QQ_p}\otimes_{\QQ_p}\DD_{\cris}(V_f(1))\big)^{\psi=1}$ such that $(1-\varphi)\mathfrak{y}_i=(1+\pi)\otimes \nu_i$. By $p$-adic interpolation it is sufficient to show that 
   \[ \partial^j\big(\big\langle \Omega_{V_f(1),1}((1+\pi)\otimes \nu_2),h^1_{\QQ_p,\Iw}(x)\big\rangle_{V_f(1)}\big)(0)=\partial^j\big(y_1m_{12}+y_2m_{22}\big)(0) \]
   for all $j\gg 0$. We have
   \begin{align}
    & \partial^j\big(\big\langle \Omega_{V_f(1),1}((1+\pi)\otimes \nu_2),x\big\rangle_{V_f(1)}\big)  = \big\langle \Omega_{V_f(1),1}\big((1+\pi)\otimes \nu_2\big)\otimes e_j,h^1_{\Iw,V_{\bar{f}}(k-1-j)}(x\otimes e_{-j})\big\rangle_{V_f(1+j)} \notag \\
    & = (-1)^j\big\langle \Omega_{V_f(1+j),1+j}\big(\partial^{-j}(1+\pi)\otimes \nu_2\otimes t^{-j}e_j\big),h^1_{\Iw,V_{\bar{f}}(k-1-j)}(x\otimes e_{-j})\big\rangle_{V_f(1+j)} \label{line1} 
   \end{align}
   by Lemma~\ref{twist} and the properties of $\Omega$ (c.f. p. 119, Th\'eor\`eme (B)(ii) in~\cite{perrinriou94}). Hence
   \begin{align}
    & \partial^j\big(\big\langle \Omega_{V_f(1),1}((1+\pi)\otimes \nu_2),x\big\rangle_{V_f(1)}\big)(0) \notag \\
    & = (-1)^j\big\langle h^1_{\Qp,V_f(1+j)}\Omega_{V_f(1+j),1+j}\big(\partial^{-j}(1+\pi)\otimes \nu_2\otimes t^{-j}e_j\big),h^1_{\Qp,V_{\bar{f}}(k-1-j)}(x\otimes e_{-j})\big\rangle_{V_f(1+j)} \label{line2}\\
    & = j!\Big\langle\exp_{\Qp,V_f(1+j)}\Big(\Big(1-\frac{\varphi^{-1}}{p}\Big)\partial_{V_f(1+j)}\big(\partial^{-j}\mathfrak{y}_2\otimes t^{-j}e_j\big)\Big),h^1_{\Qp,V_{\bar{f}}(k-1-j)}(x\otimes e_{-j})\Big\rangle_{V_f(1+j)} \label{line3} \\
    & = j! \Big[ \Big(1-\frac{\varphi^{-1}}{p}\Big)\partial_{V_f(1+j)}\big(\partial^{-j}\mathfrak{y}_2\otimes t^{-j}e_j\big),\exp^*_{\Qp,V^*(1+j)}h^1_{\Qp,V_{\bar{f}}(k-1-j)}(x\otimes e_{-j})\Big]_{V_f(1+j)} \label{line4} \\
    & = j! \Big[ \Big( 1-\frac{\varphi^{-1}}{p}\Big)\partial_{V_f(1+j)}\big(\partial^{-j}\mathfrak{y}_2\otimes t^{-j}e_j\big),\Big(1-\frac{\varphi^{-1}}{p}\Big)\partial_{V_{\bar{f}}(k-1-j)}(x\otimes e_{-j})\Big]_{V_f(1+j)} \label{line5} \\
    & = j! \big[\partial_{V_f(1+j)}\big(\partial^{-j}(1+\pi)\otimes\nu_2\otimes t^{-j}e_j\big), \partial_{V_{\bar{f}}(k-1-j)}((1-\varphi)x\otimes e_{-j})\big]_{V_f(1+j)} \label{line6}
   \end{align}
   The equalities can be explained as follows:
   \begin{itemize}
    \item  the first equality is immediate from~\eqref{line1} and the construction of $\langle\hspace{1ex},\hspace{1ex}\rangle_V$;
    \item  the implication \eqref{line2}$\Rightarrow$~\eqref{line3} follows from Lemma~\ref{omegaformula};
    \item the implication \eqref{line3}$\Rightarrow$~\eqref{line4} is the duality between $\exp_{F,V_f(1+j)}$ and $\exp^*_{F,V_f(1+j)}$;
    \item the implication \eqref{line4}$\Rightarrow$~\eqref{line5} follows from~\cite[Theorem II.6]{berger03}, and
    \item \eqref{line6} follows from ~\eqref{line5} since $1-\varphi$ is the adjoint of $1-\frac{\varphi^{-1}}{p}$ under the pairing $[\hspace{1ex},\hspace{1ex}]$.
   \end{itemize}
   Now $\partial(1+\pi)=1+\pi$, which implies that $\partial^{-j}(1+\pi)\otimes\nu_2\otimes t^{-j}e_j=(1+\pi)\otimes\nu_2\otimes t^{-j}e_j$ and hence
   \[ \partial_{V_f(1+j)}\big(\partial^{-j}(1+\pi)\otimes\nu_2\otimes t^{-j}e_j\big)=\nu_2\otimes t^{-j}e_j.\]
   By~\eqref{formulaforphi}, we can write
   \[ (1-\vp)x=(y_1m_{11}+y_2m_{21})\bar{\nu}_1+(y_1m_{12}+y_2m_{22})\bar{\nu}_2.\]
   Recall that by construction, we have $[\nu_2,\bar{\nu}_1]=-1$ and $[\nu_i,\bar{\nu}_i]=0$ for $i=1,2$. It follows that if we write $-(y_1m_{11}+y_2m_{21})=\Sigma_{i\geq 0}c_it^i$ with $c_i\in \QQ_p$, then \eqref{line6} is equal to $j!c_j$. Since also $-\partial^j\big(y_1m_{11}+y_2m_{21})\big)(0)=j!c_j$, this finishes the proof of~\eqref{eqnu2}. 
  \end{proof}
  
  We can summarize the results of this section by the following corollary:
  
  \begin{corollary}\label{commutativediagram}
   We have a commutative diagram
   \[
    \xymatrix{
    \NN(V)^{\psi=1} \ar[r]^{h^1_{\QQ_p,\Iw}} \ar[d]_{1-\varphi}        & H^1_{\Iw}(\QQ_p,V) \ar[d]_{(\uCol_1,\uCol_2)} \ar@/^1.3cm/[ddd]^{\mathcal{L}_{1,\bar{\nu}_i\otimes(1+\pi)}}\\
    (\vp^*\NN(V))^{\psi=0} \ar[r]^{\J} \ar[d]_{M} & \Lambda_{\QQ_p}(G_\infty)^{\oplus 2} \ar[d]_{\underline{M}} \\
    \big((\BB^+_{\rig,\QQ_p})^{\psi=0}\big)^{\oplus 2} \ar[r]^{\mathfrak{M}^{-1}} \ar[d]_{\pr_i} & \calH(G_\infty)^{\oplus 2} \ar[d]_{\upr_i} \\
    (\BB^+_{\rig,\QQ_p})^{\psi=0}  \ar[r]^{\mathfrak{M}^{-1}}         & \calH(G_\infty).
    }
   \]
   Here, $\pr_i$ and $\underline{\pr}_i$ denote the projection maps onto the respective $i$-th components, and for an element $x\in(\vp^*\NN(V))^{\psi=0}$, $M.x$ is defined as follows: if $x=x_1\vp(\pi^{1-k}n_1\otimes e_{k-1})+x_2\vp(\pi^{1-k}n_2\otimes e_{k-1})$ with $x_1,x_2\in(\BB^+_{\QQ_p})^{\psi=0}$, then $M.x=M\begin{pmatrix} x_1 \\ x_2\end{pmatrix}$. 
  \end{corollary}

    
   \subsection{Bounded \texorpdfstring{$p$-adic $L$-functions}{p-adic L-functions}}

     We now establish some basic properties of $L_{p,i}$ and $\tilde{L}_{p,i}$.

    \subsubsection{Decomposition of \texorpdfstring{$p$-adic $L$-functions}{p-adic L-functions}}\label{decomposition}

   Recall that $\alpha$ and $\beta$ are the roots of the quadratic $X^2-a_pX+p^{k-1}$. By \cite[Theorem 16.6]{kato04}, there exist eigenvectors $\eta_{\alpha}$ and $\eta_{\beta}$ of $\vp$ in $E(\alpha)\otimes_E\Dcris(V_f)$ with eigenvalues $\alpha$ and $\beta$ respectively such that $[\eta_\alpha,\bar{\nu}_1]=[\eta_\beta,\bar{\nu}_1]=1$ and we have
    \begin{eqnarray*}
     \big\langle\Omega_{V_f(1),1}((1+\pi)\otimes\eta_{\alpha}), \kato \big\rangle_{V_f(1)}&=&\tilde{L}_{p,\alpha},\\
     \big\langle\Omega_{V_f(1),1}((1+\pi)\otimes\eta_{\beta}), \kato \big\rangle_{V_f(1)}&=&\tilde{L}_{p,\beta}.
    \end{eqnarray*}
    It can be verified that
    \begin{align*}
     \eta_{\alpha}&=\alpha^{-1}\nu_1-\nu_2,\\
     \eta_{\beta}&=\beta^{-1}\nu_1-\nu_2.
    \end{align*}
    Therefore, by Definition~\ref{padicl} and Proposition~\ref{samemaps}, we have
    \begin{eqnarray*}
     \mathfrak{M}(\tilde{L}_{p,\alpha})&=&(\alpha^{-1} m_{12}+m_{11})L_{p,1}+(\alpha^{-1}m_{22}+m_{21})L_{p,2},\\
     \mathfrak{M}(\tilde{L}_{p,\beta})&=&(\beta^{-1}m_{12}+m_{11})L_{p,1}+(\beta^{-1}m_{22}+ m_{21})L_{p,2};
    \end{eqnarray*}
    in the notation of Section~\ref{mainresults}, we have
    \[ \mathcal{M}= \begin{pmatrix} \alpha^{-1} m_{12}+m_{11} & \alpha^{-1}m_{22}+m_{21} \\ \beta^{-1}m_{12}+m_{11} & \beta^{-1}m_{22}+ m_{21}\end{pmatrix}.\]
    The functions $L_{p,1}$ and $L_{p,2}$ can therefore be written as
    \begin{eqnarray}
     L_{p,1}&=&\frac{(\beta^{-1}m_{22}+ m_{21})\mathfrak{M}(\tilde{L}_{p,\alpha})-(\alpha^{-1}m_{22}+ m_{21})\mathfrak{M}(\tilde{L}_{p,\beta})}{(\beta^{-1}-\alpha^{-1})\det(M)}\label{l1}\\
     L_{p,2}&=&\frac{(\beta^{-1}m_{12}+ m_{11})\mathfrak{M}(\tilde{L}_{p,\alpha})-(\alpha^{-1}m_{12}+ m_{11})\mathfrak{M}(\tilde{L}_{p,\beta})}{(\alpha^{-1}-\beta^{-1})\det(M)}\label{l2}
    \end{eqnarray}

    Let $x\in\DD(V_{\bar{f}}(k-1))^{\psi=1}$. On the one hand, 
    \[
    (1-\vp)x=\uCol_1(x)\cdot[(1+\pi)\vp(\pi^{1-k} n_1 \otimes e_{k-1})]+\uCol_2(x)\cdot[(1+\pi)\vp(\pi^{1-k} n_2 \otimes e_{k-1})].
    \]
    On the other hand, Proposition~\ref{samemaps} says that
    \[
    (1-\vp)x=\mathfrak{M}\circ\LL_{1,\nu_1\otimes(1+\pi)}(x)\bar{\nu}_2\otimes t^{1-k}e_{k-1}-\mathfrak{M}\circ\LL_{1,\nu_2\otimes(1+\pi)}(x)\bar{\nu}_1\otimes t^{1-k}e_{k-1}.
    \]
    Therefore, we have
    \[
    \begin{pmatrix}\uCol_1(x)&\uCol_2(x)\end{pmatrix}\cdot\left[(1+\pi)M\right]=\begin{pmatrix}-\mathfrak{M}\circ\LL_{1,\nu_2\otimes(1+\pi)}(x)&\mathfrak{M}\circ\LL_{1,\nu_1\otimes(1+\pi)}(x)\end{pmatrix}.
    \]
    Let $\underline{M}=(\underline{m}_{ij})=\mathfrak{M}^{-1}[(1+\pi)M]$, then as elements of $\HG$
    \begin{equation}\label{prpairing}
     \begin{pmatrix}\uCol_1(x)&\uCol_2(x)\end{pmatrix}\underline{M}=\begin{pmatrix}-\LL_{1,\nu_2\otimes(1+\pi)}(x)&\LL_{1,\nu_1\otimes(1+\pi)}(x)\end{pmatrix}.
    \end{equation}

    Therefore, by exactly the same calculation as above, we have the following theorem:
    
    \begin{theorem}\label{Ldecomposition}
     Define     
     \[ \underline{\mathcal{M}}= \begin{pmatrix} \alpha^{-1} \underline{m}_{12}+\underline{m}_{11} & \alpha^{-1}\underline{m}_{22}+\underline{m}_{21} \\ \beta^{-1}\underline{m}_{12}+\underline{m}_{11} & \beta^{-1}\underline{m}_{22}+ \underline{m}_{21}\end{pmatrix}.\]
     Then we have the decomposition
     \begin{equation}\label{supersingulardecomposition}
      \begin{pmatrix}\tilde{L}_{p,\alpha}\\ \tilde{L}_{p,\beta} \end{pmatrix}=\underline{\mathcal{M}}\begin{pmatrix}\tilde{L}_{p,1}\\ \tilde{L}_{p,2} \end{pmatrix}
     \end{equation}   
    \end{theorem}
    
    Again, in the notation of Section~\ref{mainresults}, we have
    \[ \underline{\mathcal{M}}= \begin{pmatrix} \alpha^{-1} \underline{m}_{12}+\underline{m}_{11} & \alpha^{-1}\underline{m}_{22}+\underline{m}_{21} \\ \beta^{-1}\underline{m}_{12}+\underline{m}_{11} & \beta^{-1}\underline{m}_{22}+ \underline{m}_{21}\end{pmatrix}.\]


  \subsubsection{Interpolating properties}
     We calculate the values of our new $p$-adic $L$-functions at characters modulo $p$. We first state a lemma concerning such characters. 
    
    \begin{lemma}\label{divide}
     If $A\in(\BB_{\rig,\Qp}^+)^{\psi=0}$ is divisible by $\vp(\pi)$, then $\mathfrak{M}^{-1}(A)$ is zero when evaluated at any character with conductor $p$.
    \end{lemma}
    \begin{proof}
     This is a special case of Theorem~\ref{david} as proved below.
    \end{proof}

    \begin{notation}
     For any element $x\in\CC_p\otimes(\Brig)^{\psi=0}$ and $\eta$ a character on $G_\infty$, we write $\eta(x)$ for $\eta(\mathfrak{M}^{-1}(x))$.
    \end{notation}

    \begin{proposition}\label{nonzero}
     Let $\eta$ be a primitive character modulo $p$, then
     \begin{eqnarray*}
      \eta(L_{p,1})&=&\frac{\tau(\eta)}{p^{k-1}}\cdot\frac{L(f_{\eta^{-1}},1)}{\Omega_f^{\eta(-1)}},\\
      \eta(L_{p,2})&=&0.
     \end{eqnarray*}
     Similarly, if $\eta$ is the trivial character, then
     \begin{eqnarray*}
      \eta(L_{p,1})&=&\frac{a_p-p^{k-2}-1}{p^{k-1}}\cdot\frac{L(f,1)}{\Omega_f^{+}},\\
      \eta(L_{p,2})&=&\left(\frac{1}{p}-1\right)\cdot\frac{L(f,1)}{\Omega_f^{+}}.
     \end{eqnarray*}
    \end{proposition}
    \begin{proof}
     Since 
     \[ M=(t/\pi q)^{k-1}P^TM'^{-1}=(t/\pi q)^{k-1}\vp(M'^{-1})A_\vp^T  \]
     and $M'|_{\pi=0} = I$, we have $M|_{\pi=(\zeta-1)}=A_\vp^T$ for any $p$-th root of unity $\zeta$. In other words, we have $M\equiv A_\vp^T\mod\vp(\pi)$. Therefore, \eqref{l1} and \eqref{l2} imply that,
     \begin{eqnarray}
     L_{p,1}&\equiv&\frac{(\beta^{-1}a_p-1)\mathfrak{M}(\tilde{L}_{p,\alpha})-(\alpha^{-1}a_p-1)\mathfrak{M}(\tilde{L}_{p,\beta})}{(\beta^{-1}-\alpha^{-1})p^{k-1}}\mod\vp(\pi)\\
     L_{p,2}&\equiv&\frac{\beta^{-1}\mathfrak{M}(\tilde{L}_{p,\alpha})-\alpha^{-1}\mathfrak{M}(\tilde{L}_{p,\beta})}{(\alpha^{-1}-\beta^{-1})}\mod\vp(\pi)
    \end{eqnarray}
    Therefore, we are done by Lemma~\ref{divide} and the values of $\eta(\tilde{L}_{p,\alpha})$ and $\eta(\tilde{L}_{p,\beta})$ given in \cite{MTT} for example.
   \end{proof}

    \begin{corollary}\label{notzero}
     If $k\ge3$, then $L_{p,i}\ne0$ for $i\in\{1,2\}$. Moreover, if $\eta$ is a character of $\Delta$, then $L_{p,1}^\eta\ne0$.
    \end{corollary}
    \begin{proof}
     Since $k\ge3$, the result follows from the fact that $L(f_{\eta^{-1}},1)\ne0$ (by \cite[Proposition~2]{shimura76}). 
    \end{proof}
    
    \begin{remark}
     If $k=2$ and $a_p=0$, we will show that under the Mellin transform, $L_{p,1}$ and $L_{p,2}$ agree with Pollack's plus and minus $p$-adic $L$-functions up to a unit. Therefore, by \cite[Corollary~5.11]{pollack03}, it is in fact enough to assume that assumption (A) holds in order for Corollary~\ref{notzero} to hold.
    \end{remark}

    \begin{remark}
    We see that the interpolating properties of $L_{p,1}$ and $L_{p,2}$ at a character modulo $p$ are independent of the choice of $n_1,n_2$ as long as we have fixed a pair of `good bases' for $\Dcris(V_f)$ and $\Dcris(V_{\bar{f}})$.
    \end{remark}

   \begin{lemma}\label{cong}
    If $z\in\Lambda_E(G_\infty)$ and $f\in(E\otimes\BB_{\Qp})^{\psi=0}$, then $z\cdot(fn_i)\equiv(z\cdot f)n_i\mod\vp(\pi)$ for $i=1,2$.
   \end{lemma}
   \begin{proof}
    It follows from the fact that if $g\in G_\infty$, $g(\vp\NN(V))\subset\vp(\pi)\NN(V)$ for any $V$.
   \end{proof}

   \begin{corollary}\label{sameprop}
    Proposition~\ref{nonzero} (and hence Corollary~\ref{notzero}) still hold after replacing $L_{p,i}$ by $\tilde{L}_{p,i}$ for $i=1,2$.
   \end{corollary} 
   \begin{proof}
    By definitions, we have
    \[ (1-\vp)(\kato)=\begin{pmatrix}L_{p,1}&L_{p,2}\end{pmatrix}M\begin{pmatrix}\bar{\nu}_1\\\bar{\nu}_2\end{pmatrix}\otimes t^{1-k}e_{k-1}=\begin{pmatrix}\tilde{L}_{p,1}&\tilde{L}_{p,2}\end{pmatrix}\cdot\begin{pmatrix}(1+\pi)n_1\\ (1+\pi)n_2\end{pmatrix} \]
    where $\kato$ is the localization of the Kato zeta element and $M$ is as defined in \eqref{M}. This implies
    \[ \begin{pmatrix}L_{p,1}&L_{p,2}\end{pmatrix}\begin{pmatrix}n_1\\ n_2\end{pmatrix}=\begin{pmatrix}\tilde{L}_{p,1}&\tilde{L}_{p,2}\end{pmatrix}\cdot \begin{pmatrix}(1+\pi)n_1\\ (1+\pi)n_2\end{pmatrix}.\]
    Therefore, by Lemma~\ref{cong}, we have $L_{p,i}\equiv\mathfrak{M}\big(\tilde{L}_{p,i}\big)\mod\vp(\pi)$ and hence $\mathfrak{M}^{-1}(L_{p,i})$ agrees with $\tilde{L}_{p,i}$ at a character modulo $p$ by Lemma~\ref{divide}.
   \end{proof}


  \subsubsection{Infinitude of zeros}
   Let $\eta$ be a character of $\Delta$. Mazur proved that at least one of $\tilde{L}_{p,\alpha}$ and $\tilde{L}_{p,\beta}$ has infinitely many zeros if $v_p(\alpha)\ne v_p(\beta)$. This has been generalized to the case $a_p=0$ (\cite[Theorem~3.5]{pollack03}). We show that our decomposition of $\tilde{L}_{p,\alpha}$ and $\tilde{L}_{p,\beta}$ can be used to give an alternative proof to Mazur's result as well as generalize Pollack's result to the case $a_p\ne0$.

   \begin{proposition}
    If $f$ is a modular form as given in the beginning of Section~\ref{Colemanmf} and $\eta$ a character of $\Delta$, then either $\tilde{L}_{p,\alpha}^{\eta}$ or $\tilde{L}_{p,\beta}^\eta$ has infinitely many zeros.
   \end{proposition}
   \begin{proof}
    Assume not, then \cite[Lemma~3.2]{pollack03} implies that $\tilde{L}_{p,\alpha}^\eta$ and $\tilde{L}_{p,\beta}^\eta$ are $O(1)$.

    By \cite[Lemmas~3.3.5 and 3.3.6]{bergerbreuil10}, the entries of $M$ are $O(\log_p^m)$ where $m=\max\{v_p(\alpha),v_p(\beta)\}<k-1$. Therefore, with the notation above, $m_{ij}=O(\log_p^m)$ for $i,j\in\{1,2\}$. In particular, the $\eta$-component of
    \[(\beta^{-1}m_{22}+ m_{21})\tilde{L}_{p,\alpha}-(\alpha^{-1}m_{22}+ m_{21})\tilde{L}_{p,\beta}\]
    is $O(\log_p)^{m}$. By \eqref{l1}, the quantity above is divisible by $(t/\pi q)^{k-1}\sim\log_p^{k-1}$ which forces $L_{p,1}^\eta=0$ contradicting Corollary~\ref{nonzero}.
   \end{proof}

   As with \cite[Theorem~3.5]{pollack03}, we have:
   
   \begin{corollary}
    If $\alpha\notin E(\eta)$, then both $\tilde{L}_{p,\alpha}^\eta$ and $\tilde{L}_{p,\beta}^\eta$ have infinitely many zeros.
   \end{corollary}


  \subsection{Good ordinary modular forms}\label{ordinary}

   We now assume that $f$ is good ordinary at $p$. We will pick different bases from the supersingular case to define our Coleman maps. Let ${\alpha}$ be the root of $X^2-{a}_pX+p^{k-1}$ which is a $p$-adic unit and ${\beta}$ is the one with $p$-adic valuation $k-1$. By a result of Deligne and Mazur-Wiles (see for example \cite[Section 17]{kato04} for an exposition), there exists a 1-dimensional $G_{\Qp}$-subrepresentation $V_{\bar{f}}'$ in $V_{\bar{f}}$. Moreover, $V_{\bar{f}}'$ has Hodge-Tate weight 0 and $\Dcris(V_{\bar{f}}')$ can be identified with the ${\alpha}$-eigenspace of $\vp$ in $\Dcris(V_{\bar{f}})$. We fix a nonzero element $\bar{\nu}_1\in\Dcris(V_{\bar{f}}')$. Then, by \eqref{elementary}, $n_1 = \bar{\nu}_1$ is a basis of $\NN(V_{\bar{f}}')$ over $E\otimes\BB_{\Qp}^+$. Let $\bar{\nu}_2$ be a nonzero $\beta$-eigenvector of $\vp$ in $\Dcris(V_{\bar{f}})$. 

   \begin{proposition}
    We may find $n_2 \in \NN(V_{\bar{f}})$ lifting $\bar{\nu}_2$ such that $n_1, n_2$ is an $E\otimes\BB_{\Qp}^+$-basis of $\NN(V_{\bar{f}})$, and $(1 + \pi)\vp(\pi^{1-k} n_1 \otimes e_{k-1}), (1 + \pi)\vp(\pi^{1-k} n_2 \otimes e_{k-1})$ is a $\Lambda_E(G_\infty)$-basis of $(\vp^* \NN(V_{\bar{f}}(k-1)))^{\psi = 0}$. 
   \end{proposition}

   \begin{proof} 
    Let $N = \NN(V_{\bar{f}})$ and $N' = \NN(V_{\bar{f}}')$. Then the quotient $N'' = N / N'$ may be identified with the Wach module of the quotient $V_{\bar{f}} / V_{\bar{f}}'$, and we have an exact sequence
    \[ 0 \rTo (\vp^* N'(k-1))^{\psi = 0} \rTo (\vp^* N(k-1))^{\psi = 0} \rTo (\vp^* N''(k-1))^{\psi = 0} \rTo 0.\]
    It is clear that $(1 + \pi)\vp(n_1 \otimes \pi^{1-k} e_{k-1})$ is a basis of $(\vp^* N'(k-1))^{\psi = 0}$, and the result now follows on applying theorem \ref{Lambda} to $N''$. 
   \end{proof}

   Hence the change of basis matrix $M'$, with
   \[ \begin{pmatrix}\bar{\nu}_1\\ \bar{\nu}_2\end{pmatrix}=M'\begin{pmatrix}n_1\\ n_2\end{pmatrix}, \]
   is lower triangular, with $1,(t/\pi)^{k-1}$ on the diagonal. With respect to this basis, the Coleman maps given in Section~\ref{positive} enable us to define:

   \begin{definition}
    For $i=1,2$, define $L_{p,i}\in (E\otimes\BB_{\Qp}^+)^{\psi=0}$ to be the image of the localization of the Kato zeta element (on using the identification as given by \eqref{fontaineisom}) under $\rmCol_i$. Similarly, define $\tilde{L}_{p,i}$ to be the image of the localization of the Kato zeta element under $\uCol_i$.
   \end{definition}

   Since $\vp(n_1)={\alpha}n_1$, the matrix $P$ as defined in Section \ref{positive} is upper triangular and there exists a unit $u$ in $E\otimes\BB_{\Qp}^+$ such that
   \[
   P=\begin{pmatrix}\alpha&*\\ 0& uq^{k-1} \end{pmatrix}.
   \]
   Therefore, \eqref{xoverDcris} becomes
   \begin{equation}
    (1-\vp)(x)=\begin{pmatrix}\rmCol_1(x)&\rmCol_2(x)\end{pmatrix}\begin{pmatrix}{\alpha}(\frac{t}{\pi q})^{k-1}&0\\ * &u\end{pmatrix}\begin{pmatrix}\bar{\nu}_1\\\bar{\nu}_2\end{pmatrix}\otimes t^{1-k}e_{k-1}.\label{colandphi}
   \end{equation}
   
   \begin{lemma}
    Let $\nu_1$, $\nu_2$ be a basis of $\Dcris(V_f)$ such that $\vp(\nu_1)=\alpha\nu_1$ and $\vp(\nu_2)=\beta\nu_2$. Then 
    \[
    [\nu_i\otimes t^{-1}e_1,\bar{\nu}_i\otimes t^{1-k}e_{k-1}]=0
    \]
    for $i=1,2$ where $[\hspace{1ex},\hspace{1ex}]$ is the pairing defined in \eqref{crispairing}.
   \end{lemma}
   \begin{proof}
    Assume $m_1:=[\nu_1\otimes t^{-1}e_1,\bar{\nu}_1\otimes t^{1-k}e_{k-1}]\ne0$. Since $[\hspace{1ex},\hspace{1ex}]$ is compatible with $\vp$, we have
    \begin{eqnarray*}
     \vp[\nu_1\otimes t^{-1}e_1,\bar{\nu}_1\otimes t^{1-k}e_{k-1}]&=&[\vp(\nu_1\otimes t^{-1}e_1),\vp(\bar{\nu}_1\otimes t^{1-k}e_{k-1})]\\
     p^{-1}m_1&=&[\alpha p^{-1}\nu_1\otimes t^{-1}e_1,{\alpha}p^{1-k}\bar{\nu}_1\otimes t^{1-k}e_{k-1}]\\
     p^{k-1}m_1&=&\alpha^2m_1.
    \end{eqnarray*}
    Hence, $\alpha^2=p^{k-1}$, which is a contradiction. The proof for $i=2$ is similar.
   \end{proof}

   As in Section \ref{construction}, we may assume that $[\nu_1,\bar{\nu}_2]=-[\nu_2,\bar{\nu}_1]=1$ and an analogue of Proposition \ref{samemaps} says that
     \[
   \mathfrak{M}\begin{pmatrix}-\LL_{\nu_2\otimes(1+\pi)}\circ h^1_{\Iw}&\LL_{\nu_1\otimes(1+\pi)}\circ h^1_{\Iw}\end{pmatrix}=\begin{pmatrix}\rmCol_1&\rmCol_2\end{pmatrix}\begin{pmatrix}{\alpha}(\frac{t}{\pi q})^{k-1}&0\\ * &u\end{pmatrix}.
   \]
   In particular, if we apply this to the Kato zeta element, we have
   \[
   \begin{pmatrix}-\mathfrak{M}(\tilde{L}_{p,\beta})&\mathfrak{M}(\tilde{L}_{p,\alpha})\end{pmatrix}=\begin{pmatrix}L_{p,1}&L_{p,2}\end{pmatrix}\begin{pmatrix}{\alpha}(\frac{t}{\pi q})^{k-1}&0\\ * &u\end{pmatrix}
   \]
   where $\tilde{L}_{p,\beta}=\LL_{\nu_2}(\kato)$. Similarly, we have
 \begin{equation}\label{orddecomp}
   \begin{pmatrix}-\tilde{L}_{p,\beta}&\tilde{L}_{p,\alpha}\end{pmatrix}=\begin{pmatrix}\tilde{L}_{p,1}&\tilde{L}_{p,2}\end{pmatrix}\begin{pmatrix}{\alpha}\log_{p,k}&0\\ * &\tilde{u}\end{pmatrix}
   \end{equation}
where $\log_{p,k}=\prod_{j=0}^{k-2}\log_p(\chi(\gamma)^{-j}\gamma)/(\chi(\gamma)^{-j}\gamma-1)$ and $\tilde{u}\in\Lambda_E(G_\infty)^\times$.

 Therefore, as in Section~\ref{construction}, we can decompose $\tilde{L}_{p,\beta}$ into a linear combination of $\tilde{L}_{p,1}$ and $\tilde{L}_{p,2}$, whereas $\tilde{L}_{p,\alpha}=\tilde{L}_{p,2}$, up to a unit. We now say something about $\tilde{L}_{p,1}$. When $V_f$ is not locally split at $p$, $\tilde{L}_{p,\beta}$ is conjecturally equal to the critical slope $p$-adic $L$-function constructed in \cite{pollackstevens09}. We itemize this condition since we will need it again later.

\begin{itemize}
    \item \textbf{Assumption (A'):} $V_f$ is not locally split at $p$ and $k\ge3$.
   \end{itemize}

   In this case, \cite[Theorem 16.4 and 16.6]{kato04} imply that $\tilde{L}_{p,\beta}$ has the same interpolating properties as $\tilde{L}_{p,\alpha}$, namely:
   \begin{equation}
    \chi^r\eta(\tilde{L}_{p,\alpha})=\frac{c_{\eta,r}}{\beta^n}L(f_{\eta^{-1}},r+1)\qquad\text{and}\qquad\chi^r\eta(\tilde{L}_{p,\beta})=\frac{c_{\eta,r}}{\beta^n}L(f_{\eta^{-1}},r+1)
    \label{interpolate}
   \end{equation}
   where $\eta$ is a finite character of conduction $p^n>1$, $0\le r\le k-2$ and $c_{\eta,r}$ is some constant independent of $\alpha$ and $\beta$. Note that the values given by \eqref{interpolate} do not determine $\tilde{L}_{p,\beta}$ uniquely. However, they allow us to show that $\tilde{L}_{p,1},L_{p,1}\ne0$.

   \begin{proposition}
    If assumption (A') holds, then $\tilde{L}^\eta_{p,1},L_{p,1}^\eta\ne0$ for any character $\eta$ on $\Delta$.
   \end{proposition}
   \begin{proof}
    As in the proof of Proposition \ref{nonzero}, $M'|_{\pi=0}$ implies that $M|_{\pi=(\zeta-1)}=A_{\vp}^T$ for any $\zeta^p=1$, where $A_\vp=\begin{pmatrix}{\alpha}&0\\0&\beta\end{pmatrix}$ is the matrix of $\vp$ with respect to $\bar{\nu}_1$, $\bar{\nu}_2$. Therefore, $\mathfrak{M}(\tilde{L}_{p,\beta})(\zeta-1)={\alpha}L_{p,1}(\zeta-1)$. Since $V_f$ is not locally split and $k\ge3$, by the above discussion, $\eta(\tilde{L}_{p,\beta})=\frac{\tau(\eta)}{\beta}L(f_{\eta^{-1}},1)\ne0$ as in the supersingular case. Therefore, $L_{p,1}(\zeta-1)\ne0$. The statement about $\tilde{L}_{p,1}$ then follows as in Corollary~\ref{sameprop}. 
   \end{proof}
   
   We see from the proof that the interpolating properties of $\mathfrak{M}^{-1}(L_{p,1})$ and $\tilde{L}_{p,1}$ at characters modulo $p$ are the same as that of $\tilde{L}_{p,\beta}$ after multiplying a constant.

   \begin{remark}
    If $V_f$ does split locally at $p$, we can choose $n_2=\bar{\nu}_2$ and both $P$ and $M'$ would be diagonal. Therefore, we have $\tilde{L}_{p,\beta}=\mathfrak{M}^{-1}((t/\pi q)^{k-1}L_{p,1})=\log_{p,k}\tilde{L}_{p,1}$. But it is not known that whether $\tilde{L}_{p,\beta}$ is nonzero or not.
   \end{remark}


  \section{Coleman maps for the Berger--Li--Zhu basis}\label{ss}

   In this section, we will prove some results on the images of the Coleman maps under the assumption that $v_p(a_p)>\big\lfloor\frac{k-2}{p-1}\big\rfloor$, using the explicit basis of $\NN(V_{\bar f})$ written down in \cite{bergerlizhu04}. We shall also give an explicit proof that this particular basis satisfies the conclusions of theorem \ref{choice}.
  
     Write $m=\lfloor(k-2)/(p-1)\rfloor$ and define
     \[
      \lp=\prod_{n\ge0}\frac{\vp^{2n+1}(q)}{p} = \prod_{\substack{n \ge 1 \\ \text{$n$ even}}} \frac{\Phi_{n}(1 + \pi)}{p}
     \]
     and
     \[
      \qquad\text{and}\qquad\lm=\prod_{n\ge0}\frac{\vp^{2n}(q)}{p}= \prod_{\substack{n \ge 1 \\ \text{$n$ odd}}} \frac{\Phi_{n}(1 + \pi)}{p}.
     \]
     where $\Phi_n(X)$ is the $p^n$-th cyclotomic polynomial.
     Let $z_i$ be elements of $\Qp$ such that
     \[ p^{m}\left(\frac{\lm}{\lp}\right)^{k-1}=\sum_{i\ge0}z_i\pi^i, \]
     then as shown in~\cite[Proposition 3.1.1]{bergerlizhu04},
     \[
      z=\sum_{i=0}^{k-2}z_i\pi^i\in\Zp[[\pi]].
     \]
     By \cite{bergerlizhu04}, under assumption (C), i.e. $v_p(a_p)>m$, there is a lattice $T_{\bar{f}}$ in $V_{\bar{f}}$ and a basis of $\NN(T_{\bar{f}})$ such that the matrix of $\vp$ with respect to this basis, $P$, is given by
     \[
      \begin{pmatrix}0&-1\\ q^{k-1} & \delta z\end{pmatrix}
     \]
     where $\delta=a_p/p^m$. In particular, the reduction of this basis modulo $\pi$ is a ``good basis'' in the sense of \S \ref{Colemanmf}, and hence the Coleman maps may be defined integrally as in remark \ref{integrality}. By construction, for any $x\in\DD(T_{\bar{f}}(k-1))^{\psi=1}$ with 
     \[
     x=\pi^{1-k}\begin{pmatrix}x_1& x_2 \end{pmatrix}\begin{pmatrix}n_1 \\ n_2 \end{pmatrix}\otimes e_{k-1},
     \]
     we can express $\rmCol_i(x)$, $i=1,2$, in terms of $x_1$ and $x_2$:
     \begin{eqnarray}\label{explicit1}
      \rmCol_1(x)&=&x_2-\vp(x_1)+\delta zx_1,\\
      \rmCol_2(x)&=&-q^{k-1}x_1-\vp(x_2).\label{explicit2}
     \end{eqnarray}

     \begin{remark}
      The representation constructed in \cite{bergerlizhu04} is really $V_{\bar{f}}$ twisted by an unramified character. But since we assume that $\epsilon(p)=1$, it does not affect the action of $P$ and our calculations later on.
     \end{remark}
   
   
   \subsection{The image of \texorpdfstring{$\rmCol_1$}{Col1}}\label{Col1}
    
    We first give a few preliminary lemmas.
    
    \begin{lemma}\label{limit}
     For all $n\ge0$, we have
     $\vp^n(M'^{-1})(A_\vp^T)^n=\vp^{n-1}(P^T)\cdots\vp(P^T)P^TM'^{-1}.$
     Moreover, as $n\rightarrow\infty$, the quantity above tends to $0$.
    \end{lemma}
    \begin{proof}
     The equality follows from (\ref{relatingphi}) and induction. For the limit, note that $M'|_{\pi=0}=I$, hence $\vp^n(M')\rightarrow I$ as  $n\rightarrow\infty$. The eigenvalues of $A_\vp$ are $\alpha$ and $\beta$. But $\alpha^n$, $\beta^n\rightarrow0$ as $n\rightarrow\infty$, so we are done.
    \end{proof}

    \begin{lemma}\label{psi}
     Let $x=\pi^{1-k}\begin{pmatrix}x_1&x_2\end{pmatrix}\begin{pmatrix}n_1\\ n_2\end{pmatrix}\otimes e_{k-1}$. Then, $\psi(x)$ is given by
     \[
     \begin{pmatrix}\psi(x_1\delta z+x_2)&-\psi(q^{k-1}x_1)\end{pmatrix}\pi^{1-k}\begin{pmatrix}n_1\\ n_2\end{pmatrix}
     \]
    \end{lemma}
    \begin{proof}
     Recall that $\vp(\pi)=\pi q$, we have
     \begin{eqnarray*}
      x&=&\pi^{1-k}\begin{pmatrix}x_1&x_2\end{pmatrix}(P^T)^{-1}\begin{pmatrix}\vp(n_1)\\ \vp(n_2)\end{pmatrix}\\
       &=&\begin{pmatrix}x_1\delta z+x_2&-q^{k-1}x_1\end{pmatrix}\vp(\pi)^{1-k}\begin{pmatrix}\vp(n_1)\\ \vp(n_2)\end{pmatrix},
     \end{eqnarray*}
     hence the result
    \end{proof}

   \begin{lemma}\label{constantterm}
    For all $n\geq1 $, the constant term of $\psi(q^n)$ is $p^{n-1}$.
   \end{lemma}
   \begin{proof}
    Induction.
   \end{proof}
   
   \begin{lemma}\label{lincombination}
    If $f(\pi)\in E\otimes\BB_{\Qp}^+$, then there exist unique $a_i\in E$ for $1\leq i\leq k-1$ such that $f(\pi)=\sum_{i=1}^{k-1}a_i(\pi+1)^i\mod\pi^{k-1}$.
   \end{lemma}
   \begin{proof}
    Note that 
    \begin{equation}\label{expansion}
     (\pi+1)^k =\tbinom{k}{1}(\pi+1)^{k-1}-\dots+ (-1)^{k-2}\tbinom{k}{k-1}(\pi+1)+(-1)^{k-1}\mod \pi^k.    
    \end{equation}
    Suppose now that there exist $a_1,\dots,a_{k-1}\in E$ such that $(\pi+1)^k =\sum_{i=1}^{k-1}a_i(\pi+1)^i \mod \pi^k$. Subtracting this sum from~\eqref{expansion} shows that 
    \[ \big(\tbinom{k}{1}-a_{k-1}\big)(\pi+1)^{k-1}+\dots+ \big((-1)^{k-2}\tbinom{k}{k-1}-a_1\big)(\pi+1)+(-1)^{k-1} =0.\]
    But this gives a contradiction since $\{(\pi+1)^i\}_{0\leq i<k}$ is a basis of the vector space of polynomials of degree $\leq k-1$.
   \end{proof}

    \begin{proposition}\label{col1}
     Under assumption (C), the map $(\pi^{k-1}\calO_E\otimes\AA_{\Qp}^+)^{\psi=0}\subset\rmCol_1\big(\DD(T_{\bar{f}}(k-1))^{\psi=1}\big)$.
    \end{proposition}
    \begin{proof}
     Recall that (\ref{1-vp}) says
     \[
     (1-\vp)x=\begin{pmatrix}y_1&y_2\end{pmatrix}\cdot(\pi q)^{1-k}P^T\begin{pmatrix}n_1\\ n_2\end{pmatrix}\otimes e_{k-1}.
     \]
     For any $y_1\in(\pi^{k-1}\calO_E\otimes\AA_{\Qp}^+)^{\psi=0}$, we have 
     \[
     y:=\begin{pmatrix}y_1&0\end{pmatrix}\cdot(\pi q)^{1-k}P^T\begin{pmatrix}n_1\\ n_2\end{pmatrix}\otimes e_{k-1}=\begin{pmatrix}0&y_1/\pi^{k-1}\end{pmatrix}\begin{pmatrix}n_1\\ n_2\end{pmatrix}\otimes e_{k-1}.
     \]
     Then,
     \begin{eqnarray*}
      \vp^n(y)&=&\begin{pmatrix}0&\vp^n(y_1/\pi^{k-1})\end{pmatrix}\vp^{n-1}(P^T)\cdots\vp(P^T)P^T\begin{pmatrix}n_1\\ n_2\end{pmatrix}\otimes e_{k-1}\\
      &=&\begin{pmatrix}0&\vp^n(y_1/\pi^{k-1})\end{pmatrix}\vp^n(M'^{-1})(A_\vp^T)^nM'\begin{pmatrix}n_1\\ n_2\end{pmatrix}\otimes e_{k-1}.
     \end{eqnarray*}
     Hence, Lemma~\ref{limit} implies that $\vp^n(y)\rightarrow0$ as $n\rightarrow\infty$ and the series $x:=\sum_{n\ge0}\vp^{n}(y)$ converges to an element of $\DD(T_{\bar{f}}(k-1))^{\psi=1}$ with $(1-\vp)x=y$. Therefore, $y_1=\rmCol_1(x)$.
     \end{proof}
     
     \begin{proposition}\label{surjectiveCol1}
      Under assumptions (B), (C) and (D), the map $\rmCol_1:\DD(V_{\bar{f}}(k-1))\rightarrow(E\otimes\BB_{\Qp}^+)^{\psi=0}$ is surjective.
     \end{proposition}
     \begin{proof}
     By Proposition~\ref{col1}, if $y_1\in(\pi^{k-1}E\otimes\BB_{\Qp}^+)^{\psi=0}$, then $y_1\in\image(\rmCol_1)$. For an arbitrary $y_1\in(E\otimes\BB_{\Qp}^+)^{\psi=0}$, by Lemma~\ref{lincombination} there exists $y'$ in the $E$-linear span of $\{(1+\pi)^i\}_{1\leq i<k}$ such that $y_1+\vp(y')$ is divisible by $\pi^{k-1}$. Hence, by the same argument as above, the sum
     \[
      \sum_{n\ge0}\vp^n\left(\begin{pmatrix}0&(y_1+\vp(y'))/\pi^{k-1}\end{pmatrix}\begin{pmatrix}n_1\\ n_2\end{pmatrix}\right)
     \]
     converges to an element $x\in\NN(V_{\bar{f}}(k-1))$. By Lemma~\ref{psi} and the fact that $\psi(y_1)=0$, we have
     \begin{eqnarray*}
      \psi(x)-x & = & \psi\left(\begin{pmatrix}0&(y_1+\vp(y'))/\pi^{k-1}\end{pmatrix}\begin{pmatrix}n_1\\ n_2\end{pmatrix}\right)\\
                & = & \pi^{1-k}\begin{pmatrix}y'&0\end{pmatrix}\begin{pmatrix}n_1\\ n_2\end{pmatrix}
     \end{eqnarray*}
     Let $x'=x+\pi^{1-k}\begin{pmatrix}x_1&x_2\end{pmatrix}\begin{pmatrix}n_1\\ n_2\end{pmatrix}$. Then
     \[
      \psi(x')-x'=\pi^{1-k}\begin{pmatrix}y'-x_1+\psi(x_1\delta z+x_2)&-x_2-\psi(q^{k-1}x_1)\end{pmatrix}\begin{pmatrix}n_1\\ n_2\end{pmatrix}.
     \]
     Hence, $x'\in\DD(V_{\bar{f}}(k-1))^{\psi=1}$ if and only if
     \begin{equation} 
     \begin{aligned}
      x_2 &= -\psi(q^{k-1}x_1)\\
      y'&= x_1-\psi(x_1\delta z)+\psi^2(q^{k-1}x_1)
     \end{aligned}\label{findingx}
     \end{equation}
     Assume that such $x_1$ exists in the $E$-linear span of $\{(1+\pi)^i\}_{1\leq i<k}$, and let $a$ be its degree in $\pi$. Since the degrees of $\delta z$ and $q^{k-1}$ are at most $k-2$ and $(p-1)(k-1)$ respectively, the degrees of $\psi(x_1\delta z)$ and $\psi^2(q^{k-1}x_1)$ are at most $(k-2+a)/p$ and $((p-1)(k-1)+a)/p^2$ respectively. But we assume that $p\ge k-1$, so the right-hand side of \eqref{findingx} has degree $\leq a$. Since $y'$ has degree at most $k-1$ and $x_1$ is in the $E$-linear span of $\{(1+\pi)^i\}_{1\leq i<k}$, both $\psi(x_1\delta z)$ and $\psi^2(q^{k-1}x_1)$ are scalar multiples of $(1+\pi)$. We write
     \begin{displaymath}
      y'=\sum_{i=1}^{k-1}\alpha_i(1+\pi)^i,\qquad x_1=\sum_{i=1}^{k-1}\beta_i(1+\pi)^i\qquad\text{and}\qquad\delta z=\sum_{i=0}^{k-2}\gamma_i(1+\pi)^i
     \end{displaymath}
     where $\alpha_i$, $\beta_i$, $\gamma_i\in E$. Then (\ref{findingx}) says that
     \begin{align*}
      \alpha_i &= \beta_i\qquad\text{for }i\ge2 \\
      \alpha_1 &= \beta_1-\sum_{i+j=p}\beta_i\gamma_{j}+\beta_{p^2-(k-1)(p-1)}
     \end{align*}
     where $\gamma_i=\beta_i=0$ if $i<0$. But $p^2-(k-1)(p-1)>1$ and $\gamma_{p-1}=0$, the matrix relating $(\alpha_i)_{1\le i\le k-1}$ and $(\beta_i)_{1\le i\le k-1}$ is upper triangular with non-zero entries on the diagonal. Therefore, there is a bijection between $(\alpha_i)_{1\le i\le k-1}\in E^{k-1}$ and $(\beta_i)_{1\le i\le k-1}\in E^{k-1}$. In other words, given any $y'$ as above, there exists a unique $x_1$ (and hence $x_2$) such that $x'\in\DD(V_{\bar{f}}(k-1))^{\psi=1}$. For any $0\le j\le k-2$, we can therefore choose $y$ (and hence $y'$) such that $x_1\equiv \pi^j\mod \pi^{j+1}$. In this case,
     \begin{align*}
      \rmCol_1(x') &= y_1+\vp(y')-\psi(q^{k-1}x_1)-\vp(x_1)+x_1\delta z\\
      &\equiv -\psi(q^{k-1}x_1)-\vp(x_1)+x_1\delta z\mod \pi^{k-1}\\
      &\equiv (-p^{k-2-j}-p^j+a_p)\pi^j\mod \pi^{j+1},
     \end{align*}
     where we deduce the last line from the previous one using Lemma~\ref{constantterm} and the observation that $\pi q=\varphi(\pi)$. Therefore, our assumption on $a_p$ implies that for all $y_1\in(E\otimes\BB_{\Qp}^+)^{\psi=0}$, there exists some $x\in\DD(V_{\bar{f}}(k-1))^{\psi=1}$ such that $\rmCol^+(x)\equiv y_1\mod \pi^{j+1}$ by induction. Hence we are done.
    \end{proof}
 
    \begin{corollary}\label{pseudo}
     Under assumptions (B), (C) and (D), the image of $\rmCol_1:\DD(T_{\bar{f}}(k-1))^{\psi=1}\rightarrow(\EA)^{\psi=0}$ is pseudo isomorphic to $(\EA)^{\psi=0}$.
    \end{corollary}
    \begin{proof}
     It suffices to show that the said image has finite index in $(\EA)^{\psi=0}$. The proof of Proposition~\ref{col1} shows that $(\pi^{k-1}\EA)^{\psi=0}$ lies in the image and for all $0\le j\le k-2$, there exists $x_j\in\DD(T_{\bar{f}}(k-1))^{\psi=1}$ such that $\rmCol_1(x_j)\equiv\alpha_j\pi^j\mod \pi^{j+1}$ for some $\alpha_j\ne0$. Therefore, the quotient lies inside $\prod_{j=0}^{k-2}\calO_E/\alpha_i\calO_E$, so we are done.
    \end{proof}
   

   \subsection{The image of \texorpdfstring{$\rmCol_2$}{Col2}}\label{Col2}

    We now describe the image of $\rmCol_2$. We will show that it is generated by two elements.
    
    \begin{lemma}\label{gamma}
     Let $x=\pi^{1-k}\begin{pmatrix}x_1&x_2\end{pmatrix}\begin{pmatrix}n_1\\ n_2\end{pmatrix}\otimes e_{k-1}$ and $\gamma$ a topological generator of $\Gamma$, then
     \[
     \gamma(x)=\pi^{1-k}\begin{pmatrix}\gamma(x_1)&\gamma(x_2)\end{pmatrix}G_\gamma\begin{pmatrix}n_1\\ n_2\end{pmatrix}\otimes e_{k-1}
     \]
     for some $G_\gamma\in I+\pi M(2,\Zp[[\pi]])$.
    \end{lemma}
    \begin{proof}
     By \cite[Proposition 3.1.3]{bergerlizhu04}, there exists $G_{\gamma}\in I+\pi M_2(\Zp[[\pi]])$ such that $\begin{pmatrix}\gamma(n_1)\\ \gamma(n_2)\end{pmatrix}=G_{\gamma}^T\begin{pmatrix}n_1\\ n_2\end{pmatrix}$. Therefore,
     \begin{align*}
      \gamma(x)&=\gamma(\pi)^{1-k}\begin{pmatrix}\gamma(x_1)&\gamma(x_2)\end{pmatrix}G_\gamma^T\begin{pmatrix}n_1\\ n_2\end{pmatrix}\otimes\chi(\gamma)^{k-1} e_{k-1}\\
      &=\left(\frac{(1+\pi)^{\chi(\gamma)}-1}{\chi(\gamma)}\right)^{1-k}\begin{pmatrix}\gamma(x_1)&\gamma(x_2)\end{pmatrix}G_\gamma^T\begin{pmatrix}n_1\\ n_2\end{pmatrix}\otimes e_{k-1}.
     \end{align*}
     But $\chi(\gamma)\in 1+p\Zp$, which implies $((1+\pi)^{\chi(\gamma)}-1)/\chi(\gamma)\in\pi(1+p\Zp[[\pi]])$. Hence the result.
    \end{proof}

    \begin{lemma}\label{firstcoefficient2}
     Let $x=\pi^{1-k}\begin{pmatrix}x_1&x_2\end{pmatrix}\begin{pmatrix}n_1\\ n_2\end{pmatrix}\otimes e_{k-1}\in\DD(V_{\bar{f}}(k-1))^{\psi=1}$. Write $x_i=\sum_{j\geq 0}a_{i,j}\pi^j$. Then $x_1$ has order $<k-1$ if and only if $x_2$ has order $<k-1$. If this is the case, they have the same order which we denote by $d_x$. Moreover, $a_{2,d_x}=-p^{k-2-d_x}a_{1,d_x}$.
    \end{lemma}
    \begin{proof}
     Since $x\in\DD(V_{\bar{f}}(k-1))^{\psi=1}$, we have $x_2=-\psi(q^{k-1}x_1)$, hence the result by Lemma~\ref{constantterm}.
    \end{proof}

    \begin{proposition}\label{col2}
     Under assumptions (C) and (D), the image of $\rmCol_2:\DD(V_{\bar{f}}(k-1))\rightarrow (E\otimes\BB_{\Qp}^+)^{\psi=0}$ contains $ (\vp(\pi)^{k-1}E\otimes\BB_{\Qp}^+)^{\psi=0}$ and the quotient of the containment is a cyclic $\Lambda_E(\Gamma)$-module under the action of $\Gamma$ described in Lemma~\ref{gamma}.
    \end{proposition}
    \begin{proof}
     For any $y_2\in(\vp(\pi)^{k-1}E\otimes\BB_{\Qp}^+)^{\psi=0}$, we have 
     \[
     y:=\begin{pmatrix}0&y_2\end{pmatrix}\cdot(\pi q)^{1-k}P^T\begin{pmatrix}n_1\\ n_2\end{pmatrix}\otimes e_{k-1}=\vp(\pi)^{1-k}\begin{pmatrix}-y_2&y_2\delta z\end{pmatrix}\begin{pmatrix}n_1\\ n_2\end{pmatrix}\otimes e_{k-1}.
     \]
     Hence, as in the proof of Proposition \ref{col1}, $\sum_{n\ge0}\vp^n(y)$ converges which implies that $y_2$ lies in the image of $\rmCol_2$.

     Recall that if $x=\pi^{1-k}\begin{pmatrix}x_1&x_2\end{pmatrix}\begin{pmatrix}n_1\\ n_2\end{pmatrix}\otimes e_{k-1}$, then $-\rmCol_2(x)=q^{k-1}x_1+\psi(x_2)$. For $i=1,2$, write $x_i=\sum_{j\ge0}a_{i,j}\pi^j$ and
     \begin{eqnarray*}
      \Cbar(x)&=&q^{k-1}x_1-\vp(x_2)\mod\vp(\pi)^{k-1}\\
              &=&(q^{k-1}a_{1,0}+a_{2,0})+\vp(\pi)(q^{k-2}a_{1,1}+a_{2,1})+\cdots\vp(\pi)^{k-2}(qa_{1,k-2}+a_{2,k-2})\mod\vp(\pi)^{k-1}.
     \end{eqnarray*}
     We now construct a generator $f$ for $\Cbar(\DD(V)^{\psi=1})$ over $\Lambda_E(\Gamma)$ inductively. By the proof of Proposition~\ref{col1}, there exists $x_i\in\DD(V)^{\psi=1}$ of order $i$ for all $0\leq i<k-1$. Let $f_0=x_0$. For $i\ge0$, suppose that we have constructed $f_{i}$. Write
     \begin{align*}
      f_i' &=\prod_{j=0}^i\big(\gamma-\chi(\gamma)^j\big)(f_i)\\
           &=\pi^{1-k}\begin{pmatrix}f'_{i,1}&f'_{i,2}\end{pmatrix}\begin{pmatrix}n_1\\ n_2\end{pmatrix}\otimes e_{k-1}
     \end{align*}
     then it follows from Lemma \ref{gamma} that $f_i'$ is 
     of order $\geq i+1$. Let $\alpha_{i+1,1}$ and $\alpha_{i+1,2}$ be the coefficients of $\pi^{i+1}$ in the power series expansions of $f'_{i,1}$ and $f'_{i,2}$, respectively. There are two possibilities: either both $\alpha_{i+1,j}$ are non-zero, in which case we let $f_{i+1}=f_i$. Or both of them are zero, in which case we let $f_{i+1}=f_i+x_{i+1}$. 
   
     Let $f=f_{k-2}$. Then for all $0\leq i<k-1$, the order of $\prod_{j=0}^i\big(\gamma-\chi(\gamma)^j\big)(f)$ is $i$. To finish the proof, it is now sufficient to observe that by Lemma \ref{firstcoefficient2}, for all $x\in\DD(V_{\bar{f}}(k-1))^{\psi=1}$ there exist scalars $\alpha_i\in E$ for $0\leq i<k-1$ such that $x-\sum_{i=1}^{k-2}\alpha_i\prod_{j=1}^i(\gamma-\chi(\gamma)^j)f$ is of order $\geq k-1$. 
    \end{proof}

      
  \subsection{The Iwasawa transform}\label{Iwasawatransform}
  
  \begin{convention}
   For the rest of this section as well as in Sections~\ref{algorithm} and~\ref{ucol1}, we assume without loss of generality that $\chi(\gamma)=1+p$.
  \end{convention}

   \begin{lemma}\label{gammafraction*}  
    \begin{equation}\label{firstcong} 
     \frac{q}{\gamma(q)}= 1 \mod (p\pi,\pi^{p-1}).
    \end{equation}
   \end{lemma}
   \begin{proof}
    We have $q=\frac{\vp(\pi)}{\pi}$, and $\gamma(1+\pi)=(1+\pi)(1+\vp(\pi))$. Hence
    \begin{equation*}
     \frac{q}{\gamma(q)} = \frac{1+q+\vp(\pi)}{ 1+\vp(q)+\vp^2(\pi)} 
    \end{equation*}
    It remains to notice that $q=\pi^{p-1}\mod p$. Moreover, the constant term of $q$ (and hence of $\vp(q)$) is $p$, and $\sum_{j=0}^{+\infty}(-p)^j$ is the multiplicative inverse of $1+p$, which implies the result. 
   \end{proof}

   \begin{corollary}\label{firstmatrix*} 
    Both $\frac{\log^+}{\gamma(\log^+)}$ and $\frac{\log^-}{\gamma(\log^-)}$ are congruent to $1\mod (p\pi,\pi^{p-1})$ (and hence in particular congruent to $1\mod (p\pi,\pi^2)$ since we assume $p\geq 3$).
   \end{corollary}
   \begin{proof}
    Clear from Lemma~\ref{gammafraction*} and the definition of $\log^\pm$. 
   \end{proof}
  
   Define
   \[ G_\gamma^{(k-1)}=\begin{pmatrix} \left(\frac{\log^+}{\gamma(\log^+)}\right)^{k-1} & 0 \\ 0 & \left(\frac{\log^-}{\gamma(\log^-)}\right)^{k-1}\end{pmatrix}.\] 
  
   \begin{lemma} \label{matrixofgamma}
    $G_\gamma^{(k-1)}\simeq \begin{pmatrix} 1 & 0 \\ 0 & 1 \end{pmatrix}\mod (p\pi,\pi^2)$.
   \end{lemma}
   \begin{proof}
    Immediate from the definition and Corollary~\ref{firstmatrix*}.
   \end{proof}

   Let $\varpi_E$ be a uniformizer of $E$. 
    
   \begin{proposition}\label{gammamatrix} 
    $G_\gamma\simeq \begin{pmatrix} 1 & 0 \\ 0 & 1 \end{pmatrix}\mod (\varpi_E\pi,\pi^2)$.
   \end{proposition}
   \begin{proof}
    We first review the construction of $G_\gamma$ as in \cite[\S 3.1]{bergerlizhu04}. For $l\ge k$, we define recursively
    \[
     G_\gamma^{(l)}=G_\gamma^{(l-1)}+\pi^{l-1}H^{(l)}
    \]
    for some $H^{(l)}\in M(2,\Zp[[X]])$ where $X=a_p/p^m$ and $m=\lfloor\frac{k-2}{p-1}\rfloor$. Note that $X\in \mathfrak{m}_E$ by assumption (C). The matrix $G_\gamma$ is then given by the limit of $G_\gamma^{(l)}$ as $l\rightarrow\infty$. Therefore, when $k>2$, the result is immediate from Lemma~\ref{matrixofgamma}.
  
    When $k=2$, it suffices to show that $H^{(2)}\equiv0\mod \varpi_E$. By construction (see \cite[Lemma~3.1.2 and Proposition~3.1.3]{bergerlizhu04}), $H^{(2)}$ satisfies the following:
    \begin{equation}\label{H2}
      H^{(2)}- P_0 H^{(2)} \big(pP_0\big)^{-1}= -R^{(1)}\mod \pi
    \end{equation}
    for some matrix $R^{(1)}\in X M(2,\Zp[[\pi,X]])$ and $P_0=\begin{pmatrix} 0 & 1 \\ p & a_p \end{pmatrix}$. If we write $H=\begin{pmatrix} h_{11} & h_{12}\\ h_{21} & h_{22} \end{pmatrix}$, then \eqref{H2} says that
    \[ \begin{pmatrix} h_{11} & h_{12}+h_{21} \\ h_{21} & h_{22} \end{pmatrix} \equiv0\mod X,\]
    and hence we are done since $\varpi_E|X$.
   \end{proof}
 
   Let $n_i'=\vp(n_i\otimes \pi^{1-k}e_{k-1})$ for $i=1,2$. Let $T=T_{\bar{f}}(k-1)$ and $V=V_{\bar{f}}(k-1)$. (In fact, the proof works for $T=T_{\bar{f}}(m)$ for any integer $m$.) Recall that $\chi(\gamma)=1+p$.
 
   \begin{proposition}\label{gmatrix*} 
    We have $\gamma\big[(1+\pi)n_i'\big]=(1+\vp(\pi))(1+\pi)n_i'\mod (\varpi_E\vp(\pi),\vp(\pi)^2)$ for $i=1,2$.
   \end{proposition}
   \begin{proof}
    We know that $ \begin{pmatrix} n_1' \\ n_2' \end{pmatrix}= P^T\begin{pmatrix} n_1 \\ n_2 \end{pmatrix}\otimes \vp(\pi)^{1-k}e_{k-1}$. Since the actions of $\gamma$ and $\vp$ commute, we have $\gamma(P^T)G_{\gamma}^T=\varphi(G_\gamma^T)P^T$, which implies
    \[
     \begin{pmatrix} \gamma n_1' \\ \gamma n_2' \end{pmatrix}=
    \chi(\gamma)^{k-1}\varphi\left(\frac{\pi}{\gamma(\pi)}\right)^{k-1}\varphi(G_\gamma^T) \begin{pmatrix} n_1' \\ n_2' \end{pmatrix}.
    \]
    Now 
    \begin{align}
     \chi(\gamma)\frac{\pi}{\gamma(\pi)} & = \frac{\chi(\gamma)}{1+q+\vp(\pi)} \notag \\
     & \equiv 1\mod (p\pi,\pi^2) \label{crudecongruence}
    \end{align}
    where the congruence comes from the fact that the constant term of $q$ is $p$, and hence the constant term of $\frac{1}{1+q+\vp(\pi)}$ is $\sum_{j=0}^{+\infty} (-p)^j$, which is equal to $\chi(\gamma)^{-1}$. Hence $\varphi\left(\frac{\chi(\gamma)\pi}{\gamma(\pi)}\right)\equiv 1\mod (p\vp(\pi),\varphi(\pi)^2)$. Moreover, $\gamma(1+\pi)=(1+\pi)^{\chi(\gamma)}=(1+\pi)(1+\varphi(\pi))$. Hence
    \begin{equation*} 
     \gamma\big[(1+\pi)n_1'\big] \equiv(1+\vp(\pi))(1+\pi)n_1'\mod (\varpi_E\vp(\pi),\vp(\pi)^2) 
    \end{equation*}
    by Corollary~\ref{gammamatrix}
   \end{proof}
 
   We will now show that we can adapt the arguments form Proposition~\ref{gmatrix*} to pass from $\vp(\pi)(1+\pi)n_i'$ to $\vp(\pi)^2(1+\pi)n_i'\mod (\varpi_E\vp(\pi),\vp(\pi)^3)$ for $i=1,2$.
 
   \begin{lemma}\label{gmatrixgeneral1*} 
    We have $(\gamma-1)\big[\vp(\pi)(1+\pi)n_i'\big]=\vp(\pi)^2(1+\pi)n_i'\mod (\varpi_E\vp(\pi),\vp(\pi)^3)$ for $i=1,2$.
   \end{lemma}
   \begin{proof}
    We have $\gamma(\pi)=(1+\pi)(1+\vp(\pi))-1 = \pi+\vp(\pi)+\pi\vp(\pi)$, so
    \begin{align*} 
     \vp( \gamma(\pi)) & = \vp(\pi)(1+\vp(q)+\vp(\pi)\vp(q)) \\
     & \equiv\vp(\pi)\mod (\varpi_E\vp(\pi),\vp(\pi)^3).
    \end{align*}
    Therefore,
    \begin{align*} 
     \gamma\big[\vp(\pi)(1+\pi)n_1'\big] & \equiv  \vp(\pi)(1+\vp(q)+\vp(\pi)\vp(q))(1+\pi)(1+\vp(\pi))n_1'\mod (\varpi_E\vp(\pi)^2,\vp(\pi)^3)\\
      & \equiv \vp(\pi)(1+\pi)(1+\vp(\pi))n_1'\mod (\varpi_E\vp(\pi),\vp(\pi)^3)
    \end{align*}
    and hence
    \begin{equation}
     (\gamma-1)\big[\vp(\pi)(1+\pi)n_1'\big]\equiv \vp(\pi)^2(1+\pi)n_1'\mod (\varpi_E\vp(\pi),\vp(\pi)^3).
    \end{equation}
   \end{proof}
 
   The lemma generalizes as follows for arbitrary $r\geq 1$.
 
   \begin{proposition}\label{gmatrixgeneral*} 
    We have $(\gamma-1)\big[\vp(\pi)^r(1+\pi)n_i'\big]\equiv\vp(\pi)^{r+1}(1+\pi)n_i'\mod (\varpi_E\vp(\pi)^r,\vp(\pi)^{r+2})$ for $i=1,2$.
   \end{proposition}
   \begin{proof}
    Be the same calculations as in Lemma~\ref{gmatrixgeneral1*}, we have
    \begin{align*} 
     \gamma\big[\vp(\pi)^r(1+\pi)n_1'\big] & \equiv  \vp(\pi)^r(1+\vp(q)+\vp(\pi)\vp(q))^r(1+\pi)(1+\vp(\pi))n_1'\mod (\varpi_E\vp(\pi)^{r+1},\vp(\pi)^{r+2})\\
      & \equiv \vp(\pi)^r(1+\pi)(1+\vp(\pi))n_1'\mod (\varpi_E\vp(\pi)^r,\vp(\pi)^{r+2})
    \end{align*}  
   \end{proof}
 
   \begin{definition} 
    For all $r\geq 2$, denote by $I_r$ the ideal of $\vp(\calO_E\otimes_{\ZZ_p}\AA^+_{\QQ_p})$ generated by the elements 
    \[ \varpi_E^{r-1}\vp(\pi),\varpi_E^{r-2}\vp(\pi)^2,\dots, \varpi_E\vp(\pi)^{r-1},\vp(\pi)^{r+1},\] 
    and let $\mathfrak{I}_r=I_r(\vp^*\NN(T))^{\psi=0}$.
   \end{definition}
 
   Note that $\mathfrak{I}_r$ is stable under the action of $G_\infty$. 
 
   \begin{lemma}\label{submodules} 
    We have $(\gamma-1)\mathfrak{I}_r\subset \mathfrak{I}_{r+1}$.
   \end{lemma}
   \begin{proof}
    It is enough to show that $(\gamma-1)\big[x\vp(\pi)^m(1+\pi)n_i'\big]\in\mathfrak{I}_{r+1}$ for any $m\ge0$, any $x\in I_r$ and $i=1,2$.
 
    Let $x=\varpi_E^{r-j}\vp(\pi)^j$ where $1\le j\le r-1$. By Proposition~\ref{gmatrixgeneral*}, we have 
    \begin{align*} 
     (\gamma-1)\big[\varpi_E^{r-j}\vp(\pi)^{m+j}(1+\pi)n_i'\big]\equiv & \varpi_E^{r-j}\vp(\pi)^{m+j+1}(1+\pi)n_i'\mod(\varpi_E^{r-j+1}\vp(\pi)^{m+j},\varpi_E^{r-j}\vp(\pi)^{m+j+2})\\
     \equiv&0\mod\mathfrak{I}_{r+1} 
    \end{align*}
    for all $m\geq 0$. Similarly, the same holds for $x=\vp(\pi)^{r+1}$. Hence the result. 
   \end{proof}

   \begin{proposition}\label{crucial}
    We have
    \[ (\gamma-1)^r\big[(1+\pi)n_i'\big]\equiv \vp(\pi)^r(1+\pi)n_i'\mod \mathfrak{I}_r\]
    for all $r\geq 2$.
   \end{proposition}
   \begin{proof}
    We proceed by induction on $r$. Let $r=2$. By Proposition~\ref{gmatrix*}, we have
    \[ (\gamma-1)\big[(1+\pi)n_i'\big]\equiv\vp(\pi)(1+\pi)n_i'\mod (\varpi_E\vp(\pi),\vp(\pi)^2).\]
    It therefore follows from Lemma~\ref{gmatrixgeneral1*} and Proposition~\ref{gmatrixgeneral*} that
    \[ (\gamma-1)^2\big[(1+\pi)n_i'\big]\equiv\vp(\pi)^2(1+\pi)n_i'\mod (\varpi_E\vp(\pi),\vp(\pi)^3).\]   
    Assume now that the result is true for $r-1\geq 2$, so 
    \[ (\gamma-1)^{r-1}\big[(1+\pi)n_i'\big]\equiv\vp(\pi)^{r-1}(1+\pi)n_i'\mod \mathfrak{I}_{r-1}.\]
    Now 
    \begin{align*} 
     (\gamma-1)\big[\vp(\pi)^{r-1}(1+\pi)n_i'\big]& \equiv \vp(\pi)^r(1+\pi)n_i'\mod (\varpi_E\vp(\pi)^{r-1},\vp(\pi)^{r+1}) \\
     & \equiv \vp(\pi)^r(1+\pi)n_i'\mod\mathfrak{I}_r
    \end{align*}
    by Proposition~\ref{gmatrixgeneral*}. The result therefore follows from Lemma~\ref{submodules}.  
   \end{proof}

   To simplify the notation, let $X=\vp(\calO_E\otimes_{\ZZ_p}\AA^+_{\QQ_p})(1+\pi)n_1'+ \vp(\calO_E\otimes_{\ZZ_p}\AA^+_{\QQ_p})(1+\pi)n_2'$.
 
   \begin{corollary}\label{basismodp} 
    For all $x\in X$, there exist $\omega_1,\omega_2\in \Lambda_{\calO_E}(\Gamma)$ such that 
    \[ \omega_1\big((1+\pi)n_1'\big)+\omega_2\big((1+\pi)n_2'\big)-x\in \varpi_E X.\]
   \end{corollary}
   \begin{proof}
    $(\vp^*\NN(T))^{\psi=0}$ is complete in the $(\varpi_E,\vp(\pi))$-adic topology, and the $ \mathfrak{I}_r$, ${r\geq 1}$ form a neighbourhood of zero in $(\vp^*\NN(T))^{\psi=0}$. Hence the result follows from Proposition~\ref{crucial}.
   \end{proof}

   Note that $(\vp^*\NN(T))^{\psi=0}$ is the $\Delta$-orbit of $X$. The previous corollary therefore implies the following result:

   \begin{theorem}\label{lambdabasis}
    $(\vp^*\NN(T))^{\psi=0}$ is a free $\Lambda_{\calO_K}(G_\infty)$-module of rank $2$, and a basis is given by $(1+\pi)n_1'$ and $(1+\pi)n_2'$.
   \end{theorem}
   \begin{proof}
    Let $y\in (\vp^*\NN(T))^{\psi=0}$. It follows from Corollary~\ref{basismodp} and the fact that $\Lambda_{\calO_E}(G_\infty)$ is $p$-adically complete that there exists $\omega_1,\omega_2\in\Lambda_E(G_\infty)$ such that $y=\omega_1\big((1+\pi)n_1'\big)+\omega_2\big((1+\pi)n_2'\big)$. As shown in~\cite{perrinriou94}, $\NN(T)^{\psi=1}$ is a free $\Lambda_E(G_\infty)$-module of rank $2$, and the map $1-\vp:\NN(T)^{\psi=1}\rightarrow (\vp^*\NN(T))^{\psi=0}$ is injective since $V^{H_{\QQ_p}}=\{0\}$. Hence the result.  
   \end{proof}

   It therefore follows that after tensoring with $\QQ$, there is an isomorphism of $\Lambda_E(G_\infty)$-modules (the \emph{Iwasawa transform})
   \[
   \J:(\vp^*\NN(V))^{\psi=0}\rTo \Lambda_{E}(G_\infty)^{\oplus 2}
   \]
   which satisfies the following condition: if $y=y_1(1+\pi)n_1'+y_2(1+\pi)n_2'$ with $y_i\in \vp(E\otimes_{\QQ_p}\BB^+_{\QQ_p})$ (write $y=(y_1,y_2)$) and $(z_1,z_2)=\J(y_1,y_2)$, then $y=z_1[(1+\pi)n_1']+z_2[(1+\pi)n_2']$. In particular, $\J$ is additive and linear over $E$.


   \subsection{An algorithm for \texorpdfstring{$\J$}{J}}\label{algorithm}

    We now summarize the results of the previous section to give an explicit description of $\J$ when restricted to $\vp(E\otimes_{\QQ_p}\BB^+_{\QQ_p})(1+\pi)n_1'\oplus \vp(E\otimes_{\QQ_p}\BB^+_{\QQ_p})(1+\pi)n_1'\cong \vp(E\otimes_{\QQ_p}\BB^+_{\QQ_p})^{\oplus 2}$. For a non-zero $y=(y_1,y_2)\in\vp(E\otimes_{\QQ_p}\BB_{\Qp}^+)^{\oplus 2}$, we write 
    \[
    y_1=\sum_{n=0}^\infty a_n\vp(\pi)^n\qquad\text{and}\qquad y_2=\sum_{n=0}^\infty b_n\vp(\pi)^n.
    \]
    On multiplying by a power of $\varpi_E$, we may assume that $y\in\vp(\calO_E\otimes_{\ZZ_p}\AA_{\Qp}^+)^{\oplus 2}$ but $\varpi_E\nmid y$. For such a $y$, we define the order $\ord(y)$ of $y$ to be the minimum integer $n$ such that either $a_n$ or $b_n$ is a unit in $\calO_E$.

    \begin{proposition}\label{algo}
     For $y$ as above, there exists $z^{(n)}\in(\gamma-1)^n\Lambda_{\calO_E}(\Gamma)^2$ such that $y-\J^{-1}(z^{(n)})$ has order strictly greater than $n$.
    \end{proposition}
    \begin{proof} 
     This is simply a reformulation of Proposition~\ref{crucial}. In particular, one could take 
     \[z^{(n)}=(a_n(\gamma-1)^n,b_n(\gamma-1)^n).\]
    \end{proof}

    \begin{corollary}\label{divisiblebyp}
     For $y$ as above, there exists a sequence $z^{(0)},z^{(1)},\ldots$ in $\Lambda_{\calO_E}(\Gamma)^{\oplus 2}$ such that $z^{(i)}\rightarrow0$ as $i\rightarrow\infty$ and
     \[
     y-\J^{-1}\left(\sum_{i=0}^\infty z^{(i)}\right)\in \varpi_E\vp(\calO_E\otimes_{\ZZ_p}\AA_{\Qp}^+)^2
     \]
    \end{corollary}

    We write $y^{(0)}=y$ and $u^{(0)}$ for the infinite sum given by Corollary~\ref{divisiblebyp}. Define a sequence $y^{(n)}$ recursively: for $n\ge0$, let $y^{(n+1)}=(y^{(n)}-\J^{-1}(u^{(n)}))/\varpi_E$ where $u^{(n)}$ to be the sum given by Corollary~\ref{divisiblebyp} on applying it to $y^{(n)}$. Then, we have
    \[
    \J(y)=\sum_{i=0}^\infty \varpi_E^iu^{(i)}.
    \]
 
 
 \subsection{The image of \texorpdfstring{$\uCol_1$}{Col1}}\label{ucol1}

  Throughout this section, we assume that assumptions (B), (C) and (D) are satisfied. 

  \begin{definition}
   Let $\uCol=\J\circ{\bfCol}: \NN(T)^{\psi=1}\rightarrow \Lambda_{\calO_E}(G_\infty)^{\oplus 2}$, and for $i=1,2$, define
   \[ \uCol_i:\NN(T)^{\psi=1}\rTo \Lambda_{\calO_E}(G_\infty)\]
   as the composition $\pr_i\circ \uCol$, where $\pr_i$ is the projection from $\Lambda_{\calO_E}(G_\infty)^{\oplus 2}$ onto the $i$-th coefficient.
  \end{definition}

  By abuse of notation, we also write $\uCol_i$ for the natural extension $\NN(V)^{\psi=1}\rightarrow \Lambda_E(G_\infty)$. The aim of this section is to prove the following theorem. 
 
  \begin{theorem}\label{image1}
   The map $\uCol_1:\NN(V)^{\psi=1}\rightarrow \Lambda_E(G_\infty)$ is surjective. 
  \end{theorem}
 
  The idea of the proof is to translate Proposition~\ref{surjectiveCol1} using the explicit description of $\J$ given in Section~\ref{algorithm}. Note that since $\J$ is a $\Lambda_E(G_\infty)$-homomorphism, it is sufficient to show that $\varpi_E^m\in \image(\uCol_1)$ for some $m\in\ZZ$.
 
  \begin{proposition}\label{approximation}
   Let $y_2\in \vp(\calO_E\otimes_{\ZZ_p}\AA^+_{\QQ_p})$. Then there exists a sequence $z^{(i)}\in \Lambda_{\calO_E}(\Gamma)$ tending to $0$ as $i\rightarrow +\infty$ and $\tilde{x}\in \NN(T)^{\psi=1}$ and $y_2'\in \vp(\calO_E\otimes_{\ZZ_p}\AA^+_{\QQ_p})$ such that 
   \[ \J(0,y_2)-\J\circ\rmCol(\tilde{x})=\sum_{i\geq 0}(0,z^{(i)})+ \varpi_E\J(0,y_2').\]
  \end{proposition}
  \begin{proof}
   If $(0,y_2)=+\infty$, then $\varpi_E|y_2$ and we are done. 
   Assume that $\ord(y_2)=n$ and write $y_2=\sum_{r\ge0}b_r\vp(\pi)^r$. Then, by Lemma~\ref{algo}, 
   \begin{equation}\label{firststep}
    \J(0,y_2)=\J(y_1^{(1)},y_2^{(1)})+(0,b_n(\gamma-1)^n).
   \end{equation}
   where $y_2^{(i)}$ has order strictly greater than $n$. By applying $\J^{-1}$ to~\eqref{firststep}, we see that
   \[
    y_2(1+\pi)n_2'=y_1^{(1)}(1+\pi)n_1'+y_2^{(1)}(1+\pi)n_2'+b_n(\gamma-1)^n[(1+\pi)n_2'].
   \]
   Since $G_\gamma$ is diagonal $\mod \pi^{k-1}$, this implies that $y_1^{(1)}\equiv0\mod\vp(\pi)^{k-1}$. In particular, the proof of Proposition~\ref{col1} implies that there exists $x_1\in\NN(T)^{\psi=1}$ such that $(1-\vp)x_1=y_1^{(1)}(1+\pi)n_1'$. Hence, we have
   \[
    \J(0,y_2)-\J\circ\rmCol(x_1)=\J(0,y_2^{(1)})+(0,z^{(1)})
   \]
   where $z^{(1)}=b_n(\gamma-1)^n$.
   
   On applying the above to $y_2^{(1)}$ and repeat, we obtain sequences $\{x_n\in\NN(T)^{\psi=1}\}$, $\{z^{(n)}\in\Lambda_{\calO_E}(\Gamma)\}$ and $\{y_2^{(n)}\in\vp(\AA_{\Qp}^+)\}$ such that
   \[
    \J(0,y_2^{(n-1)})-\J\circ\rmCol(x_n)=\J(0,y_2^{(n)})+(0,z^{(n)}),
   \]
   the sequence $m_n=\ord(y_2^{(n)})$ is strictly increasing, $z^{(n)}\in(\gamma-1)^{m_{n-1}}\Lambda_{\calO_E}(\Gamma)$ and $\rmCol(x_n)+(y_2^{(n)}-y_2^{(n-1)})(1+\pi)n_2'=z^{(n)}[(1+\pi)n_2']$. Now $\rmCol(x_n)\in \vp(\calO_E\otimes_{\ZZ_p}\AA^+_{\QQ_p})(1+\pi)n_1'$, so (i) $x_n\rightarrow0$ and (ii) $(y_2^{(n)}-y_2^{(n-1)})\rightarrow 0$. By completeness, (ii) implies that $y_2^{(n)}$ converges to an element in $\varpi_E\vp(\calO_E\otimes_{\ZZ_p}\AA_{\Qp}^+)$. (The limit must be in $\varpi_E\vp(\calO_E\otimes_{\ZZ_p}\AA_{\Qp}^+)$ because the order of the limit is $+\infty$ by construction.)  Now, on taking sums, we have for all $n\ge1$,
   \[
    \J(0,y_2)-\sum_{i=1}^n\J\circ\rmCol(x_i)=\J(0,y_2^{(n)})+\sum_{i=1}^n(0,z^{(i)}).
   \]
   We obtain the result by letting $n\rightarrow\infty$.
  \end{proof}
 
  \begin{corollary}
   Let $y_2\in \vp(\calO_K\otimes_{\ZZ_p}\AA^+_{\QQ_p})$. Then there exists a sequence $z^{(i)}\in \Lambda_{\calO_E}(\Gamma)$ tending to $0$ as $i\rightarrow +\infty$ and $\tilde{x}\in \NN(T)^{\psi=1}$ such that 
   \[ \J(0,y_2)-\J\circ\rmCol(\tilde{x})=\sum_{i\geq 0}(0,z^{(i)}).\]  
  \end{corollary}
  \begin{proof}
   Iterate the result in Proposition~\ref{approximation} for $\J(0,y_2')$ etc. and use that both $\Lambda_{\calO_E}(\Gamma)$ and $\vp^*(\NN(T))^{\psi=0}$ are $p$-adically complete. 
  \end{proof}
 
  \begin{corollary}\label{ant}
   Let $y\in (\vp^*\NN(T))^{\psi=0}$ be of the form $y=y_2n_2'$ for some $y_2\in (\calO_E\otimes_{\ZZ_p}\AA^+_{\QQ_p})^{\psi=0}$. Then there exists a sequence $z^{(i)}\in \Lambda_{\calO_E}(G_\infty)$ tending to $0$ as $i\rightarrow +\infty$ and $\tilde{x}\in \NN(T)^{\psi=1}$ such that 
   \[ \J(y)-\J\circ\rmCol(\tilde{x})=\sum_{i\geq 0}(0,z^{(i)}).\]
  \end{corollary}
  \begin{proof}
   Immediate from the previous corollary and the observation that $(\calO_E\otimes_{\ZZ_p}\AA^+_{\QQ_p})^{\psi=0}n_2'$ is the $\Delta$-orbit of $\vp(\calO_E\otimes_{\ZZ_p}\AA^+_{\QQ_p})(1+\pi)n_2'$.
  \end{proof}

  We can now prove Theorem~\ref{image1}. By Proposition~\ref{surjectiveCol1} there exists $x\in \NN(T)^{\psi=1}$ such that $\rmCol(x)=\varpi_E^m(1+\pi)n_1' + y_2n_2'$ for some $y_2\in (\calO_E\otimes_{\ZZ_p}\AA^+_{\QQ_p})^{\psi=0}$. It is clear that $\J(\varpi_E^m(1+\pi)n_1')=(\varpi_E^m,0)$. Also, we know by Corollary~\ref{ant} that there exists a sequence $z^{(i)}\in \Lambda_{\calO_E}(G_\infty)$ tending to $0$ as $i\rightarrow +\infty$ and $\tilde{x}\in \NN(T)^{\psi=1}$ such that 
  \[ \J(y_2n_2')-\J\circ\rmCol(\tilde{x})=\sum_{i\geq 0}(0,z^{(i)}).\] 
  Hence $\J\circ\rmCol(x)-\J\circ\rmCol(\tilde{x})=(\varpi_E^m,0)+\sum_{i\geq 0}(0,z^{(i)})$, i.e.
  \[ \J\circ\rmCol(x-\tilde{x})=(\varpi_E^m,0)+\sum_{i\geq 0}(0,z^{(i)}).\]
  
  \begin{remark}
   Alas so far we don't know how to translate Proposition~\ref{col2} into a statement about $\image{\uCol_2}$.
  \end{remark}
  
  \begin{remark}
   In a forthcoming paper~\cite{leiloefflerzerbes2}, we give a description of the images of the $\uCol_i$ using Perrin-Riou's $p$-adic regulator. 
  \end{remark}


 \section{Relations to existing work}

  \subsection{Fourier transforms}\label{appendix}
   In this section, we prove a compatibility result in $p$-adic Fourier theory (theorem \ref{david} below) which will allows us to relate divisibility of elements in $\HG$ and of their images in $(\Brig)^{\psi=0}$ under the Mellin tranform. This will allow us to compare our results above to the ones in \cite{kobayashi03}, \cite{lei09} and \cite{sprung09}. Throughout, $E$ is a complete extension of $\QQ_p$.

  \subsubsection{The Fourier transform for $\ZZ_p$ and $\ZZ_p^\times$}

   We recall some standard results of $p$-adic Fourier theory. These results are due to Amice \cite{amicevelu75}; see also \cite{colmez10} for a more modern account. We denote by $C^\la(\ZZ_p, E)$ the space of locally analytic $E$-valued functions on $\ZZ_p$, with the topology it acquires as the locally convex direct limit as $n \to \infty$ of the Banach algebras of functions analytic on cosets of $p^n \ZZ_p$. A \emph{distribution} on $\ZZ_p$ is a continuous $E$-linear functional $C^\la(\ZZ_p, E) \to E$; we write $D^{\la}(\ZZ_p, E)$ for the space of distributions.

   \begin{proposition}[{\cite[theorem 2.3]{colmez10}}]
    There is an isomorphism between $D^\la(\ZZ_p, E)$ and the subset of functions $f \in E[[T]]$ converging for all $T$ in the open unit disc of $\CC_p$, given by $\mu \mapsto F_\mu(T) = \sum_{n \ge 0} T^n \mu\left(\binom{x}{n}\right)$. The value of $F_\mu$ at a point $x \in E$ (with $|x| < 1$) is $\mu(\kappa_x)$, where $\kappa_x$ is the unique character of $\ZZ_p$ such that $\kappa(1) = 1 + x$.
   \end{proposition}

   Thus we may identify $D^\la(\ZZ_p, E)$ with $E\otimes\Brig$. Under this identification, the subspace $D^\la(\ZZ_p^\times, E)$ of distributions supported in $\ZZ_p^\times$ corresponds to $(E\otimes\Brig)^{\psi=0}$ \cite[\S 2.4.5]{colmez10}.

   Suppose $p \ne 2$. An alternative description of $D^\la(\ZZ_p^\times, E)$ is given by the isomorphism $\ZZ_p^\times = (1 + p\ZZ_p) \times \Delta \cong \ZZ_p \times \Delta$, where $\Delta$ is the group of $(p-1)$st roots of unity in $\ZZ_p$. If we fix a topological generator $\gamma$ of $1 + p\ZZ_p$, we thus have an isomorphism
   \[ D^\la(\ZZ_p^\times, E) \cong E \otimes \HG,\]
   where as in section \ref{construction} above, $\HG$ is the ring of formal series $f(\gamma - 1)$, for $f \in \Qp[\Delta][[X]]$ converging for all $|X| < 1$.

   Thus for a distribution $\mu$ on $\ZZ_p^\times$, we obtain two power series 
   \[ F^{+}_\mu(\pi) \in (E\otimes\Brig)^{\psi=0}\]
   and
   \[ F^{\times}_\mu(X) \in E\otimes \HG. \]
   These are related by the Mellin transform of lemma \ref{lem2}: we have $\mathfrak{M}(F^\times_\mu(\gamma))=F^+_\mu$.

  \subsubsection{Step functions}

   Let $n \ge 0$ be an integer. We say a function $f: \ZZ_p \to E$ is a \emph{step function of order $n$} if it is constant on any coset $a + p^n \ZZ_p$; the space $\Step_n(\ZZ_p)$ of such functions is clearly a subspace of $C^{\la}(\ZZ_p, E)$ of dimension $p^n$.

   For each $n$ we have an inclusion $\Step_{n}(\ZZ_p) \to \Step_{n+1}(\ZZ_p)$. A section of this is given by the ``averaging'' map $I : \Step_{n+1}(\ZZ_p) \to \Step_{n}(\ZZ_p)$ defined by 
   \[ I(f)(x) = \frac{1}{p} \sum_{y \in \ZZ/p\ZZ} f(x + p^{n}y).\]
   For $n \ge 1$, we say a function $f \in \Step_{n}(\ZZ_p)$ is a \emph{primitive step function} if it is in the kernel of this map, and write $\PStep_n(\ZZ_p)$ for the space of such functions, which clearly has dimension $p^{n-1}(p-1)$. For consistency we take $\PStep_0(\ZZ_p) = \Step_0(\ZZ_p) = K$.

   \begin{lemma}\label{basis1}
    Let $n \ge 0$ and suppose $E$ contains a primitive $p^n$-th root of unity $\zeta_{p^n}$. Then a basis for $\Step_n(\ZZ_p)$ is given by the functions $x \mapsto (\zeta_{p^n})^{xt}$, as $t$ varies through $\ZZ/p^n\ZZ$. The subset corresponding to $t \in (\ZZ/p^n\ZZ)^\times$ is a basis for $\PStep_n(\ZZ_p)$.
   \end{lemma}
   \begin{proof} 
    This follows immediately from the fact that $x \mapsto \frac{1}{p^n}\sum_{t \in \ZZ/p^n\ZZ} (\zeta_{p^n})^{xt} (\zeta_{p^n})^{-at}$ is the characteristic function of $a + p^n \ZZ_p$.
   \end{proof}

   We also have a ``multiplicative'' version. For $n \ge 1$, we define $\Step_n(\ZZ_p^\times)$ as the functions in $\Step_n(\ZZ_p)$ which are supported in $\ZZ_p^\times$. For $n \ge 2$ the averaging map restricts to a map $\Step_n(\ZZ_p^\times) \to \Step_n(\ZZ_p^\times)$, and we define $\PStep_n(\ZZ_p^\times)$ to be its kernel. We take $\PStep_1(\ZZ_p^\times) = \Step_1(\ZZ_p^\times)$, so for all $n \ge 1$ restriction to $\ZZ_p^\times$ defines a surjective map $\PStep_n(\ZZ_p) \to \PStep_n(\ZZ_p^\times)$.

   \begin{lemma}\label{basis2}
    A basis for $\Step_n(\ZZ_p^\times)$ is given by the Dirichlet characters modulo $p^n$. For $n \ge 2$ the subset of primitive characters modulo $p^n$ gives a basis for $\PStep_n(\ZZ_p^\times)$.
   \end{lemma}
   \begin{proof} 
    Similar to the previous lemma.
   \end{proof}

  \subsubsection{Relating the additive and multiplicative transforms}

   We now suppose we are given a distribution $\mu \in D^\la(\ZZ_p^\times, E)$. Let $F_\mu^\times$ and $F_\mu^+$ be the corresponding transforms. 

   \begin{theorem}\label{david}
    For $n \ge 2$, the following are equivalent:
    \begin{enumerate}
     \item $F_\mu^+$ is divisible by the cyclotomic polynomial $\Phi_n(1 +\pi)$ in $\Brig$.
     \item $\mu$ annihilates $\PStep_n(\ZZ_p)$.
     \item $\mu$ annihilates $\PStep_n(\ZZ_p^\times)$.
     \item $F_\mu^\times(\chi)$ is zero for all primitive Dirichlet characters $\chi \bmod p^n$.
     \item $F_\mu^\times$ is divisible by $\Phi_{n-1}(1 + X)$ in $E\otimes\HG$.
    \end{enumerate}
    For $n = 1$, the same holds with the last two statements replaced by:
    \begin{enumerate}
     \item[(4')] $F_\mu^\times(\chi)$ is zero for all Dirichlet characters $\chi \bmod p$.
     \item[(5')] $F_\mu^\times$ is divisible by $X$ in $E \otimes \HG$.
    \end{enumerate}
   \end{theorem}
   \begin{proof}
    It is clear that $(1) \Leftrightarrow (2)$ for arbitrary $\mu \in D^\la(\ZZ_p, E)$ (not necessarily supported in $\ZZ_p^\times$), because of Lemma~\ref{basis1}. Since restriction of functions gives a surjective map $\PStep_n(\ZZ_p) \twoheadrightarrow \PStep_n(\ZZ_p^\times)$, we have $(2) \Leftrightarrow (3)$. The equivalence $(3) \Leftrightarrow (4)$ follows from Lemma~\ref{basis2}. 

    To show $(4) \Leftrightarrow (5)$ for $n \ge 2$, let us write $F_\mu^\times = \sum_{i = 1}^{p-1} [\tau(i)] F_i(X)$, where $F_i \in E[[X]]$ and $\tau(i) \in \Delta$ is the Teichm\"uller lift of $i$. For any primitive $p^{n-1}$st root of unity $\zeta$, there are exactly $p - 1$ primitive Dirichlet charcters modulo $p^n$ mapping $\gamma$ to $\zeta$, and their restrictions to $\Delta$ are given by $\tau(i) \mapsto \tau(i)^k$ for $k \in \ZZ / (p-1)\ZZ$. So $(4)$ is equivalent to
    \[ \sum_{i = 1}^{p-1}\tau(i)^k F_i(\zeta - 1) = 0 \]
    for all $k = 0 \dots p-2$ and all primitive $p^{n-1}$st roots of unity $\zeta$, which is equivalent to $F_i(\zeta - 1) = 0$ for each $i = 1, \dots, p-1$. In other words, each of the functions $F_i(X)$ vanishes at every root of the polynomial $\Phi_{n-1}(1 + X)$, which is clearly equivalent to $F_\mu^\times$ being divisible by $\Phi_{n-1}(1 + X)$ in $E\otimes\HG$. 

    (The only change necessary for $n = 1$ is to note that $\PStep_1(\ZZ_p^\times)$ is the linear span of all Dirichlet characters modulo $p$, not just the primitive ones.)
   \end{proof}

   We also have an accompanying result:
   \begin{lemma}
    Let $F \in (E\otimes\Brig)^{\psi=0}$. Then $\Phi_1(1 + \pi)$ divides $F(\pi)$ if and only if $\vp(\pi) = \pi \Phi_1(1 + \pi)$ divides $F(\pi)$.
   \end{lemma}
   \begin{proof}
    Since $\psi(F)(0) = 0$, we have
    \[ F(0) + \sum_{\substack{\zeta \in \mu_p\\ \zeta \ne 1}} F(\zeta - 1) = 0.\]
    Hence if $F$ vanishes at the points $\zeta - 1$ for primitive $\zeta \in \mu_p$, then it must also vanish at 0.
   \end{proof}

\subsection{The case \texorpdfstring{$a_p=0$}{ap = 0}}\label{leiswork}
   
   We now relate the construction of Coleman maps in this paper to the construction given in \cite{lei09} for modular forms with $a_p = 0$.

   \subsubsection{Construction of the Coleman maps}
   
    Consider $f$ a normalized new eigenform as in Section~\ref{Colemanmf} with $a_p=0$. To ease notation, we assume that $E=\QQ_p$. The plus and minus Coleman maps in~\cite{lei09} are constructed as follows. 

     Let $u=\chi(\gamma)$. In~\cite{pollack03}, Pollack defines the following elements of $\HG$:    
    \begin{align*}
     \log^+_{p,k} & =\prod_{j=0}^{k-2}\frac{1}{p}\prod_{n=1}^{+\infty}\frac{\Phi_{2n}(u^{-j}\gamma)}{p} \\
     \log^-_{p,k} & =\prod_{j=0}^{k-2}\frac{1}{p}\prod_{n=1}^{+\infty}\frac{\Phi_{2n-1}(u^{-j}\gamma)}{p}.
    \end{align*}
     
    Let $\nu^-=\bar{\nu}_1$, $\nu^+=\bar{\nu}_2$ be the basis of $\Dcris(V_{\bar{f}})$ as in Section~\ref{Colemanmf} and let $\eta^\pm=(1+\pi)\otimes \nu^\pm\in (\BB^+_{\rig,\QQ_p})^{\psi=0}\otimes\DD_{\cris}(V)$. Let
    \[ \mathcal{L}_{1,\eta^\pm}: H^1_{\Iw}(\QQ_p,V_{\bar{f}}(k-1))\rTo \HG\]
    be the map defined by \eqref{lmap}.

    \begin{lemma}
     $\log^\pm_{p,k}\mid \mathcal{L}_{1,\eta^\pm}(z)$ for any $z\in H^1_{\Iw}(\QQ_p, V_{\bar{f}}(k-1))$. 
    \end{lemma}
    \begin{proof}
     See~\cite[Lemma 2.2]{lei09}.
    \end{proof}
    
    One can therefore define
    \begin{align}
     \rmCol^\pm: H^1_{\Iw}(\QQ_p,V_{\bar{f}}(k-1)) & \rTo \Lambda_{\QQ_p}(G_\infty) \label{pmcoleman} \\
     z & \rMapsto \frac{\mathcal{L}_{1,\eta^\pm}(z)}{\log^\pm_{p,k}}. \notag
    \end{align}

    In this setting, we can work out the matrix $M$ in \eqref{M} explicitly. As in section \ref{ss} above, we let $n_1, n_2$ be the basis of $\NN(V_{\bar{f}})$ constructed in \cite{bergerlizhu04}. The results of \textit{op.cit.}~imply that the $\EA$-span of $n_1, n_2$ is $\NN(T_{\bar{f}})$ for a $G_{\Qp}$-stable $\calO_E$-lattice $T_{\bar{f}} \subset V_{\bar{f}}$.

    Recall that $M\in M_2(\vp(\BB^+_{\rig,\QQ_p}))$ is the matrix satisfying $\begin{pmatrix} \vp(n_1) \\ \vp(n_2)\end{pmatrix}=M\begin{pmatrix} \nu_1 \\ \nu_2\end{pmatrix}$.
    \begin{lemma}
    The matrix $M$ is given by
     \[
     \begin{pmatrix}
     0&(\lp)^{k-1}\\
     -(\lm/q)^{k-1} & 0
     \end{pmatrix}.
     \]
     \end{lemma}
     \begin{proof}
      With respect to the basis $n_1, n_2$ of $\NN(V_{\bar{f}})$ over $\BB_{\QQ_p}^+$, as chosen in \cite{bergerlizhu04}, the matrices of $\vp$ and $\gamma\in G_\infty$ are given by
    \begin{align}
     P=\begin{pmatrix} 0 & -1 \\ q^{k-1} & 0 \end{pmatrix}\qquad\text{and}\qquad\begin{pmatrix} \left(\frac{\lp}{\gamma(\lp)}\right)^{k-1} & 0 \\ 0 & \left(\frac{\lm}{\gamma(\lm)}\right)^{k-1}\end{pmatrix}
    \end{align}
    respectively. Then,
    \[
     \bar{\nu}_1=(\lp)^{k-1}n_1\qquad\text{and}\qquad \bar{\nu}_2=(\lm)^{k-1}n_2,
    \]
    so the base-change matrix $M'$ (defined in \eqref{Mprime}) is given by
    \begin{equation}\label{M'}
     \begin{pmatrix}(\lp)^{k-1}&0\\0& (\lm)^{k-1}\end{pmatrix}
    \end{equation}
    and the result follows from explicit calculations, using that $M=\left(\frac{t}{\pi q}\right)^{k-1}P^TM'^{-1}$.
   \end{proof}

  \begin{lemma}\label{philog}
     We have $\vp(\log^-(1+\pi))=\log^+(1+\pi)$ and $\vp(\log^+(1+\pi))=\frac{p}{q}\log^-(1+\pi)$.
    \end{lemma}
    \begin{proof}
     Immediate.
    \end{proof}

    \begin{lemma}\label{amicetransforms}
     For $i \in \{1, \dots, p-1\}$ we have
     \[
      \mathfrak{M}^{-1} \left( (1 + \pi)^i \log^+(1 + \pi)^{k-1} \cdot \vp(\Brig)\right) = \tau(i) \log_{p, k}^-(\gamma) \cdot \mathcal{H}(\Gamma)
     \]
     and
     \[
      \mathfrak{M}^{-1} \left( (1 + \pi)^i \log^-(1 + \pi)^{k-1} / q^{k-1} \cdot \vp(\Brig)\right) = \tau(i) \log_{p,k}^+(\gamma) \cdot \mathcal{H}(\Gamma),
     \]
     where $\tau(i) \in \Delta$ is the Teichmuller lift of $i$.
    \end{lemma}
    
    \begin{proof}
     Let us suppose first that $k = 2$. Any element $f \in (1 + \pi)^i \log^+(1 + \pi) \cdot \vp(\Brig)$ is $F_\mu^+$ for some distribution $\mu$ on $\ZZ_p$, supported in $i + p\ZZ_p \subseteq \ZZ_p^\times$; hence we have a corresponding multiplicative Fourier transform $F_\mu^\times = \mathfrak{M}^{-1}(f)$, lying in $\tau(i)\mathcal{H}(\Gamma)$. Moreover, we have the implications
     \[\begin{aligned}
      &\log^+(1 + \pi) \mid f \text{ in $\Brig$}\\
      \iff &\Phi_{n}(1 + \pi) \mid f \text{ for all even $n \ge 2$}\\
      \iff &\Phi_{n}(1 + X) \mid \mathfrak{M}^{-1}(f) \text{ for all odd $n \ge 1$ (by theorem \ref{david})}\\
      \iff &\log^-(1 + X) \mid \mathfrak{M}^{-1}(f).
     \end{aligned}\]
     The second statement is similar, noting that $q = \Phi_1(1 + \pi)$ and hence $\log^-(1 + \pi) / q$ divides $f$ if and only if $f$ vanishes at the primitive $p^n$-th roots of unity for all odd $n \ge 3$.

     For general $k \ge 2$, we note that $f \in \Brig$ vanishes to order $k - 1$ at a point $z$ if and only if $\partial^j f$ vanishes at $z$ for $j = 0, \dots, k-2$, where $\partial$ is the differential operator $(1 + \pi) \frac{\mathrm{d}}{\mathrm{d}\pi}$ introduced in \S \ref{crysrepsandwachmods}. Applying the preceding argument to each of the functions $\partial^j f$, we see that $\log^+(1 + \pi) \mid f$ if and only if $\mathfrak{M}^{-1}(\partial^j f)$ is divisible by $\log^-(1 + X)$ for $0 \le j \le k-2$. Since $\mathfrak{M}^{-1}(\partial^j f)(z) = \mathfrak{M}^{-1}(f)(u^j(1 + z) - 1)$ where $u=\chi(\gamma)$, this is equivalent to the divisibility of $f$ by $\log^-_{p, k}$. Again, the second statement follows very similarly to the first.
    \end{proof}

    \begin{proposition}\label{pmtransform}
     There exists $a^\pm\in\Lambda_E(G_\infty)^\times$ such that
      \[
       \underline{M}=\begin{pmatrix}0&-a^-\log_{p,k}^-\\a^+\log_{p,k}^+&0\end{pmatrix}.
      \]
    \end{proposition}
    \begin{proof}
     By Lemma~\ref{amicetransforms}, $\mathfrak{M}$ restricts to an isomorphism of $\HG$-modules between the subspaces $X^{\pm} = (1 + \pi) \vp(\log^\pm(1 + \pi))^{k-1} \cdot \vp(\Brig)$ and $Y^{\pm} = \log^\pm_{p, k} \cdot \HG$. In particular, there exist $a^\pm\in \HG$ such that
     \[
      \mathfrak{M}^{-1}((1+\pi)\vp(\log^\pm(1+\pi))^{k-1})=a^\pm\log_{p,k}^\pm.
     \]
     Furthermore, $(1 + \pi) \vp(\log^\pm(1 + \pi))^{k-1}$ are $\Lambda_E(G_\infty)$-module generators of $(1 + \pi) \vp(\log^\pm(1 + \pi))^{k-1} \cdot \vp(\BB^+_{\Qp})$, by Proposition~\ref{lambdabasis}. Since any finitely-generated submodule of $\HG$ is closed, they must be $\HG$-module generators of the closures of these spaces, which are clearly $X^\pm$. Therefore the images of $(1 + \pi) \vp(\log^\pm(1 + \pi))^{k-1}$ under $\mathfrak{M}^{-1}$ must be generators of $Y^\pm$, so the factors $a^\pm$ are units.
    \end{proof}

    Therefore, by \eqref{prpairing},  we have:
    \begin{corollary}
     Let $a^\pm$ be as in Proposition~\ref{pmtransform}, then $a^-\uCol_1=\rmCol^-$ and $a^+\uCol_2=\rmCol^+$.
    \end{corollary}

    
    \subsubsection{Description of the kernels}
    
     The aim of this section is to give a simple description of $\ker(\rmCol_i)$ for $i=1,2$. Recall that the basis $\bar{\nu}_1,\bar{\nu}_2$ of $\DD_{\cris}(V_{\bar{f}})$ determines a basis of $\DD_{\cris}(V_{\bar{f}}(k-1))$ via the map $\bar{\nu}_i\mapsto\bar{\nu}_i\otimes e_{k-1}t^{1-k}$. We first need to know a bit more about $\NN(V_{\bar{f}})$. As stated in~\cite[Section II.3]{berger03}, we have a comparison isomorphism
     \[ \iota: \BB^+_{\rig,\QQ_p}[t^{-1}]\otimes_{\BB^+_{\QQ_p}}\NN(V_{\bar{f}}(k-1)) \cong \BB^+_{\rig,\QQ_p}[t^{-1}]\otimes_{\QQ_p}\DD_{\cris}(V_{\bar{f}}(k-1)).\]
     
     By \eqref{M} and \eqref{xoverDcris}, if $x\in\DD(V_{\bar{f}}(k-1))^{\psi=1}$, then we can write $\iota(x)=x_1(\bar{\nu}_1\otimes e_{k-1}t^{1-k})+x_2(\bar{\nu}_2\otimes e_{k-1}t^{1-k})$ where 
    \begin{align*}
     x_1 &= x_1'(\lm)^{k-1}\\
     x_2 &= x_2'(\lp)^{k-1}
    \end{align*}
    for some $x_1',x_2'\in \BB_{\Qp}^+$.

    We will need the following auxiliary lemma.
   
    \begin{lemma}\label{sameconstantterm}
     Let $x$ be as above. Then $p^{k-2}\theta(x_1)+\theta(x_2)=0$.
    \end{lemma}
    \begin{proof}
     By \cite[Theorem II.6]{berger03}, we have 
     \begin{equation}\label{valuedualexp}
      \exp^*_{\Qp,V_{\bar{f}}(k-1)}\big(h^1_{\Qp,V_{\bar{f}}(k-1)}(x)\big)=(1-p^{-1}\varphi^{-1})\partial_V(x). 
     \end{equation}
     Since $\partial_V(x)=\theta(x_1)\bar{\nu}_1\otimes e_{k-1}t^{1-k}+\theta(x_2)\bar{\nu}_2\otimes e_{k-1}t^{1-k}$, we have
     \[ (1-p^{-1}\varphi^{-1})\partial_V(x) =\big(\theta(x_1)-p^{-1}\theta(x_2)\big)\nu_1 + \big(p^{k-2}\theta(x_1)+\theta(x_2)\big)\nu_2.\]
     The image of $\exp^*_{\Qp,V_{\bar{f}}(k-1)}$ is contained in $\Fil^0\DD_{\cris}(V_{\bar{f}}(k-1))$, which implies that $p^{k-2}\theta(x_1)+\theta(x_2)=0$.
    \end{proof}

    \begin{lemma}
     Let $x\in\DD(V_{\bar{f}}(k-1))^{\psi=1}$, and write $\iota(x)= x_1(\bar{\nu}_1\otimes e_{k-1}t^{1-k})+x_2(\bar{\nu}_2\otimes e_{k-1}t^{1-k})$ as above. Then
     \begin{itemize}
      \item[(i)] $x\in\ker(\rmCol_1)$ if and only if $\vp(x_1)=-p^{k-1}\psi(x_1)$;
      \item[(ii)] $x\in\ker(\rmCol_2)$ if and only if $\vp(x_2)=-p^{k-1}\psi(x_2)$.
     \end{itemize}    
    \end{lemma}
    \begin{proof}
     We will prove the proposition for $\rmCol_1$; the proof for $\rmCol_2$ is analogous. Note that the condition that $\psi(x)=x$ translates as $\psi(x_1)=-p^{1-k}x_2$ and $\psi(x_2)=x_1$. By Lemma \ref{philog}, $\rmCol_1(x)=x_2'-\vp(x_1')=0$ if and only if $x_2=\vp(x_1)$. Hence, $\rmCol_1(x)=0$ if and only if $\vp(x_1)=-p^{k-1}\psi(x_1)$.  
    \end{proof}
    
    \begin{proposition}\label{evalofkernel}
     Let $x$ be as above, and write $x_i=f_i(\pi)$ with $f_i(X)\in \QQ_p[[X]]$. Then
      \begin{itemize}
       \item[(i)] $x\in\ker(\rmCol_1)$ if and only if 
        \begin{align}
         \Tr_{\Qpn/ \QQ_{p,n-1}}\big(f_1(\zeta_{p^n}-1)\big) &= -p^{2-k}f_1(\zeta_{p^{n-2}}-1)\label{firstkernel}\text{ for all $n\geq 2$, and}\\
                \Tr_{\QQ_{p,1}/\Qp} ( f_1(\zeta_p-1) )&= -(1+p^{2-k})f_1(0);\label{firstkernel2}
        \end{align}
     
       \item[(ii)] $x\in\ker(\rmCol_2)$ if and only if 
        \begin{align*}
         \Tr_{\Qpn/ \QQ_{p,n-1}}\big(f_2(\zeta_{p^n}-1)\big) & = -p^{2-k}f_2(\zeta_{p^{n-2}}-1) \text{ for all $n\geq 2$, and}\\
         \Tr_{\QQ_{p,1}/\Qp} ( f_2(\zeta_p-1) ) & = -(1+p^{2-k})f_2(0).
        \end{align*}
      \end{itemize}
    \end{proposition}
    \begin{proof}
     We prove the proposition for $\rmCol_1$. Recall that 
     \[
     \vp\psi(x_1)=p^{-1}\sum_{\zeta^p=1}f_1\big(\zeta(1+\pi)-1\big).
     \]
     Hence, $\vp(x_1)=-p^{k-1}\psi(x_1)$ implies that
     \begin{equation}\label{psiequation} 
      \sum_{\zeta^p=1}f_1\big(\zeta(1+\pi)-1\big)=-p^{2-k}\varphi^2\big(f_1(\pi)\big).
     \end{equation}
     Let $n\geq 2$. On applying $\theta\circ\vp^{-n}$ to~\eqref{psiequation} implies that
     \[
      \Tr_{\Qpn/ \QQ_{p,n-1}}\big(f_1(\zeta_{p^n}-1)\big) =\sum_{\zeta^p=1}f_1(\zeta\zeta_{p^n}-1) 
       =-p^{2-k}f_1(\zeta_{p^{n-2}}-1).
     \]
     Similarly, we obtain the second condition by applying $\theta$ to \eqref{psiequation}.

     Conversely, assume that \eqref{firstkernel} holds for all $n\ge2$, then $\vp(f_1)+p^{k-1}\psi(f_1)=0$ at $\zeta_{p^{n}}-1$. Recall that $x_1=x_1'(\lm)^{k-1}$ where $x_1'\in\BB_{\Qp}^+$. By Lemma~\ref{philog},
     \[
     \vp(x_1)+p^{k-1}\psi(x_1)=(\vp(x_1')+\psi(q^{k-1}x_1'))(\lp)^{k-1}.
     \]
     Hence, the power series in $\QQ\otimes\Zp[[X]]$ corresponding to $(\vp(x_1')+\psi(q^{k-1}x_1'))$ has infinitely many zeros, so it must be zero itself and we are done.  
    \end{proof}

    As a corollary, we obtain the following descriptions of $\ker(\rmCol_i)$. 

    \begin{corollary}\label{evalofkernel2}
     For $x\in\DD(V_{\bar{f}}(k-1))^{\psi=1}$, write $e_n(x)=\exp^*_{n,V_{\bar{f}(k-1)}}\circ\Pr_n\circ h^1_{\Qp,\Iw}(x)$ where $\Pr_n$ is the projection from $H^1_{\Iw}(\Qp,V_{\bar{f}}(k-1))$ to $H^1(\Qpn,V_{\bar{f}}(k-1))$. Let $i=1$ (respectively $i=2$), then 
     \[
     \ker(\rmCol_i)=\{x\in\DD(V_{\bar{f}}(k-1))^{\psi=1}:e_0(x)=0\text{ and }e_{n+1}(x)=p^{-1}e_{n}(x)\forall\text{ odd (respectively even) } n\ge1 \}. 
     \]
    \end{corollary}
    \begin{proof}
     Again, we only prove this for $i=1$. By \cite[Th\'eor\`eme IV.2.1]{cherbonniercolmez99}, we have $e_n(x)=p^{-n}\partial_V(\vp^{-n}(x))$ for all $n\geq 1$. But $\vp^{-2}$ is the multiplication by $-p^{k-1}$ on $\Dcris(V_{\bar{f}}(k-1))$. Using again that $\image(\exp^*_{n,V_{\bar{f}(k-1)}})\subset\Fil^0\DD_{\cris}(V)$, we see that
     \begin{align*}
      e_{2n}(x) &= p^{-2n}\cdot (-p)^{n(k-1)}f_1(\zeta_{p^{2n}}-1)\bar{\nu}_1\otimes t^{1-k}e_{k-1}\\
      e_{2n+1}(x) &= p^{-2n-1}\cdot(-p)^{n(k-1)}f_2(\zeta_{p^{2n+1}}-1)\bar{\nu}_1\otimes t^{1-k}e_{k-1}
     \end{align*}
     and $f_2(\zeta_{p^{2n}}-1)=f_1(\zeta_{p^{2n-1}}-1)=0$ for all $n\geq 1$. Therefore, \eqref{firstkernel} holds for $2n-1$ and for $2n$ if and only if $e_{2n}(x)=\Tr_{F_{2n+1}/F_{2n}}(e_{2n+1}(x))=p^{-1}e_{2n-1}(x)$.

     Now $e_0(x)=(f_1(0)-p^{-1}f_2(0))\bar{\nu}_1\otimes t^{1-k}e_{k-1}$ by~\eqref{valuedualexp} and $p^{k-2}f_1(0)+f_2(0)=0$ by Lemma \ref{sameconstantterm}, so 
     \[e_0(x)=(1+p^{k-3})f_1(0)\bar{\nu}_1\otimes t^{1-k}e_{k-1}=-(p^{2-k}+p^{-1})f_2(0)\bar{\nu}_1\otimes t^{1-k}e_{k-1}\]
     The condition \eqref{firstkernel2} is therefore equivalent to $f_1(0)=0$, which in turns is equivalent to $e_0(x)=0$.
    \end{proof}

    In the rest of this section, we will relate Corollary~\ref{evalofkernel2} to the description of $\ker(\rmCol^\pm)$ in~\cite[Section 2.2]{lei09}. Recall that $H^1_f(\Qpn,T_f(1))^\pm$ is defined by
    \[\left\{x\in H^1_f(\Qpn,T_f(1)):\mathrm{cor}_{n/m+1}x\in H^1_f(\QQ_{p,m},T_f(1))\ \forall m\mathrm{\ even\ (odd)},m<n\right\}.\]
    Denote by $H^1_\pm(\Qpn,T_{\bar{f}}(k-1))$ the annihilator of $H^1_f(\Qpn,T_f(1))^\pm$ under the pairing
    \begin{equation}\label{pairing}
     [,]_n:H^1(\Qpn,T_f(1))\times H^1(\Qpn,T_{\bar{f}}(k-1))\rightarrow \ZZ_p.
    \end{equation}
    As shown in~\cite[Section 2.2.4]{lei09}, we have $\ker(\rmCol^\pm)=\varprojlim_n H^1_\pm(\Qpn,T_{\bar{f}}(k-1))$. Hence, we can identify the kernels described in Corollary~\ref{evalofkernel2} with $\ker(\rmCol^\pm)$ described in \cite{lei09} via the isomorphism $h^1_{\Iw,V_{\bar{f}}(k-1)}$: 

    \begin{proposition}
     For any $x\in H^1(\Qpn,T_{\bar{f}}(k-1))$ and $m\le n$, let $e_m(x)=\exp^*_{m,V_{\bar{f}}(k-1)}(\mathrm{cor}_{n/m}(x))$. Then, $H^1_\pm(\Qpn,T_{\bar{f}}(k-1))$ coincides with the following set:
     \[
     \{x\in H^1(\Qpn,T_{\bar{f}}(k-1)):e_0(x)=0 \text{ and } e_{m}(x)=p^{-1}e_{m-1}(x) \forall m\text{ odd (even)},m\le n\}
     \]
    \end{proposition}
    \begin{proof}
     On the one hand, \eqref{pairing} factors through
     \[ H^1_f(\Qpn,T_f(1))\times \frac{H^1(\Qpn,T_{\bar{f}}(k-1))}{H^1_f(\Qpn,T_{\bar{f}}(k-1))}\rightarrow \ZZ_p. \]
     On the other hand, the pairing
     \[
     [\sim,\sim]'_n:\Big(\Qpn\otimes\Dcris(V_f(1))\Big)\times\Big(\Qpn\otimes\Dcris(V_{\bar{f}}(k-1))\Big)\rightarrow\Qpn\stackrel{\Tr_{n/0}}{\longrightarrow} \QQ_p
     \]
     factors through
     \[
     \Big(\Qpn\otimes\Dcris(V_f(1))/\Dcris^0(V_f(1))\Big)\times\Big(\Qpn\otimes\Dcris^0(V_{\bar{f}}(k-1))\Big)\rightarrow \QQ_p.
     \]
     Hence, the compatibility of the two pairings, namely $[\exp_{n,V_f(1)}(\sim),\sim]_n=\Tr_{n/0}[\sim,\exp^*_{n,V_{\bar{f}}(k-1)}(\sim)]'_n$, implies that $H^1_\pm(\Qpn,T_{\bar{f}}(k-1))$ is the $\exp^*_{n,V_{\bar{f}}(k-1)}$-preimage of $\Big(\Qpn^{\pm}\otimes\Dcris(V_f(1))/\Dcris^0(V_f(1))\Big)^\bot$ where
     \[
     \Qpn^{\pm}=\{x\in \Qpn:\Tr_{n/m+1}(x)\in\QQ_{p,m}\ \forall m\mathrm{\ even\ (odd)},m<n\}.
     \]
     But we have:
     \[
     \Big(\Qpn^{\pm}\otimes\Dcris(V_f(1))/\Dcris^0(V_f(1))\Big)^\bot=\Big(\Qpn^\pm\Big)^\bot\otimes\Dcris^0(V_{\bar{f}}(k-1))
     \] 
     where $\Big(\Qpn^\pm\Big)^\bot$ is the orthogonal complement of $\Qpn^\pm$ under the pairing
     \begin{align*}
      \Qpn\times\Qpn &\rightarrow \Qp\\
      (x,y) &\mapsto \Tr_{n/0}(xy).
     \end{align*}
     By simple linear algebra, we have
     \[
     \Big(\Qpn^\pm\Big)^\bot=\{x\in \Qpn:\Tr_{n/0}(x)=0\text{ and }\Tr_{n/m+1}(x)\in\QQ_{p,m}\ \forall m\mathrm{\ odd\ (even)},m<n\},
     \]
     hence the lemma.
    \end{proof}


   \subsection{Elliptic curves with \texorpdfstring{$a_p = 0$}{ap = 0}}\label{elliptic}
   
   We now specialize to the case when $f$ corresponds to an elliptic curve $E$ over $\QQ$ with $a_p = 0$. Then $V_{\bar{f}}(k-1) = V_f(1) = \Qp \otimes_{\ZZ_p} T$, where $T = T_p(E)$. Furthermore, $E[p]$ is irreducible as a mod $p$ representation of $G_{\Qp}$; thus $T$ is the unique $G_{\Qp}$-stable lattice in $\Qp \otimes_{\ZZ_p} T_p(E)$ up to scaling, and in particular we may take the lattice $T_f(1)$ constructed in \cite{bergerlizhu04} (which is only defined up to scaling) to coincide with $T$.

   In this situation, we can recover results of Kobayashi~\cite{kobayashi03} which give a precise description of the images $\DD(T)^{\psi=1}$ under the Coleman maps. Recall that if $x\in\DD(V)^{\psi=1}$, say $x=(x_1n_1+x_2n_2)\otimes\pi^{-1}e_1$, then we have
   \begin{align*}
    \rmCol_1(x) & =x_2-\varphi(x_1) \\
    \rmCol_2(x) & =qx_1+\vp(x_1)
   \end{align*}
   where we have replaced $\rmCol_2$ by $-\rmCol_2$ for simplicity.
   
   \begin{proposition}\label{plusimage}
    The map $\rmCol_{1}:\DD(T)^{\psi=1}\rightarrow (\AA_{\QQ_p}^+)^{\psi=0}$ is surjective.
   \end{proposition}
   \begin{proof}
    We first show that $\big(\pi\AA^+_{\QQ_p}\big)^{\psi=0}\subset\image(\rmCol_{1})$. If $y\in \big(\pi\AA^+_{\QQ_p}\big)^{\psi=0}$, then the series $\sum_{i\geq 1}(-1)^i\frac{\varphi^{2i-1}(y)}{q\dots \varphi^{2i-2}(q)}$ and  $\sum_{i\geq 0}(-1)^i\frac{\varphi^{2i}(y)}{\varphi(q)\dots\varphi^{2i-1}(q)}$ converge in $\AA^+_{\QQ_p}$ to elements $x_1$ and $x_2$, respectively, and it is easy to see that $\psi(qx_2)=-x_1$ and $\psi(x_1)=x_2$. It follows that if we let $x=x_1\log^-(1+\pi)\nu_1+x_2\log^+(1+\pi)\nu_2$, then $x\in\DD(T)^{\psi=1}$, and moreover $\rmCol_1(x)=x_2-\varphi(x_1) = y$.
     
     In order to prove surjectivity of $\rmCol_{1}$, it is hence sufficient to show that there exists $y\in\image(\rmCol_{1})$ with $y\equiv1\mod\pi$. Let $y\in \AA^{+,\psi=0}_{\QQ_p}$ such that $\pi\mid((1+\pi)^p+y)$. As above, the series $\sum_{i\geq 1}(-1)^i\frac{\varphi^{2i-1}(y)+\varphi^{2i}(1+\pi)}{q\dots \varphi^{2i-2}(q)}$ and $\sum_{i\geq 0}(-1)^i\frac{\varphi^{2i}(y)+\varphi^{2i+1}(1+\pi)}{\varphi(q)\dots\varphi^{2i-1}(q)}$converge in $\AA^+_{\QQ_p}$. Let 
     \begin{align*}
      z_1 & =\frac{1}{2}\left((1+\pi)+\sum_{i\geq 1}(-1)^i\frac{\varphi^{2i-1}(y)+\varphi^{2i}(1+\pi)}{q\dots \varphi^{2i-2}(q)}\right),\\
      z_2 & = \frac{1}{2}\left(-\psi(q(1+\pi))+\sum_{i\geq 0}(-1)^i\frac{\varphi^{2i}(y)+\varphi^{2i+1}(1+\pi)}{\varphi(q)\dots\varphi^{2i-1}(q)}\right).
     \end{align*}
     It is easy to see that $\psi^2(q(1+\pi))=0$, so $\psi(qz_1)=-z_2$ and $\psi(z_2)=z_1$. It follows that if we let $x=z_1\log^-(1+\pi)\nu_1+z_2\log^+(1+\pi)\nu_2$, then $x\in\DD(T)^{\psi=1}$, and moreover
     \[ \rmCol_{1}(x)=z_2-\varphi(z_1) = 1\mod\pi. \] 
   \end{proof}
   
   \begin{corollary}\label{plusimage*}
    The map $\uCol_{1}:\DD(T)^{\psi=1}\rightarrow \Lambda(G_\infty)$ is surjective.
   \end{corollary}
   \begin{proof}
    By Proposition~\ref{plusimage}, there exists $x\in \NN(T)^{\psi=1}$ such that $\rmCol_1(x)=1+\pi$. The result therefore follows by precisley the same argument as in the proof of Theorem~\ref{image1}.
   \end{proof}
    
    \begin{proposition}\label{minusimage}
     The image of $\rmCol_{2}:\DD(T)^{\psi=1}\rightarrow \big(\AA_{\QQ_p}^+\big)^{\psi=0}$ is equal to $\big(\AA^{+,\psi=0}_{\QQ_p}\big)^{\Delta}+\varphi(\pi)\AA^{+,\psi=0}_{\QQ_p}$.
    \end{proposition}
    \begin{proof}
     A similar argument to the one in the proof of Proposition~\ref{col1} shows that $\varphi(\pi)\AA^{+,\psi=0}_{\QQ_p}\subset\image(\rmCol_{2})$. In~\cite{fontaine90}, Fontaine shows that $\big(\AA^+_{\QQ_p}\big)^{\Delta}=\ZZ_p[[\pi_0]]$, where $\pi_0=-p+\sum_{a\in\FF_p}[\varepsilon]^{[a]}$. Note that $\theta(\pi_0)=0$ and $\theta\circ\varphi^{-1}(\pi_0)=-p$, so $\pi_0=-p+\alpha q$ for some $\alpha\in\AA^+_{\QQ_p}$ satisfying $\alpha=1\mod\pi$. Now $\{ [\varepsilon]^{[a]}\}_{a\in\FF_p^\times}$ is a basis for $\AA^+_{\QQ_p}$ over $\varphi(\AA^+_{\QQ_p})$, so $\psi(\pi_0)=1-p$, and hence $\pi_0+p-1\in\AA_{\QQ_p}^{+,\psi=0}$. In order to prove that $\big(\AA^{+,\psi=0}_{\QQ_p}\big)^{\Delta}\subset \image(\rmCol_{2})$, it is therefore sufficient to prove the following results:
     
     \begin{enumerate}
      \item \label{claim1} $\pi_0+p-1\in\image(\rmCol_{2})$;
      \item \label{claim2} If $y\in \big(\AA^{+,\psi=0}_{\QQ_p}\big)^{\Delta}$, then $y=c(\pi_0+p-1)\mod\varphi(\pi)$ for some $c\in {\ZZ_p}$.
     \end{enumerate}

     \begin{proof}[Proof of claim \ref{claim1}] 
      Note that since $\pi_0+p-1=-1+q\mod\varphi(\pi)$, (a) is equivalent to showing that there exists $y\in \image(\rmCol_{2})$ such that $y=-1+q\mod\varphi(\pi)$. If $i(x)=x_1\log^-(1+\pi)\nu_1+x_2\log^+(1+\pi)\nu_2$ for some $x\in \DD(T)^{\psi=1}$, then $\rmCol_{2}(x)=qx_1+\varphi(x_2)$. As shown in Lemma~\ref{sameconstantterm}, we have $\theta(x_1)=-\theta(x_2)$, so 
      \[ \rmCol_{2}(x)\equiv \theta(x_2)(1-q)\mod\varphi(\pi).\]
      Suppose now that $\theta(x_2)=0$ for all $x\in\DD(T)^{\psi=1}$. Then, the fact that $\rmCol_{1}(x)\equiv\theta(x_2)-\theta(x_1)\mod\pi$ implies that $\rmCol_1(x)\in \pi\AA^+_{\QQ_p}$ for all $x\in\DD(T)^{\psi=1}$, which contradicts the surjectivity of $\rmCol_1$.
     \end{proof}
     
     \begin{proof}[Proof of claim \ref{claim2}] 
      We will show that if $y\in \big(\AA^{+,\psi=0}_{\QQ_p}\big)^{\Delta}$, then 
      \begin{equation}\label{congruence}
       y=c(-1+q)\mod\varphi(\pi)
      \end{equation}
      for some $c\in {\Zp}$. Write $y=f(\pi_0)=g(\pi)$. In order to show~\eqref{congruence}, it is sufficient to prove that $g(0)=-(p-1)g(\zeta_p-1)$.  The condition that $y\in\ker(\psi)$ translates as
      \[ \frac{1}{p}\sum_{\xi^p=1}f\big(-p+\sum_{a\in\FF_p}\xi^a(\pi+1)^{[a]}\big)=0. \]
      Evaluating this condition at $\pi=0$ shows that $f(0)+(p-1)f(-p)=0$. By definition, we have $\pi_0=-p+\sum_{a\in\FF_p}(\pi+1)^{[a]}$, so $g(0)=f(0)$ and $g(\zeta_p-1)=f(-p)$, which finishes the proof.
    \end{proof}

    This completes the proof of proposition \ref{minusimage}. \end{proof}
    
    Let $\eta:\Delta\rightarrow (\ZZ/ p\ZZ)^\times$ be a tame character. For a $\Lambda(G_\infty)$-module $A$, denote by $A^\eta$ the $\Lambda(G_\infty)$-submodule of $A$ on which $\Delta$ acts via $\eta$. The following result is an immediate consequence of Proposition~\ref{minusimage}.
    
    \begin{corollary}\label{deltacomponents}
     We have
     \[ \image(\rmCol_{2})^\eta = 
      \begin{cases}
       \big(\AA^{+,\psi=0}_{\QQ_p}\big)^{\Delta} & \text{if $\eta=1$} \\
       \big(\varphi(\pi)\AA^{+,\psi=0}_{\QQ_p}\big)^\eta & \text{otherwise}
      \end{cases}
     \]
    \end{corollary}
    
    We can translate Proposition~\ref{minusimage} and Corollary~\ref{deltacomponents} into a statement about $\image(\uCol_2)$.
    
   \begin{proposition}\label{minusimage*}
    The image of $\uCol_{2}:\DD(T)^{\psi=1}\rightarrow \Lambda(G_\infty)$ is equal to $\big(\sum_{i=1}^{p-1}\delta^i\big)\Lambda(G_\infty)+(\gamma-1)\Lambda(G_\infty)$.
   \end{proposition}
   \begin{proof}
    Let $y_2=\vp(\pi)(1+\pi)\in\image(\rmCol_2)$. As shown in the proof of Proposition~\ref{minusimage}, $y=(0,y_2)\in\image(\rmCol)$; more precisely, there exists $x\in\NN(T)^{\psi=1}$ such that $\rmCol(x)=y$. Applying the algorithm for $\J$ (see Section~\ref{algorithm}) to $y$ shows that $\uCol_2(x)=(\gamma-1) \mod (p,(\gamma-1)^2)$, so the $\Lambda(G_\infty)$-submodule of $\Lambda(G_\infty)$ generated by $\uCol_2(x)$ is equal to the ideal generated by $(\gamma-1)$.
    
    Furthermore, $y_2'\sum_{i=1}^{p-1}(\pi+1)^i\in \image(\rmCol_2)$ by Proposition~\ref{minusimage}, and every $y\in \image(\rmCol_2)$ is congruent to a scalar multiple of $y_2'$ mod $\vp(\pi)$. If $x'\in \NN(T)^{\psi=1}$ satisfies $\rmCol_2(x')=y_2'$, then again the algorithm for $\J$ implies that $\uCol_2(x)=\sum_{i=1}^{p-1}\delta^i \mod (\gamma-1)$. This finishes the proof. 
   \end{proof}
   
   \begin{corollary}\label{deltacomponents*}
     We have
     \[ \image(\uCol_{2})^\eta = 
      \begin{cases}
       \Lambda(G_\infty)^{\Delta}=\big(\sum_{i=1}^{p-1}\delta^i\big)\Lambda(G_\infty) & \text{if $\eta=1$} \\
       \big((\gamma-1)\Lambda(G_\infty)\big)^\eta & \text{otherwise}
      \end{cases}
     \]    
   \end{corollary}
    
   Note that the results of Corollaries~\ref{plusimage*} and~\ref{deltacomponents*} are equivalent to Theorem 6.2 in~\cite{kobayashi03}.
  
   
  \subsection{The case \texorpdfstring{$k=2$}{k = 2}}

  In this section we consider the case of modular forms which have weight 2 and are non-ordinary at $p$. For modular forms with trivial character and coefficients in $\QQ$ (hence corresponding to elliptic curves), but with $a_p \ne 0$, this case was studied in detail by Sprung. 

  \subsubsection{Coleman maps via the Perrin-Riou pairing}    
   
   We first review Sprung's construction of the Coleman maps for elliptic curves over $\QQ$ with $p \mid a_p$ but $a_p \ne 0$, and explain how we can rewrite these Coleman maps using Perrin-Riou's pairing. 

   Let $f$ be a modular form as in Section~\ref{Colemanmf} with $k=2$. Define for $n\ge1$
   \[
   \begin{pmatrix}\Theta_n^1 & \Upsilon_n^1\\ \Theta_n^0 & \Upsilon_n^0\end{pmatrix}=\begin{pmatrix}0 & \Phi_{n}(\gamma)\\ -1 & a_p\end{pmatrix}\cdots\begin{pmatrix}0 & \Phi_1(\gamma)\\ -1 & a_p\end{pmatrix}\in M_2(\HG).
   \]
   Then, we have:
   \begin{lemma}
    Let $i\in\ZZ$ and write \[A_n^i=\begin{pmatrix}0 & p\\ -1 & a_p\end{pmatrix}^i\begin{pmatrix}\Theta_n^1 & \Upsilon_n^1\\ \Theta_n^0 & \Upsilon_n^0\end{pmatrix}.\] Then, $A_n^{i-n}$ converges in $M_2(\HG)$ as $n\rightarrow\infty$ for a fixed $i$. Write $A_\infty^i$ for the limit, then all entries of $A_\infty^i$ are $O(\log_p^{1/2})$. Moreover, if $\eta$ is a character on $G_\infty$ which factors through $G_n$ but not $G_{n-1}$, then $\eta(A_\infty^i)=\eta(A_m^{i-m})$ for any $m\ge n-1$.
   \end{lemma}
   \begin{proof}
    \cite[Lemma~3.21]{sprung09}
   \end{proof}
   
   \begin{proposition}
   For any $ \mathbf{z}\in H^1_{\Iw} (V_{\bar{f}}(1))$ and $0\ne\omega\in\Dcris^1(V_f)$, the entries of the row vector 
    \[
    \begin{pmatrix}\frac{1}{p}\LL_{1,(1+\pi)\otimes\vp(\omega)}(z)&-\LL_{1,(1+\pi)\otimes\omega}(z)\end{pmatrix}A_\infty^{-1}
    \]
   are both divisible by $\log_p(\gamma)/(\gamma-1)$.
    \end{proposition}
    \begin{proof} 
     For $n\in\ZZ$, write $u_{n}=(\alpha^{n}-\beta^{n})/(\alpha-\beta)$ where $\alpha$ and $\beta$ are the roots of $X^2-a_pX+p$. Then, $\vp^{n}=u_{n}\vp-pu_{n-1}$ and
     \[
     \begin{pmatrix}0 & p\\ -1 & a_p\end{pmatrix}^{n}=\begin{pmatrix}-pu_{n-1}&pu_{n} \\ -u_{n} &u_{n+1}\end{pmatrix}.
     \]
     Therefore, if $n>1$ and $\eta$ is a character of $G_\infty$ which factors through $G_n$ but not $G_{n-1}$ (so $\eta(\gamma)$ is a primitive $p^{n-1}$-th root of unity), we have 
     \[
     \eta(A_\infty^{-1})=\begin{pmatrix}-pu_{-n-1}&pu_{-n} \\ -u_{-n} &u_{-n+1}\end{pmatrix}\begin{pmatrix}0&0\\-1&a_p\end{pmatrix}\eta\begin{pmatrix}\Theta_{n-2}^1 & \Upsilon_{n-2}^1\\ \Theta_{n-2}^0 & \Upsilon_{n-2}^0\end{pmatrix}
     \]
     where the last matrix is the identity if $n=2$.

     By \cite[Section 1.1.4]{lei09}, we have
     \[
     \eta(\LL_{1,(1+\pi)\otimes v}(\mathbf{z}))=\frac{1}{\tau(\eta^{-1})}\sum_{\sigma\in G_n}\eta^{-1}(\sigma)[\vp^{-n}(v),\exp^*_{n,1}(z_n^\sigma)]_n
     \] 
     for any $v\in\Dcris(V_f)$ and $z\in H^1_{\Iw} (V_{\bar{f}}(1))$. Hence, if $\omega\in\Dcris^1(V_f)$, then
     \[
     \eta\left(\begin{pmatrix}\frac{1}{p}\LL_{1,(1+\pi)\otimes\vp(\omega)}(z)&-\LL_{1,(1+\pi)\otimes\omega}(z)\end{pmatrix}A_\infty^{-1}\right)=0
     \]
     because 
     \[
     \begin{pmatrix}\frac{1}{p}u_{-n+1}&-u_{-n}\end{pmatrix}\begin{pmatrix}-pu_{-n-1}&pu_{-n} \\ -u_{-n} &u_{-n+1}\end{pmatrix}\begin{pmatrix}0&0\\-1&a_p\end{pmatrix}= 0,
     \]
     which implies that
     \[
     \begin{pmatrix}\frac{1}{p}\vp^{-n+1}(\omega)&-\vp^{-n}(\omega)\end{pmatrix}\begin{pmatrix}-pu_{-n-1}&pu_{-n} \\ -u_{-n} &u_{-n+1}\end{pmatrix}\begin{pmatrix}0&0\\-1&a_p\end{pmatrix}\equiv 0\mod\Dcris^1(V_f).
     \]
    \end{proof}

    By \cite{perrinriou94}, the image of $\LL_{1,(1+\pi)\otimes v}$ is $O(\log_p^{1/2})$ for any $v\in\Dcris(V_f)$, so we obtain two Coleman maps:
    
    \begin{definition}\label{sprungcoleman}
     Fix a non-zero element $\omega\in\Dcris^1(V_{f})$. For $*=\vartheta,\upsilon$ and $z\in H^1_{\Iw} (V_{\bar{f}}(1))$, $\rmCol^*(z)\in\Lambda_E(G_\infty)$ is defined by
     \begin{equation}\label{sprungsdefinition}
      \begin{pmatrix}\rmCol^\vartheta(z)& \rmCol^\upsilon(z)\end{pmatrix}\cdot\log_p(\gamma)/p(\gamma-1)=\begin{pmatrix}\frac{1}{p}\LL_{(1+\pi)\otimes \vp(\omega)}(z)&-\LL_{(1+\pi)\otimes \omega}(z)\end{pmatrix}A_\infty^{-1}.
     \end{equation}
     In particular, we can define two $p$-adic $L$-functions
     \[
     \tilde{L}_p^*=\rmCol^*(\kato)\in\Lambda_E(G_\infty)
     \]
     where $\kato$ is the localization of the Kato zeta element and $*=\vartheta,\upsilon$.
    \end{definition}

    \begin{remark}
     The results above hold for any modular forms with $k=2$, $p\nmid N$ and $v_p(a_p) \ge 1/2$. This setting is slightly more general than that in \cite{sprung09}.
    \end{remark}


   \subsubsection{Compatibility of Coleman maps}

    Since condition (C) holds and $k=2$, with respect to the canonical basis of $\NN(V_f)$ given above, $P$ is simply
    \begin{equation}
     \begin{pmatrix}0&-1\\ q & a_p\end{pmatrix}.\label{simpleP}
    \end{equation}

    Write $B_\infty^i$ (respectively $B_n^i$) for the matrix obtained from $A_\infty^i$ (respectively $A_n^i$) by replacing $\Phi_m(\gamma)$ by $\vp^{m-1}(q)$ for all $m$. Then, we have:
    
    \begin{lemma}\label{Changes}
     Under the notation above, $M'=B_\infty^0$.
    \end{lemma}
    \begin{proof}
     By (\ref{simpleP}), $(B_n^{-n})^T=P\vp(P)\cdots\vp^{n-1}(P)A_\vp^{-n}$. For $\gamma\in G_\infty$, we write $\Gn=(B_n^{-n})^T\cdot\gamma\left((B_n^{-n})^T\right)^{-1}$. Then,
     \[
     P\cdot\vp\left(\Gn\right)\cdot\gamma(P)^{-1}=\Gnn.
     \]
     Hence, if we write $G_\gamma$ for the limit of $\Gn$ as $n\rightarrow\infty$, then 
     \[
     P\cdot\vp\left(G_\gamma\right)\cdot\gamma(P)^{-1}=G_\gamma,
     \]
     It is easy to check that $G_\gamma$ satisfies $G_{\gamma_1\gamma_2}=G_{\gamma_1}\cdot\gamma_1(G_{\gamma_2})$ for any $\gamma_1,\gamma_2\in G_\infty$. Hence, we recover the action of $G_\infty$ on the Wach module $\NN(V_f)$. In other words, $G_\gamma$ is the matrix of $\gamma$ with respect to the basis $n_1$, $n_2$ chosen in \cite{bergerlizhu04}. Since $G_\gamma=(B_\infty^0)^T\cdot\gamma\left((B_\infty^0)^T\right)^{-1}$ and $G_\gamma|_{\pi=0}=I$, we have
     \[
     B_\infty^0\begin{pmatrix}n_1\\n_2\end{pmatrix}\in \Big((\EB)\otimes\NN(V_f)\Big)^{G_\infty}=\Dcris(V_f)
     \]
     and $M'=B_\infty^0$.
    \end{proof}

    We write $A^c=\det(A)A^{-1}$ if $A$ is an invertible matrix, then we have:

    \begin{corollary}\label{describesprung} 
     The matrix $M$ can be obtained from $(A_\infty^{-1})^c$ by replacing $\Phi_m$ by $\vp(q)^m$.     
    \end{corollary}
    \begin{proof}
     Recall that 
      \[
        M=\frac{t}{\pi q}P^T(M')^{-1}=\frac{t}{\pi q}\vp(M'^{-1})A_\vp^T.
      \]
     By Lemma~\ref{Changes}, $\det(M')=\det(B_\infty^0)=\prod_{n\ge0}\frac{\vp^n(q)}{p}=t/\pi$. But $\det A_{vp}=p$ and $A_{\infty}^{i+1}=A_\vp^TA_{\infty}^i$ for all $i$. Hence, we have
      \[
        M=\vp\left((A_{\vp}^T)^{-1}B_\infty^0\right)^c=\vp(B_\infty^{-1})^c
      \]
     and we are done.
    \end{proof}

    On setting $\nu_1=-\omega$ in \eqref{sprungsdefinition}, \eqref{prpairing} implies that
    \begin{equation}\label{sprungrelation}
     \begin{pmatrix}\uCol_1&\uCol_2\end{pmatrix}\underline{M}A_{\infty}^{-1}=\begin{pmatrix}\rmCol^\vartheta\circ h^1_{\Iw} & \rmCol^\upsilon\circ h^1_{\Iw}\end{pmatrix}\log_p(\gamma)/p(\gamma-1).
    \end{equation}

    By \cite{sprung09}, 
    \[
    \image(\rmCol^\vartheta)=\image(\rmCol^\upsilon)=\Lambda_E(G_\infty)\]
    and \eqref{sprungrelation} implies that the matrix $\underline{M}A_{\infty}^{-1}$ defines a $\Lambda_E(G_\infty)$-linear map from $\Lambda_E(G_\infty)^{\oplus2}$ onto $(\log_p(\gamma)/p(\gamma-1)\Lambda_E(G_\infty))^{\oplus2}$. Hence, there exists $A\in GL_2(\Lambda_E(G_\infty))$, $\underline{M}A_{\infty}^{-1}=[\log_p(\gamma)/p(\gamma-1)]A$. This implies
    \[
    \begin{pmatrix}\uCol_1&\uCol_2\end{pmatrix}A=\begin{pmatrix}\rmCol^\vartheta\circ h^1_{\Iw} & \rmCol^\upsilon\circ h^1_{\Iw}\end{pmatrix}.
    \]
    We also see that $\underline{M}$ and $(A_{\infty}^{-1})^c$ agree up to an element in $GL_2(\Lambda_E(G_\infty))$ which is a generalization of Proposition~\ref{pmtransform} because of the description of $M$ in Corollary~\ref{describesprung}.


\section{Main conjectures}\label{mainconjecture}

\subsection{Kato's main conjecture}\label{katozeta}

In general, if $V$ is a $p$-adic representation of $G_\QQ$ unramified outside a finite set of primes, and $T$ is a $\ZZ_p$-lattice in $V$ stable under $G_\QQ$, we write
\begin{align*}
\Hh^i(T) &= \varprojlim_n H^i_{\text{\'et}}\left(\operatorname{Spec} \ZZ[\zeta_{p^n}, \tfrac{1}{p}], j_* T\right),\\
\Hh^i(V) &= \Qp \otimes_{\ZZ_p} \Hh^i(T).
\end{align*}
for $i=1,2$; see \cite[\S\S 8.2 \& 12.2]{kato04}. Here $j$ is the natural map $\operatorname{Spec} \QQ(\zeta_{p^n}) \to \operatorname{Spec} \ZZ[\zeta_{p^n}, \tfrac{1}{p}]$. Note that $\Hh^i(V)$ is independent of the choice of lattice $T$.

We now continue under the notation of Section~\ref{Colemanmf} and Section~\ref{ordinary}. Fix a uniformizer $\varpi$ of $\calO_E$. Let $\ZZ(T_f)\subset \Hh^1(T_f)$ denote the $\Lambda_{\calO_E}(G_\infty)$-module generated by the Kato zeta elements as defined in \cite[Theorem 12.5]{kato04} and write $\ZZ(V_f)=\ZZ(T_f)\otimes\QQ$. The following assumption will be needed for some of the results below.

\begin{itemize}
\item\textbf{Assumption (E)}: there exists a basis of $T_f$ for which the image of $\Gal(\overline{\QQ} / \QQ_\infty)$ in $\operatorname{GL}_2(\OO_E)$ contains $\operatorname{SL}_2(\ZZ_p)$.
\end{itemize}

\begin{theorem}[{\cite[theorem 12.5]{kato04}}]\label{kato}Let $\eta:\Delta\rightarrow\Zp^\times$ be a character, then:
\begin{itemize}
\item[(a)] $\Hh^2(T_f)$ is a torsion $\Lambda_{\calO_E}(G_\infty)$-module.
\item[(b)] $\Hh^1(T_f)$ is a torsion free $\Lambda_{\calO_E}(G_\infty)$-module and $\Hh^1(V_f)$ is a free $\Lambda_{E}(G_\infty)$-module of rank 1.
\item[(c)] The quotient $\Hh^1(V_f)/\ZZ(V_f)$ is a torsion $\Lambda_{E}(G_\infty)$-module.
\item[(d)] $\Char_{\Lambda_E(\Gamma)}(\Hh^1(V_f)^\eta/\ZZ(V_f)^\eta)\subset\Char_{\Lambda_E(\Gamma)}(\Hh^2(V_f)^\eta)$.
\item[(e)] If assumption (E) holds, then $\ZZ(T_f)\subset\Hh^1(T_f)$. Moreover, $\Hh^1(T_f)$ is a free $\Lambda_{\calO_E}(G_\infty)$-module of rank 1 and
\[\Char_{\Lambda_{\calO_E}(\Gamma)}(\Hh^1(T_f)^\eta/\ZZ(T_f)^\eta)\subset\Char_{\Lambda_{\calO_E}(\Gamma)}(\Hh^2(T_f)^\eta).\] 
\end{itemize}
\end{theorem} 

Kato's main conjecture states that:

\begin{conjecture}\label{katoMC}
Let $\eta:\Delta\rightarrow\Zp^\times$ be a character, then $\ZZ(T_f)^{\eta}\subset\Hh^1(T_f)^\eta$ and
\[\Char_{\Lambda_{\calO_E}(\Gamma)}(\Hh^1(T_f)^\eta/\ZZ(T_f)^\eta)=\Char_{\Lambda_{\calO_E}(\Gamma)}(\Hh^2(T_f)^\eta).\]
\end{conjecture}

\begin{remark}
 The above formulation of the conjecture can be found in \cite[\S 5]{kobayashi03}; it is more convenient for our purposes than the original formulation (Conjecture 12.10 of \cite{kato04}).
\end{remark}


\subsection{Reformulation of Kato's main conjecture}

Let $K$ be a number field. The $p$-Selmer group of $f$ over $K$ is defined by  
\[
\Sel_p(f/K)=\ker\left(H^1(K,V_f/T_f(1))\rightarrow\prod_\nu\frac{H^1(K_\nu,V_f/T_f(1))}{H^1_f(K_\nu,V_f/T_f(1))}\right)
\]
where $\nu$ runs through all the places of $K$.

We choose a ``good basis'' $\nu_1, \nu_2$ of $\Dcris(V_{\bar{f}})$ in the sense of subsection \ref{Colemanmf}. Lemma \ref{choice} shows that we may find a lift $n_1, n_2$ of this to a basis of $\NN(V_{\bar{f}})$ such that $(1 + \pi) \vp(\pi^{1-k} n_1 \otimes e_{k-1}), (1 + \pi) \vp(\pi^{1-k} n_2 \otimes e_{k-1})$ is a $\Lambda_E$-basis of $\NN(V_{\bar{f}}(k-1))$. We choose such basis $(n_1, n_2)$.

With respect to this basis, we write $H^1_f(\Qpn,V_f/T_f(1))^i$ for the annihilator of the projection of $\ker(\uCol_i)$ in $H^1(\Qpn,T_{\bar{f}}(k-1))$ under the pairing
\[
H^1(\Qpn,T_{\bar{f}}(k-1))\times H^1(\Qpn,V_f/T_f(1))\rightarrow E/\calO_E.
\]
This enables us to make the following definition:

\begin{definition}\label{definitionselmer}
\begin{align*}
\Sel_p^i(f/\QQ(\mu_{p^n})) &= \ker\left(\Sel_p(f/\QQ(\mu_{p^n}))\rightarrow\frac{H^1(\Qpn,V_f/T_f(1))}{H^1_f(\Qpn,V_f/T_f(1))^i}\right) \\
\Sel_p^i(f/\QQ_\infty) &= \varinjlim_n\ \Sel_p^i(f/\QQ(\mu_{p^n})).
\end{align*}
\end{definition}

By the Poitou-Tate exact sequence (see \cite[Section 7]{kobayashi03} and \cite[Section 4]{lei09}), we have
\begin{equation}\label{poitoutate}
 \Hh^1(T_{\bar{f}}(k-1))\rightarrow\image(\uCol_i)\rightarrow\Sel_p^i(f/\QQ_\infty)^\vee\rightarrow\Hh^2(T_{\bar{f}}(k-1))\rightarrow0
\end{equation}
where $(\cdot)^\vee$ denotes the Pontryagin dual.

\begin{theorem}\label{torsionSelmer}
Under assumption (A) (if $f$ is supersingular at $p$) or assumption (A') (if $f$ is ordinary at $p$), $\Sel_p^i(f/\QQ_\infty)$ is $\Lambda_{\calO_E}(G_\infty)$-cotorsion. Moreover, there exist some $n_i\ge0$ such that
\[
\varpi^{n_i}\tilde{L}_{p,i}^\eta\in\Char_{\Lambda_{\calO_E}(\Gamma)}(\Sel_p^i(f/\QQ_\infty)^{\vee,\eta})
\]
where $\eta$ is any character on $\Delta$ when $i=1$ and it is the trivial character when $i=2$.
\end{theorem}
\begin{proof}
Assume $f$ is supersingular at $p$. By Corollary \ref{notzero}, assumption (A) implies that $\tilde{L}^\eta_{p,i}\ne0$. Hence, the cokernel of the first map in (\ref{poitoutate}) is $\Lambda_{\calO_E}(G_\infty)$-torsion. But $\Hh^2(T_{\bar{f}}(k-1))$ is also $\Lambda_{\calO_E}(G_\infty)$-torsion by \cite{kato04}, so $\Sel_p^i(f/\QQ_\infty)^\vee$ is $\Lambda_{\calO_E}(G_\infty)$-torsion, too. 

As in \cite[Theorem 7.3]{kobayashi03}, the first arrow of (\ref{poitoutate}) is now injective and there exist $n\ge0$ such that
\begin{equation}\label{equivalentMC}
0\rightarrow\Hh^1(T_{\bar{f}}(k-1))/\ZZ(T_{\bar{f}}(k-1))\rightarrow\image(\uCol_i)/(\varpi^{n_i}\tilde{L}_{p,i})\rightarrow\Sel_p^i(f/\QQ_\infty)^\vee\rightarrow\Hh^2(T_{\bar{f}}(k-1))\rightarrow0.
\end{equation}
Hence, the second part of the theorem follows from Theorem \ref{kato}(d) on taking $\eta$-components. The proof for the ordinary case is analogous.
\end{proof}

\begin{corollary}\label{Katoequivalence}
Let $\eta$ be a character on $\Delta$ as above. If assumptions (A) (or (A') depending on whether $f$ is supersingular or ordinary at $p$) and (E) hold, then Kato's main conjecture is equivalent to
\[
\Char_{\Lambda_{\calO_E}(\Gamma)}(\Sel_p^i(f/\QQ_\infty)^{\vee,\eta})=\Char_{\Lambda_{\calO_E}(\Gamma)}(\image(\uCol_i)^\eta/(\tilde{L}_{p,i}^\eta)).
\]
\end{corollary}
\begin{proof}
It follows immediately from (\ref{equivalentMC}).
\end{proof}

\begin{remark}
 We do \emph{not} assume that $n_1, n_2$ is an $\EA$-basis for $\NN(T_{\bar{f}})$. Hence $\image(\uCol_i)$ need not be contained in $\Lambda_{\calO_E}(G_\infty)$; but it is still clearly a $\Lambda_{\calO_E}(G_\infty)$-submodule of $\Lambda_{E}(G_\infty)$.
\end{remark}

By Theorem~\ref{image1}, if $f$ is supersingular at $p$, then assumptions (B), (C) and (D) imply that $\image(\uCol_1)=\Lambda_{E}(G_\infty)$. Therefore, we can reformulate Kato's main conjecture in the following form:

\begin{corollary}
If $f$ is supersingular at $p$ and assumptions (A)-(D) all hold, then Kato's main conjecture (after tensoring by $\QQ$) is equivalent to the assertion that $\Char_{\Lambda_{E}(\Gamma)}(\Sel_p^1(f/\QQ_\infty)^{\vee,\eta})$ is generated by $\tilde{L}_{p,1}^\eta$.
\end{corollary}

 
 \section*{Acknowledgements} We would like to thank John Coates and Tony Scholl for their interest and encouragement and Laurent Berger for sending us his proof of Theorem~\ref{Lambda} and for answering our questions about Wach modules. The ideas of the paper were developed while the third author was visiting the Newton Institute in Cambridge in autumn 2009. She would like to thank the organizers of the programme on `Non-abelian fundamental groups' for their hospitality. We are also grateful to the anonymous referee for his or her extremely careful reading of the paper and numerous helpful comments.

\providecommand{\bysame}{\leavevmode\hbox to3em{\hrulefill}\thinspace}

\end{document}